\documentclass{article}

\usepackage[letterpaper,margin=1.0in]{geometry}
\usepackage{authblk}
\usepackage[numbers,sort]{natbib}

\usepackage[utf8]{inputenc} 
\usepackage[T1]{fontenc}    
\usepackage{hyperref}       
\usepackage{url}            
\usepackage{booktabs}       
\usepackage{amsfonts}       
\usepackage{nicefrac}       
\usepackage{microtype}      
\usepackage{xcolor}         

\title{Accelerating Inexact HyperGradient Descent for Bilevel Optimization}

\author[1]{Haikuo Yang}
\author[1]{Luo Luo}
\author[2]{Chris Junchi Li}
\author[2,3]{Michael I.~Jordan}
\affil[1]{\normalsize School of Data Science, Fudan University}
\affil[2]{\normalsize Department of Electrical Engineering and Computer Sciences, University of California, Berkeley}
\affil[3]{\normalsize Department of Statistics, University of California, Berkeley}

\date{\today}

\usepackage{amsmath}\allowdisplaybreaks
\usepackage{amsfonts,bm}
\usepackage{amssymb}
\usepackage{amsthm}
\usepackage{algorithm}
\usepackage{algorithmic}
\usepackage{multirow}
\usepackage{booktabs}
\usepackage{xcolor}

\def\eqref#1{equation~(\ref{#1})}

\def\ceil#1{\left\lceil #1 \right\rceil}
\def\floor#1{\left\lfloor #1 \right\rfloor}
\def\1{\bf{1}}

\newcommand{\norm}[1]{\left\| #1 \right\|_2}
\def\inner#1#2{\langle #1, #2 \rangle}

\def\fD{{\mathcal{D}}}

\def\fH{{\mathcal{H}}}

\def\fK{{\mathcal{K}}}

\def\fO{{\mathcal{O}}}

\def\sB{{\mathbb{B}}}

\def\BR{{\mathbb{R}}}

\def\mA {{\bf A}}

\def\mH {{\bf H}}
\def\mI {{\bf I}}

\def\mU {{\bf U}}

\DeclareMathOperator*{\argmin}{arg\,min}

\makeatletter
\def\Ddots{\mathinner{\mkern1mu\raise\p@
\vbox{\kern7\p@\hbox{.}}\mkern2mu
\raise4\p@\hbox{.}\mkern2mu\raise7\p@\hbox{.}\mkern1mu}}
\makeatother

\def\pr#1{\left( #1 \right)}
\def\fpr#1{\left\{ #1 \right\}}

\makeatletter
\newcommand*{\rom}[1]{\expandafter\@slowromancap\romannumeral #1@}
\makeatother
\usepackage{microtype}
\usepackage{graphicx}
\usepackage{subfigure}
\usepackage{threeparttable}
\usepackage{booktabs} 

\usepackage{amsmath}
\usepackage{amssymb}
\usepackage{mathtools}
\usepackage{amsthm}

\usepackage[capitalize,noabbrev]{cleveref}

\theoremstyle{plain}
\newtheorem{theorem}{Theorem}[section]
\newtheorem{proposition}[theorem]{Proposition}
\newtheorem{lemma}[theorem]{Lemma}
\newtheorem{corollary}[theorem]{Corollary}
\newtheorem{condition}[theorem]{Condition}
\theoremstyle{definition}
\newtheorem{definition}[theorem]{Definition}
\newtheorem{assumption}[theorem]{Assumption}
\theoremstyle{remark}
\newtheorem{remark}[theorem]{Remark}

\usepackage[textsize=tiny]{todonotes}

\usepackage{enumitem}
\def\pb{}
\begin{document}

\pagenumbering{arabic}

\maketitle

\begin{abstract}
We present a method for solving general nonconvex-strongly-convex bilevel optimization problems. Our method---the \emph{Restarted Accelerated HyperGradient Descent} (\texttt{RAHGD}) method---finds an $\epsilon$-first-order stationary point of the objective with $\tilde{\mathcal{O}}(\kappa^{3.25}\epsilon^{-1.75})$ oracle complexity, where $\kappa$ is the condition number of the lower-level objective and $\epsilon$ is the desired accuracy. We also propose a perturbed variant of \texttt{RAHGD} for finding an $\big(\epsilon,\mathcal{O}(\kappa^{2.5}\sqrt{\epsilon}\,)\big)$-second-order stationary point within the same order of oracle complexity. Our results achieve the best-known theoretical guarantees for finding stationary points in bilevel optimization and also improve upon the existing upper complexity bound for finding second-order stationary points in nonconvex-strongly-concave minimax optimization problems, setting a new state-of-the-art benchmark. Empirical studies are conducted to validate the theoretical results in this paper.
\end{abstract}

\pb\section{Introduction}\label{intro}
Bilevel optimization is emerging as a key unifying problem formulation in machine learning, encompassing a variety of applications including meta-learning, model-free reinforcement learning and hyperparameter  optimization~\citep{franceschi2018bilevel,stadie2020learning}.
Our work focuses on a version of the general problem that is particularly
relevant to machine learning---the \emph{nonconvex-strongly-convex bilevel optimization problem}:
\begin{subequations}\label{bilevel_ab}
\begin{align}\label{bilevel_a}
&
\min_{x\in\BR^{d_x}}~
\Phi(x) \triangleq f(x, y^*(x))
,\\
& {\rm s.t}~~~~~
y^*(x) = \argmin_{y\in\BR^{d_y}}~g(x,y)
,
\label{bilevel_b}
\end{align}
\end{subequations}
where the upper-level function $f(x,y)$ is smooth and possibly nonconvex, and the lower-level function $g(x,y)$ is smooth and strongly convex with respect to $y$ for any given~$x$.%
\footnote{As the readers will see in Lemma~\ref{lem:gradient_Lip} additional smoothness conditions capture the smoothness of the overall objective function $\Phi(x)$.}
Bilevel optimization is more expressive but harder to solve than classical single-level optimization since the objective $\Phi(x)$ in~(\ref{bilevel_a}) involves the argument input $y^*(x)$ which is the solution of the lower-level problem~(\ref{bilevel_b}).
In contradistinction to classical optimization, bilevel optimization problem~(\ref{bilevel_ab}) involves solving an optimization problem where the minimization variable is taken as the minimizer of a lower-level optimization problem.

Most existing work on nonconvex-strongly-convex bilevel optimization~\citep{ghadimi2018approximation,ji2021bilevel,ji2022will} focuses on finding approximate \emph{first-order stationary points} (FOSP) of the objective.
Recently, \citet{huang2022efficiently} extended the scope of work in this area, proposing the  (perturbed) \emph{approximate implicit differentiation} (AID) algorithm which can find an $\big(\epsilon,\fO(\kappa^{2.5}\sqrt{\epsilon}\,)\big)$-\emph{second-order stationary points} (SOSP) within a $\tilde\fO(\kappa^4\epsilon^{-2})$ oracle complexity, where $\kappa\ge 1$ is the condition number of any $f(x,\cdot)$ and $\epsilon>0$ is the desired accuracy.
Given this result, a key further challenge is to study whether the $\epsilon^{-2}$-dependency in the upper complexity bound can be improved under additional Lipschitz assumptions on high-order derivatives~\citep{huang2022efficiently}.%
\footnote{It is worth noting that the $O(\epsilon^{-2})$ complexity is \emph{optimal} for finding an \emph{$\epsilon$-first-order stationary point} in terms of the dependency on $\epsilon$ under the Lipschitz gradient assumption~\citep{carmon2020lower} when $\kappa$ is treated as an $O(1)$-constant.
This is primarily due to that nonconvex optimization can be viewed as a special case of our bilevel problem, and hence the hard instance can be inherited to prove the analogous lower bound.}

Given this context, a natural question to ask is:
\emph{Can we design an algorithm that improves upon known algorithmic complexities for finding approximate first-order and second-order stationary points in nonconvex-strongly-convex bilevel optimization?}

\pb\subsection{Contributions}
We resolve this question by designing a particular form of acceleration of hypergradient descent and thereby improving the oracle complexity.
Our contributions are four-fold:

\begin{enumerate}[leftmargin=5mm,label=(\roman*)]
\item
We propose a method that we refer to as \emph{Restarted Accelerated Hypergradient Descent} (\texttt{RAHGD}) that applies Nesterov's \emph{accelerated gradient descent} (AGD) to approximate the solution $y^*(x)$ of the inner problem (\ref{bilevel_b}) and combines it with the \emph{conjugate gradient} (CG) method to construct an inexact hypergradient of the objective.
The algorithm makes use of proper restarting and acceleration to optimize the objective $\Phi(\cdot)$ based on the obtained inexact hypergradient.
We show that \texttt{RAHGD} can find an $\epsilon$-FOSP of the objective within $\fO(\kappa^{3.25}\epsilon^{-1.75})$ first-order oracle queries [Section~\ref{sec_convergence}].

\item
For the task of finding approximate second-order stationary points, we add a perturbation step to \texttt{RAHGD} and introduce the \emph{Perturbed Restarted Accelerated HyperGradient Descent} (\texttt{PRAHGD}) algorithm.
We show that \texttt{PRAHGD} can efficiently escape saddle points and find an~$\big(\epsilon,\fO(\kappa^{2.5}\sqrt{\epsilon}\,)\big)$-second-order stationary point of the objective $\Phi$ within $\tilde\fO\big(\kappa^{3.25}\epsilon^{-1.75}\big)$ oracle queries.
This improves over the best known complexity in bilevel optimization due to~\citet{huang2022efficiently} by a factor of~$\tilde\fO(\kappa^{0.75}\epsilon^{-0.25})$ [Section~\ref{sec_perturb}].

\item
We apply the theoretical framework of \texttt{PRAHGD} to the problem of minimax optimization.
Specially, we propose a \texttt{PRAHGD} variant crafted for nonconvex-strongly-concave minimax optimization. We refer to the resulting algorithm as \emph{Perturbed Restarted Accelerate Gradient Descent Ascent} (\texttt{PRAGDA}). We show that \texttt{PRAGDA} provably finds an $
\fO\big(\epsilon,\fO(\kappa^{1.5}\sqrt{\epsilon}\,)\big)
$-SOSP with a first-order oracle query complexity of $
\tilde\fO(\kappa^{1.75}\epsilon^{-1.75})
$.
This improves upon the best known first-order (including gradient/Hessian-vector/Jacobian-vector-product) oracle query complexity bound of $
\tilde\fO(\kappa^{1.5}\epsilon^{-2}+\kappa^{2}\epsilon^{-1.5})
$ due to \citet{luo2022finding} [Section~\ref{sec_indications}].

\item
We conduct a variety of empirical studies of bilevel optimization.
Specifically, we evaluate the effectiveness of our proposed algorithms (\texttt{RAHGD} / \texttt{PRAHGD} / \texttt{PRAGDA}) by applying them to three different tasks: a synthetic minimax problem, data hypercleaning for the MNIST dataset, and hyperparameter optimization for logistic regression.
Our studies demonstrate that our algorithms outperform several established baseline algorithms, such as BA, AID-BiO, ITD-BiO, PAID-BiO and iMCN, with inevitably faster empirical convergence.
The results provide empirical evidence in support of the effectiveness of our proposed algorithmic framework for bilevel and minimax optimization [Appendix~\ref{sec:experiment}].
\end{enumerate}

\pb\subsection{Overview of Our Algorithm Design and Main Techniques}\label{sec:overview}
We overview the algorithm design in this subsection.
Inspired by the success of the accelerated gradient descent method for nonconvex optimization~\citep[see, e.g.,][]{jin2018accelerated,li2022restarted}, we propose a novel method called the~\emph{restarted accelerated hypergradient descent} (\texttt{RAHGD}) algorithm.
The gradient of $\Phi(x)$, which we called the \emph{hypergradient}, can be computed via the following equation~\citep{ghadimi2018approximation,ji2021bilevel}:
\begin{equation}\label{equ:hypergradient}
\nabla\Phi(x)
=
\nabla_x f(x, y^*(x))
-\nabla^2_{xy} g(x, y^*(x)) \big(\nabla^2_{yy} g(x, y^*(x))\big)^{-1} \nabla_y f(x, y^*(x))
.
\end{equation}
Unfortunately, directly applying first-order algorithms by iterating with the exact hypergradient $\nabla\Phi(x)$ is costly or intractable for large-scale problems, given the need to obtain $y^*(x)$ and particularly given the need to invert the matrix $\nabla^2_{yy} g(x, y^*(x))$.

For given $x=x_k\in\BR^{d_x}$, we aim to construct an estimate of~$\nabla \Phi(x_k)$ with reasonable computational cost and sufficient accuracy.
The strong convexity of $g(x_k,\cdot)$ motivates us to apply AGD for finding $y_k\approx y^*(x_k)$ and hence used as a replacement of arguments estimating $\nabla\Phi(x_k)$.
To avoid direct computation of the term $\big(\nabla^2_{yy} g(x_{k}, y_{k})\big)^{-1} \nabla_y f(x_{k}, y_{k})$, we observe that it is the solution of the following quadratic problem:
\begin{align}\label{prob:CG}
\min_{v\in\BR^{d_y}}~
\frac{1}{2}v^\top\nabla^2_{yy} g(x_{k}, y_{k})v
-
v^\top \nabla_y f(x_{k}, y_{k})
\end{align}
Accordingly, we can fastly estimate $v_k\approx\big(\nabla^2_{yy} g(x_{k}, y_{k})\big)^{-1} \nabla_y f(x_{k}, y_{k})$ and solve (\ref{prob:CG}) using a conjugate gradient subroutine.
Based on $y_k$ and $v_k$, we obtain the expression for an \emph{inexact hypergradient}:
\begin{align}\label{equ:inexact hypergradient}
\hat{\nabla} \Phi(x_k)
=
\nabla_x f(x_{k}, y_{k}) - \nabla^2_{xy} g(x_{k}, y_{k}) v_{k}
\end{align}
which can serve as a surrogate of the true hypergradient $\nabla \Phi(x_k)$ in our first-order algorithmic design.

We formally present \texttt{RAHGD} in Algorithm~\ref{alg:AHGD}. 
The main issue to address for \texttt{RAHGD} is the computational cost for achieving sufficient accuracy of $\hat{\nabla} \Phi(x_k)$.
Interestingly, our theoretical analysis shows that all of the additional cost arises from the computations of $y_k$ and $v_k$, and it can thus be bounded sharply. 
As a result, our algorithm can find approximate first-order stationary points with reduced oracle complexities than existing methods \citep{ghadimi2018approximation,ji2021bilevel}.
We also introduce the \emph{perturbed} \texttt{RAHGD} (\texttt{PRAHGD}) 
in Algorithm~\ref{alg:AHGD} for escaping saddle points.
Extending the analysis of \texttt{RAHGD}, we show that \texttt{PRAHGD} can find approximate second-order stationary points more efficiently than existing arts~\citep{huang2022efficiently}.

\pb\subsection{Related Work}
The subject of bilevel optimization problem has a long history with early work tracing back to the 1970s \citep{bracken1973mathematical}.
Recent algorithmic advances in this field have driven successful applications in areas such as meta-learning \citep{bertinetto2018meta,franceschi2018bilevel,ji2020convergence}, reinforcement learning \citep{konda1999actor,stadie2020learning,hong2020two} and hyperparameter optimization \citep{feurer2019hyperparameter,shaban2019truncated,grazzi2020iteration}.

There have also been theoretical advances in bilevel optimization in recent years.
\citet{ghadimi2018approximation} presented a convergence rate for the AID approach when $f(x,y)$ is convex, analyzing the complexity of an accelerated algorithm that uses gradient descent to approximate $y^*(x_k)$ in the inner loop and uses AGD in the outer loop.
Further improvements in dependence on the condition number and analysis of the convergence were achieved via the \emph{iterative differentiation} (ITD) approach by~\citet{ji2021bilevel,ji2022will}, who analyzed the complexity of AID
and ITD and also provided a complexicity analysis for a randomized version.
\citet{hong2020two} proposed the TTSA algorithm---a provable single-loop algorithm that updates two variables in an alternating manner%
---and presented applications to the problem of reinforcement learning under randomized scenarios.
For stochastic bilevel problems, various methods have been proposed, such as BSA by~\citet{ghadimi2018approximation}, TTSA by~\citet{hong2020two}, stocBiO by~\citet{ji2021bilevel}, and ALSET by~\citet{chen2021closing}.
More recent research on this front has focused on variance reduction and momentum techniques, resulting in cutting-edge stochastic first-order oracle complexities.

While much of the literature on bilevel optimization has focused on finding first-order stationary points, the problem of finding second-order stationary points has been largely unadressed.
\citet{huang2022efficiently} recently proposed a perturbed algorithm for finding approximate second-order stationary points.
The algorithm adopts gradient descent (GD) to approximately solve the lower-level minimization problem and conjugate gradient (CG) to solve for Hessian-vector product with GD used in the outer loop.
For the problem of classical optimization, second-order methods such as those proposed in~\citet{nesterov2006cubic,curtis2017trust} have been used to obtain $\epsilon$-accurate SOSPs in single-level optimization with a complexity of $\fO(\epsilon^{-1.5})$; however, they require expensive operations such as inverting Hessian matrices.
A significant body of recent literature has been focusing on first-order methods for obtaining an approximate $\big(\epsilon,\mathcal{O}(\kappa^{2.5}\sqrt{\epsilon}\,)\big)$-SOSP, with the best-known query complexity of $\tilde{\fO}(\epsilon^{-1.75})$ of gradient and Hessian-vector products~\citep{agarwal2017finding,carmon2018accelerated,carmon2017convex,jin2017escape,jin2018accelerated,li2022restarted}.

An important special case of the bilevel optimization problem~(\ref{bilevel_ab})---the problem of minimax optimization, where $g = -f$ in Eq.~(\ref{bilevel_b})---has been extensively studied in the literature (with $g = -f$ set in Eq.~(\ref{bilevel_b})).
Minimax optimization has been the focus of attention in the machine learning community recently due to its applications to training GANs~\citep{goodfellow2020generative,arjovsky2017wasserstein}, to adversarial learning~\citep{goodfellow2014explaining,sinha2017certifying} and to optimal transport~\citep{lin2020projection,huang2021riemannian}.
On the theoretically font,~\citet{nouiehed2019solving,jin2020local} studied the complexity of Multistep Gradient Descent Ascent (GDmax), and
\citet{lin2020gradient,lu2020hybrid} provided the first convergence analysis for the single-loop \emph{gradient descent ascent} (GDA) algorithm.
More recently,~\citet{luo2020stochastic} applied the stochastic variance reduction technique to the nonconvex-strongly-concave case, achieving the first optimal complexity upper bound when $\kappa$ is treated as an $O(1)$-constant.
\citet{zhang2020single} proposed a stabilized smoothed GDA algorithm that achieves a better complexity for the nonconvex-concave problem.
\citet{fiez2021global} provided asymptotic results showing that GDA converges to a local minimax point almost surely.
Nevertheless, to the best of our knowledge, all the previous works targeted finding approximate stationary points of $\Phi(x)$, and the theory for finding the local minimax points is absent in the literature.
It was not until very recently that~\citet{luo2022finding,chen2021escaping} independently proposed (inexact) cubic-regularized Newton methods for solving this problem; these are second-order algorithms that provably converge to a local minimax point.
These algorithms are limited, however, to minimax optimization and they cannot be used to solve the more general bilevel optimization problems.

\paragraph{Organization.}
The rest of this work is organized as follows. 
Section~\ref{sec_preliminaries} delineates the assumptions and specific algorithmic subroutines.
Section~\ref{sec_convergence} formally presents the \texttt{RAHGD} algorithm along with its complexity bound for finding approximation first-order stationary points.
Section~\ref{sec_perturb} proposes the \texttt{PRAHGD}, the perturbed version of \texttt{RAHGD}, along with its complexity bound for finding approximate second-order stationary points.
Section~\ref{sec_indications} presents the application for minimax optimization.
Section~\ref{sec_conclude} concludes the paper and discusses future directions.
Presentations of technical analysis and empirical studies are deferred to the supplementary materials.

\paragraph{Notation.}
We let $\norm{\cdot}$ be the spectral norm of matrices and the Euclidean norm of vectors. 
Given a real symmetric matrix $A$, we let $\lambda_{\max}(A)$ ($\lambda_{\min}(A)$) denote its largest (smallest) eigenvalue.
We use the notation $\sB(r)$ to present the closed Euclidean ball with radius $r$ centered at the origin.
We denote $Gc(f,\epsilon)$, $JV(f,\epsilon)$ and $HV(f,\epsilon)$ as the oracle complexities of gradients, Jacobian-vector products and Hessian-vector products, respectively.
Finally, we adopt the notation $\mathcal{O}(\cdot)$ to hide only absolute constants which do not depend on any problem parameters, and also~$\tilde{\mathcal{O}}(\cdot)$ for constants that include a polylogarithmic factor.

\begin{table*}[t]
\begin{center}
\caption{Comparison of complexities for nonconvex bilevel optimization algorithms of finding approximate FOSPs.}\vskip0.15cm
\label{table:comparision_fir}
\begin{tabular}{ccccc}
\toprule
Algorithm
&
$Gc(f,\epsilon)$
&
$Gc(g,\epsilon)$
&
$JV(g,\epsilon)$
&
$HV(g,\epsilon)$ 
\\
\midrule
BA~\citep{ghadimi2018approximation}
&
$\fO(\kappa^4\epsilon^{-2})$
&
$\fO(\kappa^5\epsilon^{-2.5})$
&
$\fO(\kappa^4\epsilon^{-2})$
&
$\tilde\fO(\kappa^{4.5}\epsilon^{-2})$
\\
AID-BiO~\citep{ji2021bilevel}
&
$\fO(\kappa^3\epsilon^{-2})$
&
$\fO(\kappa^4\epsilon^{-2})$
&
$\fO(\kappa^3\epsilon^{-2})$
&
$\fO(\kappa^{3.5}\epsilon^{-2})$
\\
ITD-BiO~\citep{ji2021bilevel}
&
$\fO(\kappa^3\epsilon^{-2})$
&
$\tilde\fO(\kappa^4\epsilon^{-2})$
&
$\tilde\fO(\kappa^4\epsilon^{-2})$
&
$\tilde\fO(\kappa^{4}\epsilon^{-2})$
\\
\texttt{RAHGD}~(this work)
&
$\tilde\fO(\kappa^{2.75}\epsilon^{-1.75})$
&
$\tilde\fO(\kappa^{3.25}\epsilon^{-1.75})$
&
$\tilde\fO(\kappa^{2.75}\epsilon^{-1.75})$
&
$\tilde\fO(\kappa^{3.25}\epsilon^{-1.75})$
\\
\bottomrule
\end{tabular}
\end{center}
\end{table*}

\begin{table*}[t]
\begin{center}
\caption{Comparison of complexities for nonconvex bilevel optimization algorithms of finding approximate SOSPs.}
\label{table:comparision_sec}
\vskip0.15cm
\begin{tabular}{ccccc}
\toprule
Algorithm
&
$Gc(f,\epsilon)$
&
$Gc(g,\epsilon)$
&
$JV(g,\epsilon)$
&
$HV(g,\epsilon)$
\\
\midrule
Perturbed AID~\citep{huang2022efficiently}
&
$\tilde\fO(\kappa^3\epsilon^{-2})$
&
$\tilde\fO(\kappa^4\epsilon^{-2})$
&
$\tilde\fO(\kappa^3\epsilon^{-2})$
&
$\tilde\fO(\kappa^{3.5}\epsilon^{-2})$
\\
\texttt{PRAHGD}~(this work)
&
$\tilde\fO(\kappa^{2.75}\epsilon^{-1.75})$
&
$\tilde\fO(\kappa^{3.25}\epsilon^{-1.75})$
&
$\tilde\fO(\kappa^{2.75}\epsilon^{-1.75})$
&
$\tilde\fO(\kappa^{3.25}\epsilon^{-1.75})$
\\
\bottomrule
\end{tabular}
\end{center}
\begin{tablenotes}\footnotesize
\item 
---
Notation $\tilde\fO$ omits a polylogarithmic factor in relevant parameters. $\kappa$: condition number of the lower-level objective.
\item
---
$Gc(f,\epsilon)$ and $Gc(g,\epsilon)$: number of gradient evaluations w.r.t.~$f$ and $g$.
\item
---
$JV(g,\epsilon)$: number of Jacobian-vector products $\nabla_{xy}^2 g(x, y)v$.
\item
---
$HV(g,\epsilon)$: number of Hessian-vector products $\nabla_{yy}^2 g(x, y)v$.
\end{tablenotes}
\end{table*}

\pb\section{Preliminaries}\label{sec_preliminaries}
In this section, we first proceed to establish convergence of the algorithmic subroutines related to our algorithm---\emph{accelerated gradient descent} and the \emph{conjugate gradient method}.
Then, we present the notations and assumptions necessary for our problem setting.
We proceed to establish convergence of these two algorithmic subroutines in the following paragraphs.

\begin{algorithm}[!b]
\caption{${\rm AGD} (h, z_0, T, \alpha, \beta)$}
\label{alg:AGD}
\begin{algorithmic}[1]
\STATE \textbf{Input:}
objective $h(\cdot)$;
initialization $z_0$;
iteration number $T\ge 1$;
step-size $\alpha>0$;
momentum param.~$\beta\in (0,1)$
\\[0.05cm]
\STATE $\tilde z_0\leftarrow z_0$ \\[0.05cm]
\STATE\textbf{for} $t=0,\dots,T-1$ \textbf{do} \\[0.05cm]
\STATE\quad $z_{t+1}\leftarrow{\tilde z}_t-\alpha\nabla h(\tilde z_t)$ \\[0.1cm]
\STATE\quad ${\tilde z}_{t+1} \leftarrow z_{t+1}+\beta(z_{t+1}-z_t)$ \\[0.1cm]
\STATE\textbf{end for} \\[0.05cm]
\STATE\textbf{Output:} $z_T$
\end{algorithmic}
\end{algorithm}

\paragraph{Subroutine 1: Accelerated Gradient Descent.}
Our first component is \emph{Nesterov's accelerated gradient descent} (AGD), which is an acceleration of the first-order method in smooth convex optimization.
We describe the details of AGD for minimizing a given smooth and strongly convex function in Algorithm~\ref{alg:AGD}, which exhibits the following \emph{optimal} convergence rate~\citep{nesterov2013introductory}:

\begin{lemma}[\citep{nesterov2013introductory}]\label{lem:AGD}
Running Algorithm~\ref{alg:AGD} on an $\ell_h$-smooth and $\mu_h$-strongly convex objective function $h(\cdot)$ with $\alpha = 1/\ell_h$ and $\beta = (\sqrt{\kappa_h}-1)/(\sqrt{\kappa_h}+1)$ produces an output $z_T$ satisfying 
$$
\|z_T - z^*\|_2^2
\le
(1+\kappa_h)\left(1-\frac{1}{\sqrt{\kappa_h}}\right)^T\|z_0 - z^*\|_2^2
,
$$
where $z^* = \argmin_z h(z)$ and $\kappa_h = \ell_h/\mu_h$ denotes the condition number of the objective $h$.
\end{lemma}

\begin{algorithm}[!tb]
\caption{${\rm CG} (A, b, T, q_0)$}
\label{alg:CG}
\begin{algorithmic}[1]
\STATE \textbf{Input:}
quadratic objective (as in Eq.~(\ref{equ:CG}));
initialization $q_0$;
iteration number $T\ge 1$
\STATE
$r_0\leftarrow Aq_0-b$
,
$p_0\leftarrow-r_0$
\\[0.05cm]
\STATE\textbf{for} $t= 0,\dots,T-1$ \textbf{do}
\\[0.05cm]
\STATE\quad $\alpha_t \leftarrow \dfrac{r_t^Tr_t}{p_t^T A p_t}$
\\[0.1cm]
\STATE\quad $q_{t+1} \leftarrow q_t + \alpha_t p_t $ \\[0.1cm]
\STATE\quad $r_{t+1} \leftarrow r_t + \alpha_t A p_t $ \\[0.1cm]
\STATE\quad $\beta_{t+1} \leftarrow \dfrac{r_{t+1}^Tr_{t+1}}{r_t^Tr_t} $ \\[0.1cm]
\STATE\quad $p_{t+1} \leftarrow -r_{t+1} + \beta_{t+1} p_t $ \\[0.1cm]
\STATE\textbf{end for} \\[0.05cm]
\STATE\textbf{Output:} $q_T$
\end{algorithmic}
\end{algorithm}

\paragraph{Subroutine 2: Conjugate Gradient Method.}
The \emph{(linear) conjugate gradient} (CG) method was proposed by Hestenes and Stiefel in the 1950s as an iterative method for solving linear systems with positive definite coefficient matrices.
It serves as an alternative to Gaussian elimination that is well-suited for solving large problems.
CG can be formulated as the minimization of the quadratic objective function
\begin{equation}\label{equ:CG}
\frac{1}{2}q^\top Aq - q^\top b
,
\end{equation}
where $A\in\BR^{d\times d}$ is a positive definite matrix and $b\in\BR^d$ is a fixed vector.
We summarize the setup of CG for minimizing function~(\ref{equ:CG}) in Algorithm~\ref{alg:CG}, and record the following convergence property~\citep{nocedal2006numerical}:

\vskip .1in

\begin{lemma}[\citep{nocedal2006numerical}]\label{lem:CG}
Running Algorithm~\ref{alg:CG} for minimizing quadratic function~(\ref{equ:CG}) produces $q_T$ satisfying
$$
\|q_T-q^*\|_2
\le
2\sqrt{\kappa_A}\left(\dfrac{\sqrt{\kappa_A}-1}{\sqrt{\kappa_A}+1} \right)^T\|q_0 - q^*\|_2
, 
$$
where $q^* = A^{-1} b$ denotes the unique minimizer of Eq.~(\ref{equ:CG}), and $\kappa_A = \lambda_{\max}(A) / \lambda_{\min}(A)$ denotes the condition number of (positive definite) matrix $A$.%
\footnote{%
When minimizing the quadratic objective~\eqref{equ:CG}, CG and AGD enjoy comparable convergence speeds.
In fact in squared Euclidean metric, Lemma~\ref{lem:CG} implies a convergence rate of $
\|q_T-q^*\|_2^2
\le
4\kappa_A\exp\left(-\frac{4}{\sqrt{\kappa_A}+1}\cdot T\right)\|q_0 - q^*\|_2^2
$ and hence running CG instead of AGD for~\eqref{equ:CG} improves the coefficient in the exponent by an asymptotic factor of four while maintaining the $\fO(\kappa_A)$-prefactor up to a numerical constant.
Our two algorithmic subroutines AGD and CG are logically connected, and our adoption of CG whenever possible is partly due to its additional advantage of requiring fewer input parameters (corresponding to $\alpha$, $\beta$ in Algorithm~\ref{alg:AGD}).
}
\end{lemma}

In the rest of this section we impose the following assumptions on the upper-level function $f$ and the lower-level function $g$.
We then turn to the details of our theoretical analysis:

\begin{assumption}\label{asm:1}
The upper-level function $f(x,y)$ and lower-level function $g(x,y)$ satisfy the following conditions:
\begin{enumerate}[leftmargin=5mm,label=(\roman*)]
\item
Function $g(x,y)$ is three times differentiable and $\mu$-strongly convex with respect to $y$ for any fixed $x$;
\item
Function $f(x,y)$ is twice differentiable and $M$-Lipschitz continuous with respect to $x$ and $y$;
\item
Gradient $\nabla f(x,y)$ and $\nabla g(x,y)$ are $\ell$-Lipschitz continuous with respect to $x$ and $y$;
\item
Jacobian $\nabla_{xy}^2 f(x,y)$, $\nabla_{xy}^2 g(x,y)$ and Hessians $\nabla_{xx}^2f(x,y)$,  $\nabla_{yy}^2f(x,y)$, $\nabla_{yy}^2 g(x,y)$ are $\rho$-Lipschitz continuous with respect to $x$ and $y$;
\item
Third-order derivatives $\nabla_{xyx}^3g(x,y),\nabla_{yxy}^3g(x,y)$ and $\nabla_{yyy}^3g(x,y)$ are $\nu$-Lipschitz continuous with respect to $x$ and $y$.
\end{enumerate}
\end{assumption}

These assumptions are standard for the bilevel optimization problem we are studying.
We also introduce an appropriate notion of condition number for the lower-level function $g(x,y)$.

\begin{definition}\label{dfn:1}
Under Assumption~\ref{asm:1}, we refer to $\kappa \triangleq \ell/\mu$ the \emph{condition number} of the lower-level objective $g(x,y)$.
\end{definition}

Leveraging such a notion, we can show that the solution to the lower-level optimization problem $
y^*(x) = \argmin_{y\in\BR^{d_y}}~g(x,y)
$ is $\kappa$-Lipschitz continuous in $x$ under Assumption~\ref{asm:1}, as indicated in the following lemma:

\begin{lemma}
\label{lem:y*-Lip}
Suppose Assumption~\ref{asm:1} holds, then $y^*(x)$ is $\kappa$-Lipschitz continuous, that is, we have $
\norm{y^*(x) - y^*(x')}
\le
\kappa \norm{x - x'}
$ for any $x, x' \in \BR^{d_x}$.
\end{lemma}

We also can show that $\Phi(x)$ admits Lipschitz continuous gradients and Lipschitz continuous Hessians, as shown in the following lemmas:
\begin{lemma}\label{lem:gradient_Lip}
Suppose Assumption~\ref{asm:1} holds, then $\Phi(x)$ is $\tilde{L}$-gradient Lipschitz continuous, that is, we have $
\|\nabla \Phi(x) - \nabla \Phi(x')\|
\le
\tilde{L}\|x - x'\|
$ for any $x,x'\in \BR^{d_x}$, where $\tilde{L}=\fO(\kappa^3)$.
\end{lemma}

\begin{lemma}\label{lem:Hessian_Lip}
Suppose Assumption~\ref{asm:1} holds, then $\Phi(x)$ is $\tilde{\rho}$-Hessian Lipschitz continuous, that is, $
\|\nabla^2\Phi(x) - \nabla^2\Phi(x')\|
\le
\tilde{\rho}\|x - x'\|
$ for any $x,x'\in \BR^{d_x}$, where $\tilde{\rho}=\fO(\kappa^5)$.
\end{lemma}

The detailed form of $\tilde L $ and $\tilde \rho$ can be found in Appendix~\ref{sec:basic lemmas}.
We give the formal definition of an $\epsilon$-first-order stationary point as well as an $(\epsilon,\tau)$-second-order stationary point, as follows:

\begin{definition}[Approximate First-Order Stationary Point]\label{dfn:first order}
Under Assumption~\ref{asm:1}, we call $x$ an \emph{$\epsilon$-first-order stationary point} of $\Phi(x)$ if $\norm{\nabla\Phi(x)} \le \epsilon$.
\end{definition}

\begin{definition}[Approximate Second-Order Stationary Point]\label{dfn:second order}
Under Assumption~\ref{asm:1}, we call $x$ an \emph{$(\epsilon,\tau)$-second-order stationary point} of $\Phi(x)$ if $\norm{\nabla\Phi(x)}\le \epsilon
$ and
$\lambda_{\min}(\nabla^2\Phi(x))\ge -\tau$.
\end{definition}
We remark that these concepts are commonly used in the nonconvex optimization literature~\citep{nesterov2006cubic}.
The approximate second-order stationary point is sometimes referred to as an “approximate local minimizer.”

With all these preliminaries at hand, we are ready to proceed with the (perturbed) restarted accelerated hypergradient descent method.

\begin{algorithm}[t]
\caption{(Perturbed) Restarted Accelerated HyperGradient Descent, (\texttt{P})\texttt{RAHGD}} \label{alg:AHGD}
\begin{algorithmic}[1]
\STATE \textbf{Input:}
initial vector $x_{0,0}$;
step-size $\eta > 0$;
momentum parameter $\theta \in (0,1)$;
parameters $\alpha > 0,\beta \in (0,1)$;
parameter $\{T_{t,k}\}$ of AGD;
parameter $\{T_{t,k}'\}$ of CG;
iteration threshold $K\ge 1$;
parameter $B$ for triggering restarting;
perturbation radius $r > 0$;
option $\mbox{\texttt{Perturbation}}\in\{0,1\}$
\\[0.1cm]
\STATE
$
k\leftarrow 0,\ 
t\leftarrow 0,\ 
x_{0,-1}\leftarrow x_{0,0}
,\
y_{0,-1}\leftarrow{\rm AGD}\left(g(x_{1,-1},\,\cdot\,), 0, T_{0,-1}, \alpha, \beta\right)
,\ 
v_{0,-1} \leftarrow y_{0,-1}$\\[0.1cm]
\STATE \textbf{while} $k<K$ \\[0.1cm]
\STATE\quad $w_{t,k}\leftarrow x_{t,k}+(1-\theta)(x_{t,k}-x_{t,k-1})$ \\[0.1cm]
\STATE\quad $y_{t,k} \leftarrow {\rm AGD} (g(w_{t,k},\,\cdot\,), y_{t,k-1}, T_{t,k}, \alpha, \beta)$ \\[0.1cm]\label{line:agd}
\STATE\quad $v_{t,k} \leftarrow {\rm CG}(\nabla^2_{yy}g(w_{t,k},y_{t,k}), \nabla_y f(w_{t,k},y_{t,k})$, $T'_{t,k},{v_{t,k-1}})$ \\[0.1cm]
\STATE\quad $u_{t,k} \leftarrow \nabla_x f(w_{t,k}, y_{t,k}) - \nabla^2_{xy} g(w_{t,k}, y_{t,k}) v_{t,k}$ \\[0.1cm]
\STATE\quad $x_{t,k+1}\leftarrow w_{t,k}-\eta u_{t,k}$ \\[0.1cm]
\STATE\quad $k\leftarrow k+1$ \\[0.1cm]
\STATE\quad \textbf{if} $k\sum_{i=0}^{k-1}\|x_{t,i+1}-x_{t,i}\|^2> B^2$ \\[0.1cm]
\STATE\qquad $v_{t+1,-1} \leftarrow v_{t,k}$\\[0.1cm]
\STATE\quad\quad 
$x_{t+1,0}\leftarrow \left\{
\begin{array}{ll}
   x_{t,k},
   & \text{if~~ \texttt{Perturbation}$\,=0$}
   \\
   x_{t,k} + \xi ~~\text{with}~~ \xi \sim \text{Unif}(\sB(r)),
   & \text{if~~ \texttt{Perturbation}$\,=1$}
\end{array}
\right.$
\\[0.1cm]\label{line:restart}
\STATE\quad\quad 
$x_{t+1,-1}\leftarrow x_{t+1,0}$
\\[0.1cm]
\STATE \qquad $k\leftarrow 0$,  $t \leftarrow t+1$
\\[0.1cm]
\STATE \qquad $y_{t,-1}\leftarrow{\rm AGD} (g(x_{t,-1},\,\cdot\,), 0, T_{t,-1}, \alpha, \beta)$
\\[0.1cm]
\STATE\quad \textbf{end if}
\\[0.1cm]
\STATE\textbf{end while}
\\[0.1cm]
\STATE
$K_0\leftarrow\argmin_{\lfloor \frac{K}{2}\rfloor\leq k\leq K-1}
\norm{x_{t,k+1}-x_{t,k}}$
\\[0.1cm]
\STATE \textbf{Output:}
$\hat w\leftarrow\frac{1}{K_0+1}\sum_{k=0}^{K_0}w_{t,k}$
\end{algorithmic}
\end{algorithm}

\pb\section{Restarted Accelerated HyperGradient Descent Algorithm}\label{sec_convergence}
In this section, we present our \emph{restarted accelerated hypergradient descent} (\texttt{RAHGD}) algorithm and provide corresponding query complexity upper bound results.
We present the details of \texttt{RAHGD} in Algorithm~\ref{alg:AHGD}, which has a nested loop structure.
The outer loop, indexed by $k$, uses the accelerated gradient descent method to find the solver of~(\ref{bilevel_a}).
The AGD step in Line~\ref{line:agd} is used to find the inexact solver of~(\ref{bilevel_b}).
The CG step is added to compute the Hessian-vector product, as shown in~(\ref{equ:hypergradient}).
We note that the iteration numbers of the AGD and CG steps play an important role in the convergence analysis of Algorithm~\ref{alg:AHGD}; moreover, 
at the end of this section we will show that the total iteration number of AGD and CG can be bounded sharply.
Finally, note that there is a restarting step in Line~\ref{line:restart} where the option~\texttt{Perturbation} is taken as $\,=0$.

We let subscript $t$ index the times of restarting.  We note that the subscript $t$ of epoch number is added in Algorithm~\ref{alg:AHGD} purely for the sake of an easier convergence analysis.
The incurred storage of iterations across all epochs can be avoided when implementing Algorithm~\ref{alg:AHGD} in practice.

In accelerated nonconvex optimization, a straightforward application of AGD cannot ensure consistent decrements of the objective function. Inspired by the work of~\citet{li2022restarted}, we add a restarting step in Line~\ref{line:restart}---we define $\mathcal{K}$ to be the iteration number when the “if condition” triggers, and hence the iterates from $k=0$ to $k = \fK$ constructs one single epoch, where $
\fK
=
\min_k\left\{
k\ge 1:\, k\sum_{t=0}^{k-1}\|x_{t+1}-x_t\|_2^2 > B^2
\right\}
$.
Then we can have the objective function consistently decrease with respect to each epoch when we run Algorithm~\ref{alg:AHGD}.
We provide the convergence results for \texttt{RAHGD} in the rest of this section. 

Denote $
v_k^*
=
\big(\nabla^2_{yy} g(w_{k}, y_{k})\big)^{-1} \nabla_y f(w_{k}, y_{k})
$.
\vspace{0.1cm}
Due to the bilevel optimization problem we are 
considering the following conditions on the inexact gradient.
Recall that the overall objective function $\Phi(x)$ is $\tilde L$-gradient Lipschitz continuous, and both the upper-level function $f(x,y)$ and the lower-level function $g(x,y)$ are $\ell$-gradient Lipschitz continuous:

\begin{condition}\label{con:4.1}
Let $w_{-1} = x_{-1}$.
Then for some $\sigma > 0$, we assume that the estimators $y_k \in \BR^{d_y}$ and $v_k \in \BR^{d_y}$ satisfy the conditions
\begin{equation}
\label{equ:cond1}
\|y_k - y^*(w_k)\|_2\le \frac{\sigma}{2\tilde{L}}
,\qquad
\text{for each $k = -1,0,1,2,\dots$}
\end{equation}
and
\begin{equation}
\label{equ:cond2}
\|v_k - v_k^*\|\le \frac{\sigma}{2\ell}
,\qquad
\text{for each $k = 0,1,2,\dots$}
\end{equation}
\end{condition}

\begin{remark}\label{remark1}
We will show at the end of this section that Condition~\ref{con:4.1} is guaranteed to hold after running AGD and CG for a sufficient number of iterates.
\end{remark}

Under Condition~\ref{con:4.1}, the bias of $\hat{\nabla}\Phi(x_k)$ defined in \eqref{equ:inexact hypergradient} can be bounded as shown in the following lemma:

\begin{lemma}[Inexact gradients]\label{lem:bias of gradient}
Suppose Assumption~\ref{asm:1} and Condition~\ref{con:4.1} hold, then we have $
\| \nabla\Phi(w_k)- \hat{\nabla}\Phi(w_k) \|_2 \le \sigma
$.
\end{lemma}

In the following theorem we show that the iteration complexity in the outer loop is bounded.

\begin{theorem}[\texttt{RAHGD} finding $\epsilon$-FOSP]\label{thm:one-oder}
Suppose that Assumptions~\ref{asm:1} and Condition~\ref{con:4.1} hold.
Let
\begin{align*}
\eta = \frac{1}{4\tilde L}
,\quad
B = \sqrt{\frac{\epsilon}{\tilde\rho}}
,\quad
\theta = 4(\tilde\rho\epsilon\eta^2)^{1/4}
,\quad
K = \frac{1}{\theta}
,\quad
\alpha = \frac{1}{\ell}
,\quad
\beta=\frac{\sqrt{\kappa} -1}{\sqrt{\kappa}+1}
,\quad
\sigma = \epsilon^2
.
\end{align*}
and assume that $\epsilon\le\frac{\tilde L^2}{\tilde \rho}$.
Denote $\Delta = \Phi(x_{\rm{int}}) - \min_x~\Phi(x)$.
Then \emph{\texttt{\texttt{RAHGD}}} in Algorithm~\ref{alg:AHGD} terminates within $
\fO(\Delta\tilde L^{0.5}\tilde\rho^{0.25}\epsilon^{-1.75})
$ iterates, outputting $\hat{w}$ satisfying $
\norm{\nabla \Phi(\hat{w})}\le 83\epsilon
$.
\end{theorem}

Theorem~\ref{thm:one-oder} says that Algorithm~\ref{alg:AHGD} can find an $\epsilon$-first-order stationary point with $\fO(\kappa^{2.75}\epsilon^{-1.75})$ iterations in the outer loop.
The following result indicates that Condition~\ref{con:4.1} holds if we run AGD and CG for a sufficient number of iterations.
In addition, the total number of iterations in one epoch is at most $\fO(\kappa^{0.5}\fK\log(1/\epsilon))$:

\begin{proposition}\label{thm:T_bound1}
Suppose Assumption~\ref{asm:1} holds.
In the $t$-th epoch, we set the inner loop iteration number $T_{t,k}$ and the CG iteration number $T_{t,k}'$.
We run Algorithm~\ref{alg:AHGD} with the parameter chosen in Theorem~\ref{thm:one-oder}.
Then all $y_{t,k}$ and $v_{t,k}$ satisfy Condition~\ref{con:4.1}.
For each $t$, we also have the following bounds for the inner loops 
$
\sum_{k=-1}^{\fK-1}T_{t,k}
\le
\fO(\kappa^{0.5}\fK\log(1/\epsilon))
$
\text{and}
$
\sum_{k=0}^{\fK-1}T_{t,k}'
\le
\fO(\kappa^{0.5}\fK\log(1/\epsilon)).
$
\end{proposition}

The detailed forms of $T_{t,k}$ and $T_{t,k}'$ can be found in Appendix~\ref{apd:sec3}.
Combined with Theorem~\ref{thm:one-oder}, we finally obtain the total number of oracle calls as follows:

\begin{corollary}[Oracle complexity of \texttt{RAHGD} finding $\epsilon$-FOSP]\label{cor:one-order}
Under Assumption~\ref{asm:1}, we run \emph{\texttt{RAHGD}} in Algorithm~\ref{alg:AHGD} with the parameters set as in Theorem~\ref{thm:one-oder} and Proposition~\ref{thm:T_bound1}. The output $\hat{w}$ is then an $\epsilon$-first-order stationary point of~$\Phi(x)$.
Additionally, the oracle complexities satisfy
$Gc(f,\epsilon)  =\tilde\fO(\kappa^{2.75}\epsilon^{-1.75})$, $Gc(g,\epsilon) = \tilde\fO(\kappa^{3.25}\epsilon^{-1.75})$, $JV(g,\epsilon) = \tilde\fO(\kappa^{2.75}\epsilon^{-1.75})$ and $HV(g,\epsilon) = \tilde\fO(\kappa^{3.25}\epsilon^{-1.75})$.
\end{corollary}

The algorithm can be adapted to solving the single-level nonconvex minimization problem where $\kappa$ reduces to 1, and the given complexity matches the state-of-the-art~\citep{carmon2018accelerated,agarwal2017finding,carmon2017convex,jin2018accelerated,li2022restarted}.
The best known lower bound in this setting is $O(\epsilon^{-1.714})$ \citep{carmon2020lower}.
Closing this $O(\epsilon^{-0.036})$-gap remains open even in nonconvex minimization settings.

\pb\section{Perturbed Restarted Accelerated HyperGradient Descent Algorithm}\label{sec_perturb}
In this section, we introduce perturbation to our \texttt{RAHGD} algorithm.
In many nonconvex problems encountered in practice in machine learning, most first-order stationary points presented are saddle points~\citep{dauphin2014identifying,lee2019first,jin2017escape}.
Recall that finding second-order stationary points require not only zero gradient, but also a positive semidefinite Hessian, a condition that certifies that we escape all saddle points successfully.
Earlier work of~\citet{jin2018accelerated,li2022restarted} shows that one can obtain an approximate second-order stationary point by intermittently \emph{perturbing} the algorithm using random noise.
We present the details of our \emph{perturbed restarted accelerated hypergradient descent} (\texttt{P\texttt{\texttt{RAHGD}}}) in Algorithm~\ref{alg:AHGD}.
Compared with \texttt{RAHGD}, a noise-perturbation step is added in Algorithm~\ref{alg:AHGD} [Line~\ref{line:restart}, option~\texttt{Perturbation}$\,=1$].

We proceed with the complexity analysis for \texttt{PRAHGD}, where we show that \texttt{PRAHGD} in Algorithm~\ref{alg:AHGD} outputs an $(\epsilon,\sqrt{\tilde\rho\epsilon}\,)$-second-order stationary point within $\tilde \fO(\kappa^{3.25}\epsilon^{-1.75})$ oracle queries:

\begin{theorem}[\texttt{PRAHGD} finding $(\epsilon,O(\sqrt{\epsilon}))$-SOSP]\label{thm:second-order}
Suppose that Assumption~\ref{asm:1} and Condition~\ref{con:4.1} hold. We assume that $\epsilon\le\frac{\tilde L^2}{\tilde \rho}$.
Let
\begin{align*}
\chi = \fO\left(\log\frac{d_x}{\zeta\epsilon}\right)
,\quad
\eta = \frac{1}{4\tilde L}
,\quad
K = \frac{2\chi}{\theta}
,\quad
B = \frac{1}{288\chi^2}\sqrt{\frac{\epsilon}{\tilde\rho}},\quad
\theta = \frac{1}{2}
(\tilde\rho\epsilon\eta^2)^{1/4}
,\quad
\\
\sigma = \min\left\{\dfrac{\tilde \rho B \zeta r \theta}{2\sqrt{d_x}}, \epsilon^2\right\}
,\quad
\alpha=\frac{1}{\ell}
,\quad
\beta=\frac{\sqrt{\kappa}-1}{\sqrt{\kappa}+1}
,\quad
r
=
\min\left\{
\frac{\tilde LB^2}{4C},\frac{B + B^2}{\sqrt{2}}
,
\frac{\theta B}{20K},\sqrt{\frac{\theta B^2}{2K}}
\right\}
,
\end{align*}
for some positive constant $C$.
Denote $
\Delta = \Phi(x_{\rm{int}}) - \min_{x\in\BR^{d_x}}\Phi(x)
$.
Then \emph{\texttt{PRAHGD}} in Algorithm~\ref{alg:AHGD} terminates in at most $
\fO\big(
\Delta\tilde L^{0.5}\tilde\rho^{0.25}\chi^6
\cdot
\epsilon^{-1.75}
\big)
$ iterations and the output satisfies $
\norm{\nabla \Phi(\hat{w})}\le \epsilon
$ and $
\lambda_{\min}(\nabla^2\Phi(\hat{w}))
\ge
-1.011\sqrt{\tilde\rho\epsilon}
$ with probability at least $1-\zeta$.
\end{theorem}

Theorem~\ref{thm:second-order} says that \texttt{PRAHGD} in Algorithm~\ref{alg:AHGD} can find an $(\epsilon,\sqrt{\tilde\rho\epsilon}\,)$-second-order stationary point within $\tilde{\fO}(\kappa^{2.75}\epsilon^{-1.75})$ iterations in the outer loop.
The following proposition shows that Condition~\ref{con:4.1} holds in this setting. In addition, the total number of iterations in one epoch is at most $\fO(\kappa^{0.5}\,\fK\log(1/\epsilon))$:

\begin{proposition}\label{thm:T_bound2}
Suppose Assumption~\ref{asm:1} holds. In the $t$-th epoch, we set the inner loop iteration number $T_{t,k}$ and the CG iteration number $T_{t,k}'$. We run Algorithm~\ref{alg:AHGD} with the parameters chosen in Theorem~\ref{thm:second-order}. Then all $y_{t,k}$ and $v_{t,k}$ satisfy the Condition~\ref{con:4.1}. For each $t$, we also have 
 the inner loops $
\sum_{k=-1}^{\fK-1}T_{t,k}
\le
\fO\left( \kappa^{0.5}\fK\log(1/\epsilon) \right)
$
and
$
\sum_{k=0}^{\fK-1}T_{t,k}'
\le
\fO\left( \kappa^{0.5}\fK\log(1/\epsilon) \right)
$
holds.
\end{proposition}

The detailed form of $T_{t,k}$ and $T_{t,k}'$ can be found in Appendix~\ref{apd:sec_perturb}. Combining this result with Theorem~\ref{thm:second-order}, we finally obtain the total number of gradient oracle calls as follows:

\begin{corollary}[Oracle complexity of \texttt{PRAHGD} finding $(\epsilon,O(\sqrt{\epsilon}))$-SOSP]\label{cor:second-order}
Under Assumption~\ref{asm:1},  we run \emph{\texttt{PRAHGD}} in  Algorithm~\ref{alg:AHGD} with all parameters set as in Theorem~\ref{thm:second-order}.
The output $\hat{w}$ is then an $\big(\epsilon, \sqrt{\tilde\rho\epsilon}\,\big)$-second-order stationary point of~$\Phi(x)$.
Additionally, the oracle complexities satisfy that $Gc(f,\epsilon) = \tilde\fO(\kappa^{2.75}\epsilon^{-1.75})$, $Gc(g,\epsilon) = \tilde\fO(\kappa^{3.25}\epsilon^{-1.75})$, $JV(g,\epsilon) = \tilde\fO(\kappa^{2.75}\epsilon^{-1.75})$ and $HV(g,\epsilon) = \tilde\fO(\kappa^{3.25}\epsilon^{-1.75})$.
\end{corollary}

We remark that the listed query complexities are identical to the corresponding ones in Corollary~\ref{cor:one-order} of Section~\ref{sec_convergence}, up to a polylogarithmic factor, indicating that the perturbed version imposes nearly no additional cost while allowing the avoidance of saddle points.

\begin{algorithm}[!tb]
\caption{\texttt{PRAGDA}}\label{alg:PRAGDA}
\begin{algorithmic}[1]
\STATE\textbf{Input:}
initial vector $x_{0,0}$;
step-size $\eta > 0$;
momentum param.~$\theta \in (0,1)$;
params.~$\alpha > 0,\beta \in (0,1),\{T_{t,k}\}$ of AGD;
iteration threshold $K\ge 1$;
param.~$B$ for triggering restarting;
perturbation radius $r > 0$
\STATE $k\leftarrow0,\ t\leftarrow0,\ x_{0,-1}\leftarrow x_{0,0}$\\[0.1cm]
\STATE $y_{0,-1}\leftarrow{\rm AGD} (-\bar{f}(x_{0,-1},\,\cdot\,), 0, T_{0,-1}, \alpha, \beta)$  \\[0.1cm]
\STATE \textbf{while} $k<K$ \\[0.1cm]
\STATE\quad $w_{t,k}\leftarrow x_{t,k}+(1-\theta)(x_{t,k}-x_{t,k-1})$ \\[0.1cm]
\STATE\quad $y_{t,k} \leftarrow {\rm AGD} (-\bar{f}(w_{t,k},\,\cdot\,), y_{t,k-1}, T_{t,k}, \alpha, \beta)$ \\[0.1cm]
\STATE\quad $x_{t,k+1}\leftarrow w_{t,k}-\eta \nabla_x \bar{f}(w_{t,k}, y_{t,k})$ \\[0.1cm]
\STATE\quad $k\leftarrow k+1$ \\[0.1cm]
\STATE\quad \textbf{if} $k\sum_{i=0}^{k-1}\|
x_{t,i+1} - x_{t,i}
\|^2> B^2$ \\[0.1cm]
\STATE\quad\quad $x_{t+1,0}\leftarrow x_{t,k} + \xi ,\ \xi \sim \text{Unif}(\sB(r))$\\[0.1cm]
\STATE \qquad $x_{t+1,-1}\leftarrow x_{t+1,0}$
\STATE\qquad $k\leftarrow0,\ t \leftarrow t+1$ \\[0.1cm]
\STATE \qquad $y_{t,-1}\leftarrow{\rm AGD} (-\bar{f}(x_{t,-1},\,\cdot\,), 0, T_{t,-1}, \alpha, \beta)$  \\[0.1cm]
\STATE\quad \textbf{end if} \\[0.1cm]
\STATE\textbf{end while} \\[0.1cm]
\STATE $K_0\leftarrow\argmin_{\lfloor \frac{K}{2}\rfloor\leq k\leq K-1}\norm{x_{t,k+1}-x_{t,k}}$ \\[0.1cm]
\STATE \textbf{Output:} $\hat w\leftarrow\frac{1}{K_0+1}\sum_{k=0}^{K_0}w_{t,k}$
\end{algorithmic}
\end{algorithm}

\pb\section{Improved Convergence for Accelerating Minimax Optimization}\label{sec_indications}
This section applies the ideas of \texttt{PRAHGD} to find approximate second-order stationary points in minimax optimization problem of the form
\begin{align}\label{prob:minimax}
\min_{x\in\BR^{d_x}}~\left\{
\bar{\Phi}(x)\triangleq \max_{y\in \BR^{d_y}} \bar f(x,y)
\right\}
,
\end{align}
where $\bar{f}(x,y)$ is strongly concave in $y$ but possibly nonconvex in $x$. 
Problems of form (\ref{prob:minimax}) can be regarded as a special case of a bilevel optimization problem by taking 
$f(x,y)=\bar{f}(x,y)$ and $g(x,y)=-\bar{f}(x,y)$.
Danskin's theorem yields
$
\nabla\bar{\Phi}(x)=\nabla_x \bar f(x,y^*(x))
$,
in this case, which is in fact consistent with hypergradient of form~(\ref{equ:hypergradient}) 
with the optimality condition for the lower-level problem invoked, that is, $\nabla_y f(x, y^*(x))=0$.
This implies that when applying \texttt{PRAHGD} to the minimax optimization problem~(\ref{prob:minimax}), \emph{no} CG subroutine is called and \emph{no} Jacobian-vector or Hessian-vector product operation is invoked.

We first show in Lemma~\ref{lem:minimax} that the minimax problem enjoys tighter Lipschitz continuouity parameters than the general bilevel problem:
\begin{lemma}\label{lem:minimax}
Suppose that $\bar{f}(x,y)$ is $\ell$-smooth, $\rho$-Hessian Lipschitz continuous with respect to $x$ and $y$ and $\mu$-strongly concave in $y$ but possibly nonconvex in $x$.
Then the objective $\bar\Phi(x)$ is $(\kappa+1)\ell$-smooth and admits $(4\sqrt{2}\kappa^3\rho)$-Lipschitz continuous Hessians. 
\end{lemma}

We formally introduce the \emph{perturbed restarted accelerated
gradient descent ascent} (\texttt{PRAGDA}) as in Algorithm~\ref{alg:PRAGDA}.
Utilizing the \texttt{PRAHGD} complexity result as in Theorems~\ref{thm:second-order} and~\ref{thm:T_bound2} together with Lemma~\ref{lem:minimax}, we can take $\tilde L = (\kappa+1)\ell$ and $\tilde\rho = 4\sqrt{2}\kappa^3\rho$ to conclude an improved oracle complexity upper bounds for finding second-order stationary points for this particular problem, indicated by the following result:

\begin{theorem}[Oracle complexity of \texttt{PRAGDA} finding $(\epsilon,O(\sqrt{\epsilon}))$-SOSP]\label{thm:minimax}
Under the settings of Lemma \ref{lem:minimax},  Algorithm~\ref{alg:PRAGDA} outputs an $\big(\epsilon,\fO(\kappa^{1.5}\sqrt{\epsilon}\,)\big)$-second-order stationary point of $\bar\Phi(x)$ in~\eqref{prob:minimax} within $\tilde\fO(\kappa^{1.75}\epsilon^{-1.75})$ gradient oracle calls.
\end{theorem}

Prior to this work, the state-of-the-art algorithm was attained by the \emph{inexact minimax cubic Newton} (iMCN) method~\citep{luo2022finding}, which under comparable settings outputs an $\big(\epsilon,\fO(\kappa^{1.5}\sqrt{\epsilon}\,)\big)$-approximate SOSP within oracle quries of $\tilde\fO(\kappa^2\epsilon^{-1.5})$ gradients, $\tilde\fO(\kappa^{1.5}\epsilon^{-2})$ Hessian-vector products and $\tilde\fO(\kappa\epsilon^{-2})$ Jecobian-vector products.
We compare the query complexity upper bound of \texttt{PRAGDA} with iMCN in detail.
As can be observed, the total oracle complexity of \texttt{PRAGDA} is no worse than that of iMCN since $
\tilde\fO(\kappa^{1.75}\epsilon^{-1.75})
\leq
\tilde\fO(
\kappa^2\epsilon^{-1.5}
+
\kappa^{1.5}\epsilon^{-2}
)
$, a simple application of AM-GM inequality.
Moreover, \texttt{PRAGDA} only requires gradient oracle calls while iMCN additionally requires Hessian-vector and Jacobian-vector oracle calls.
To summarize, \texttt{PRAGDA} enjoys an oracle complexity that is no inferior than that of iMCN, whereas in both of the regimes $\kappa\gg \epsilon^{-1}$ and $\kappa\ll \epsilon^{-1}$ \texttt{PRAGDA}'s complexity is strictly superior.

\pb\section{Discussion}\label{sec_conclude}
We have presented the \emph{Restarted Accelerated HyperGradient Descent} (\texttt{RAHGD}) method for solving nonconvex-strongly-convex bilevel optimization problems.
Our accelerated method is able to find an $\epsilon$-first-order stationary point of the objective within $\tilde\fO(\kappa^{3.25}\epsilon^{-1.75})$ oracle complexity, where $\kappa$ is the condition number of the lower-level objective and $\epsilon$ is the desired accuracy.
Furthermore, we have proposed a perturbed variant of \texttt{RAHGD} for finding an $\big(\epsilon,\fO(\kappa^{2.5}\sqrt{\epsilon}\,)\big)$-second-order stationary point within the same order of oracle complexity up to a polylogarithmic factor.
As a byproduct, our algorithm variant \texttt{PRAGDA} improves upon the existing upper complexity bound for finding second-order stationary points in nonconvex-strongly-concave minimax optimization problems.
Important directions for future research include extending our results to the nonconvex-convex setting and stochastic setting in bilevel optimization and
also minimax optimization as its special case.

\bibliographystyle{plainnat}
\bibliography{ref}

\newpage\appendix

\pb\section{Basic Lemmas}\label{sec:basic lemmas}
In this section, we provide some basic lemmas.
\begin{lemma}
Suppose Assumption~\ref{asm:1} holds, then $y^*(x)$ is $\kappa$-Lipschitz continuous, that is, 
\begin{align*}    
\norm{y^*(x) - y^*(x')} \le \kappa \norm{x - x'}
\end{align*}
for any $x, x' \in \BR^{d_x}$.
\end{lemma}

\begin{proof}
Recall that 
\begin{align*}
y^*(x) = \arg\min_{y\in\mathbb{R}^{d_y}} g(x,y).
\end{align*}
The optimality condition leads to $\nabla_y g(x,y^*(x)) = 0$ for each $x\in\BR^{d_x}$.
By taking a further derivative with respect to $x$ on both sides and applying the chain rule~\citep{rudin1976principles}, we obtain
\begin{equation*}
\nabla_{yx}^2 g(x,y^*(x))
+
\nabla_{yy}^2 g(x,y^*(x))
\frac{\partial y^*(x)}{\partial x}
=0
.
\end{equation*}
The smoothness and strong convexity of $g$ in $y$ immediately indicate
\begin{align*}
\frac{\partial y^*(x)}{\partial x}
=
-(\nabla_{yy}^2 g(x,y^*(x)))^{-1} \nabla_{yx}^2 g(x,y^*(x))
\ .
\end{align*}
Thus we have
\begin{align*}
\norm{\frac{\partial y^*(x)}{\partial x}} 
&=
\norm{(\nabla_{yy}^2 g(x,y^*(x)))^{-1} \nabla_{yx}^2 g(x,y^*(x))  }
\le
\frac{\ell}{\mu} = \kappa
\ ,
\end{align*}
where the inequality is based on the fact that $g(x,y)$ is $\ell$-smooth with respect to $x$ and $y$ and $\mu$-strongly convex with respect to y for any $x$.

Therefore, we proved that $y^*(x)$ is $\kappa$-Lipschitz continuous.
\end{proof}
We also can show that $\Phi(x)$ admits Lipschitz continuous gradients and Lipschitz continuous Hessians, as shown in the following lemmas:
\begin{lemma}
\label{lem:gradient_Lip appd}
Suppose Assumption~\ref{asm:1} holds, then $\Phi(x)$ is $\tilde{L}$-gradient Lipschitz continuous, that is,
\begin{align*}   
\|\nabla \Phi(x) - \nabla \Phi(x')\|
\le
\tilde{L}\|x - x'\|
\end{align*}
for any $x, x' \in \BR^{d_x}$, where 
\begin{equation*}
\tilde L = \ell + \frac{2\ell^2+\rho M}{\mu}+
\frac{\ell^3 + 2\rho\ell M}{\mu^2}+
\frac{\rho\ell^2M}{\mu^3}\ .
\end{equation*}
\end{lemma}
\begin{proof}
Recall that 
\begin{equation*}
\nabla\Phi(x)
=
\nabla_x f(x, y^*(x))
-\nabla^2_{xy} g(x, y^*(x)) \big(\nabla^2_{yy} g(x, y^*(x))\big)^{-1} \nabla_y f(x, y^*(x))\ .
\end{equation*}
We denote $\fH_1(x) = \nabla_x f(x, y^*(x))$, $\fH_2(x) = \nabla^2_{xy} g(x, y^*(x))$, $\fH_3(x) = \big(\nabla^2_{yy} g(x, y^*(x))\big)^{-1}$ and $\fH_4(x) = \nabla_y f(x, y^*(x))$, then 
\begin{equation*}
\nabla\Phi(x) = \fH_1(x) - \fH_2(x)\fH_3(x)\fH_4(x)\ .
\end{equation*}
We first consider $\fH_1(x),\ \fH_2(x)$ and $\fH_4(x)$. 
For any $x,\  x'\in\mathbb{R}^{d_x}$, we have
\begin{align*}
\|\fH_1(x) -  \fH_1(x')\|&\le \ell (\|x-x'\| + \|y^*(x) - y^*(x') \|)\\
&\le \ell (1 + \kappa) \|x-x' \|\ ,
\end{align*}
where we use triangle inequality in the first inequality and Lemma~\ref{lem:gradient_Lip appd} in the second one. 

We also have
\begin{align*}
\|\fH_2(x) -  \fH_2(x')\|&\le \rho (\|x-x'\| + \|y^*(x) - y^*(x') \|)\\
&\le \rho (1 + \kappa) \|x-x' \|
\end{align*}
and
\begin{align*}
\|\fH_4(x) -  \fH_4(x')\|&\le \ell (\|x-x'\| + \|y^*(x) - y^*(x') \|)\\
&\le \ell (1 + \kappa) \|x-x' \|.
\end{align*}
We then consider $\fH_3(x)$. For any $x,\  x'\in\mathbb{R}^{d_x}$, we have
\begin{align*}
&\quad\|\fH_3(x) - \fH_3(x')\| \\
&= \left\|\big(\nabla^2_{yy} g(x, y^*(x))\big)^{-1} - \big(\nabla^2_{yy} g(x', y^*(x'))\big)^{-1}\right\|\\
&\le \left\|\big(\nabla^2_{yy} g(x, y^*(x))\big)^{-1} \right\| 
\left\|\nabla^2_{yy} g(x', y^*(x')) - \nabla^2_{yy} g(x, y^*(x)) \right\|
\left\| \big(\nabla^2_{yy} g(x', y^*(x'))\big)^{-1} \right\|\\
&\le \frac{1}{\mu^2} \rho (\|x-x'\| + \|y^*(x) - y^*(x') \|)\\
&\le \frac{\rho(1+\kappa)}{\mu^2} \|x-x'\|\ .
\end{align*}
We also have 
\begin{equation*}
\|\fH_2(x) \|\le \ell,\quad \|\fH_3(x) \|\le \frac{1}{\mu}  \quad\text{and}\quad \|\fH_4(x)\|\le M.
\end{equation*}
for any $x\in\mathbb{R}^{d_x}$.
Then for any $x,\  x'\in\mathbb{R}^{d_x}$ we have
\begin{align*}
&\quad \|\nabla\Phi(x) - \nabla\Phi(x') \|\\
&\le \|\fH_1(x) - \fH_1(x')\| + \|\fH_2(x)\fH_3(x)\fH_4(x) - \fH_2(x')\fH_3(x')\fH_4(x') \|\\
&\le \ell(1+\kappa)\|x-x'\| + 
\|\fH_2(x)\fH_3(x)\fH_4(x) - \fH_2(x)\fH_3(x)\fH_4(x') \| \\ 
&\quad+ \|\fH_2(x)\fH_3(x)\fH_4(x') - \fH_2(x)\fH_3(x')\fH_4(x')\| \\
&
\quad + \| \fH_2(x)\fH_3(x')\fH_4(x') - \fH_2(x')\fH_3(x')\fH_4(x') \|
\\
&\le \ell(1+\kappa)\|x-x'\| + \|\fH_2(x)\|\|\fH_3(x)\| \|\fH_4(x) - \fH_4(x')\| \\
&\quad + \|\fH_2(x)\|\|\fH_4(x')\| \|\fH_3(x) - \fH_3(x')\| \\
&\quad+ \|\fH_3(x')\|\|\fH_4(x')\| \|\fH_2(x) - \fH_2(x')\|\\
&\le \ell(1+\kappa)\|x-x'\| +  \frac{\ell^2}{\mu}(1+\kappa)\|x-x'\| + \frac{\ell\rho M}{\mu^2}(1+\kappa) \|x-x'\|+ \frac{M\rho}{\mu}(1+\kappa)\|x-x'\| \\
&= \pr{\ell + \frac{2\ell^2+\rho M}{\mu}+
\frac{\ell^3 + 2\rho\ell M}{\mu^2}+
\frac{\rho\ell^2M}{\mu^3}}\|x-x'\|\ .
\end{align*}
\end{proof}
\begin{lemma}
\citep[Lemma 3.4]{huang2022efficiently}.
Suppose Assumption~\ref{asm:1} holds, then $\Phi(x)$ is $\tilde{\rho}$-Hessian Lipschitz continuous, that is,  
$\|\nabla^2\Phi(x) - \nabla^2\Phi(x')\| \le \tilde{\rho}\|x - x'\|$
for any~$x, x'\in\mathbb{R}^{d_x}$, where 
\begin{align*}
\tilde \rho &= \left[
\left(\rho + \frac{2\ell\rho + M\nu}{\mu} + 
\frac{2M\ell\nu+\rho\ell^2}{\mu^2}+
\frac{M\ell^2\nu}{\mu^3}
\right)
\left(1+\frac{\ell}{\mu} \right)\right.\\
&\left.
\quad+
\left( \frac{2\ell\rho}{\mu}
+\frac{4M\rho^2+2\ell^2\rho}{\mu^2}+
\frac{2M\ell\rho^2}{\mu^3}
\right)
\left(1+\frac{\ell}{\mu} \right)^2+
\left(
\frac{M\rho^2}{\mu^2}+
\frac{\rho\ell}{\mu}
\right)
\left(
1+\frac{\ell}{\mu}
\right)^3
\right]\ .
\end{align*}
\end{lemma}

\begin{lemma}[Inexact gradients]
Suppose Assumption~\ref{asm:1} and Condition~\ref{con:4.1} hold, then we have $$
\| \nabla\Phi(w_k)- \hat{\nabla}\Phi(w_k) \|_2 \le \sigma\ .
$$
\end{lemma}
\begin{proof}
Recall that
\begin{equation*}
\nabla\Phi(x)
=
\nabla_x f(x, y^*(x))
-\nabla^2_{xy} g(x, y^*(x)) \big(\nabla^2_{yy} g(x, y^*(x))\big)^{-1} \nabla_y f(x, y^*(x))
\end{equation*}
and
\begin{align*}
\hat{\nabla} \Phi(x_k) =  \nabla_x f(x_{k}, y_{k}) - \nabla^2_{xy} g(x_{k}, y_{k}) v_{k}\ .
\end{align*}
We define
\begin{align*}
\bar{\nabla} \Phi(x_{k}) 
= \nabla_x f(x_{k}, y_{k}) - \nabla^2_{xy} g(x_{k}, y_{k}) \big(\nabla^2_{yy} g(x_{k}, y_{k})\big)^{-1} \nabla_y f(x_{k}, y_{k})\ ,
\end{align*}
then we have
\begin{align*}
||\nabla\Phi(w_k) - \hat{\nabla}\Phi(w_k) ||_2
& = ||\nabla\Phi(w_k) - \bar{\nabla}\Phi(w_k) + \bar{\nabla}\Phi(w_k) - \hat{\nabla}\Phi(w_k)||_2\\
&\le ||\nabla\Phi(w_k) - \bar{\nabla}\Phi(w_k)||_2 + ||\bar{\nabla}\Phi(w_k) - \hat{\nabla}\Phi(w_k)||_2\\
&\le \tilde{L} ||y_k - y^*(w_k)||_2 + \ell \left|\left|v_k - \big(\nabla^2_{yy} g(w_{k}, y_{k})\big)^{-1} \nabla_y f(w_{k}, y_{k})\right|\right|_2\\
&\le  \sigma\ ,
\end{align*}
where we use triangle inequality in the first inequality, Lemma~\ref{lem:gradient_Lip} and Assumption~\ref{asm:1}(c) in the second inequality and Condition~\ref{con:4.1} in the last one.
\end{proof}

\begin{lemma}
\citep[Lemmas 1 and 3]{luo2022finding}
Assume that $\bar{f}(x,y)$ is $\ell$-smooth, $\rho$-Hessian Lipschitz continuous with respect to $x$ and $y$ and $\mu$-strongly concave in $y$ but possibly nonconvex in $x$, then the objective $\bar\Phi(x)$ is $(\kappa+1)\ell$-smooth and $(4\sqrt{2}\kappa^3\rho)$-Hessian Lipschitz continuous. 
\end{lemma}

\pb\section{Proofs for Section~\ref{sec_convergence}}\label{apd:sec3}
In this section, we provide the proofs for theorems in Section~\ref{sec_convergence}. We separate our proof into three parts.
We first prove that $\Phi(x)$ decrease at least $\fO(\epsilon^{3/2})$ in one epoch and thus the total number of epochs is bounded. Then we show that our \texttt{RAHGD} in Algorithm~\ref{alg:AHGD} can output an $\epsilon$-FOSP. Finally, we provide the oracle calls complexity analysis.

\pb\subsection{Proof of Theorem~\ref{thm:one-oder}}
\label{sec:proof_first_order}
To prove Theorem~\ref{thm:one-oder} we primarily consider the algorithmic behavior in a single epoch; the desired result hence follows by a simple recursive argument.
We omit the subscript $t$ for notation simplicity.
For each epoch except the final one, we have $1 \le \fK \le K$,
\begin{align}
&
\fK\sum_{i=0}^{\fK-1} \|x_{i-1}-x_i\|_2^2
>
B^2
, \label{2a}\\
&
\|x_k - x_0\|_2^2
\le
k\sum_{i=0}^{k-1}\|x_{i+1}-x_i\|_2^2
\le
B^2
,
&
\forall k < \fK
, \label{2b}\\
&
\|w_k-x_0\|_2
\le
\|x_k - x_0\|_2 + \|x_k - x_{k-1}\|_2
\le
2B
,
&
\forall k < \fK
, \label{2c}\\
&
\norm{w_k - w_{k-1}}
\le
2B
,
&
\forall k < \fK
, \label{2d}
\end{align}
where~\eqref{2d} can be proved by induction as follows.
For $k = 0$, we have
$$
\norm{w_0 - w_{-1}}
=
0
\le
2B
.
$$
For $k = 1$, we have
$$
\norm{w_1 - w_0}
=
\norm{(x_1 - x_0) + (1-\theta)(x_1-x_0)}
\le
2B
.
$$
For $k \ge 2$, we have
\begin{align*}
\norm{w_k - w_{k-1}}
&\le
(2-\theta)\norm{x_k-x_{k-1}} + (1-\theta)\norm{x_{k-1}-x_{k-2}}
\\&\le
2 \sqrt{2\norm{x_k-x_{k-1}}^2 + 2\norm{x_{k-1} - x_{k-2}}^2}
\le
2B
.
\end{align*}

\noindent
In the last epoch, the “if condition” does not trigger and the while loop breaks until $k = K$.
Hence, we have
\begin{align}
&
\|x_k - x_0\|_2^2
\le
k\sum_{i=0}^{k-1} \|x_{i+1} - x_i\|_2^2
\le
B^2
,
&
\forall k \le K
, \label{3a}
\\&
\|w_k - x_0\|_2
\le
2B
,
&
\forall k \le K
. \label{3b}
\end{align}

In the forthcoming, we first prepare five introductory lemmas, namely Lemmas~\ref{lem:3.2}---~\ref{lem:epoch_num1}, and then prove our main Theorem~\ref{thm:one-oder}. Lemma~\ref{lem:3.2} portrays how $\Phi(x)$ decreases within an epoch in the case of large gradients. When gradient is small, we utilize quadratic functions to approximate $\Phi(x)$ as shown in Lemma~\ref{lem:3.3} and Lemma~\ref{lem:3.4} and then further we introduce Lemma~\ref{lem:3.5} which illustrates how $\Phi(x)$ decrease in the small gradient case. Combining Lemma~\ref{lem:3.2} and Lemma~\ref{lem:3.5} we can introduce Lemma~\ref{lem:epoch_num1} which shows that the total number of epochs is bounded since $\Phi(x)$ is bounded below. Finally we are ready to prove Theorem~\ref{thm:one-oder}. Figure~\ref{fig:proof_sketch} shows a pictoria description of our proving process.
\begin{figure}[!tb]
    \centering
     \includegraphics[width=11cm]{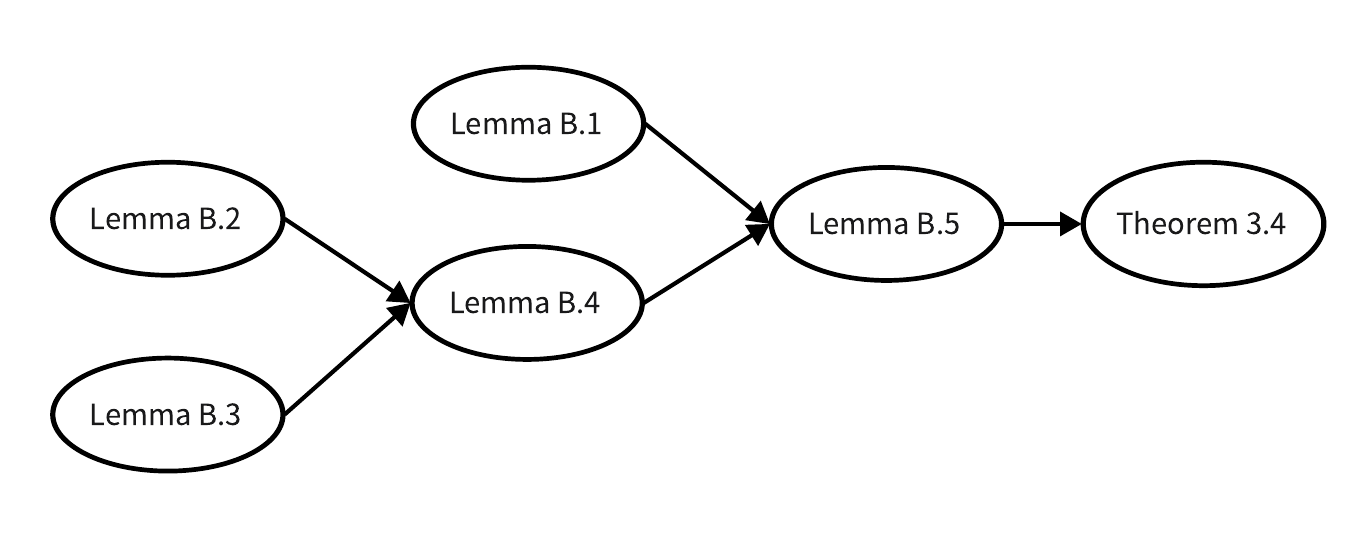}
    \caption{Proof process for Theorem~\ref{thm:one-oder}}
    \label{fig:proof_sketch}
\end{figure}

We first consider the case when $\|\nabla \Phi(w_{\fK-1})\|_2$ is large in the following lemma. 
\begin{lemma}
\label{lem:3.2}
Suppose that Assumption~\ref{asm:1} and Condition~\ref{con:4.1} hold.
Let $\eta \le \frac{1}{4\tilde{L}}$ and $0 \le \theta \le 1$.
When the “if condition” triggers and $\|\nabla \Phi(w_{\fK-1})\|_2 > \frac{B}{\eta}$, we have
$$
\Phi(x_\fK) - \Phi(x_0)
\le
- \frac{B^2}{4\eta}+ \sigma B  + \frac{5\eta\sigma^2\fK}{8}
\ .
$$
\end{lemma}

\begin{proof}
Since $\Phi(x)$ has $\tilde{L}$-Lipschitz continuous gradient, we have
\begin{align*}
\Phi(x_{k+1}) &\le \Phi(w_k) +\inner {\nabla\Phi(w_k)}{ x_{k+1}-w_k} + \frac{\tilde{L}}{2}\|x_{k+1}-w_k\|_2^2\\
&\le \Phi(w_k) - \eta\inner{\nabla\Phi(w_k)}{\hat{\nabla}\Phi(w_k)} + \frac{\eta}{8}\|\hat{\nabla}\Phi(w_k) \|_2^2\ ,
\end{align*}
where we use $\eta \le \frac{1}{4\tilde{L}}$.
We also have
\begin{align*}
\Phi(x_k) \ge \Phi(w_k) + \inner{\nabla\Phi(w_k)}{x_k - w_k} - \frac{\tilde{L}}{2} \|x_k - w_k\|_2^2\ .
\end{align*}
Combining above inequalities leads to
\begin{align*}
&\quad\Phi(x_{k+1}) - \Phi(x_k)\\
&\le 
-\inner{\nabla\Phi(w_k)}{x_k-w_k} + \frac{\tilde{L}}{2} \|x_k - w_k\|_2^2 - \eta\inner{\nabla\Phi(w_k)}{\hat{\nabla}\Phi(w_k)} + \frac{\eta}{8}\|\hat{\nabla}\Phi(w_k) \|_2^2
\\&=
\frac{1}{\eta}\inner{x_{k+1} - w_k}{x_k - w_k} + \inner{\hat{\nabla}\Phi(w_k) - \nabla\Phi(w_k)}{x_k - w_k} + \frac{\tilde{L}}{2} \|x_k - w_k\|_2^2
\\&\qquad \qquad
- \eta\inner{\nabla\Phi(w_k)}{\hat{\nabla}\Phi(w_k)} + \frac{\eta}{8}\|\hat{\nabla}\Phi(w_k) \|_2^2
\\&=
\frac{1}{2\eta}\pr{\|x_{k+1}-w_k\|_2^2 + \|x_k - w_k \|_2^2 - \|x_{k+1}-x_k\|_2^2} + \inner{\hat{\nabla}\Phi(w_k) - \nabla\Phi(w_k)}{x_k - w_k}
\\&\qquad\qquad
+ \frac{\tilde{L}}{2} \|x_k - w_k\|_2^2 
- \eta\inner{\nabla\Phi(w_k)}{\hat{\nabla}\Phi(w_k)} + \frac{\eta}{8}\|\hat{\nabla}\Phi(w_k) \|_2^2
\\&\overset{({\rm a})}{\le}
\frac{5}{8\eta} \|x_k-w_k\|_2^2 - \frac{1}{2\eta}\|x_{k+1}-x_k\|_2^2 +  \inner{\hat{\nabla}\Phi(w_k) - \nabla\Phi(w_k)}{x_k - w_k}+\frac{5\eta}{8}\|\hat{\nabla}\Phi(w_k) \|_2^2
\\&\qquad\qquad
- \eta\inner{\nabla\Phi(w_k)}{\hat{\nabla}\Phi(w_k)}
\\&\overset{({\rm b})}{\le}
\frac{5}{8\eta} \|x_k-x_{k-1}\|_2^2 - \frac{1}{2\eta}\|x_{k+1}-x_k\|_2^2 + \|\hat{\nabla}\Phi(w_k) - \nabla\Phi(w_k)\|_2\cdot\|x_k - x_{k-1}\|_2
\\&\qquad\qquad
+\frac{5\eta}{8}\|\hat{\nabla}\Phi(w_k) \|_2^2- \eta\inner{\nabla\Phi(w_k)}{\hat{\nabla}\Phi(w_k)}
\\&=
\frac{5}{8\eta} \|x_k-x_{k-1}\|_2^2 - \frac{1}{2\eta}\|x_{k+1}-x_k\|_2^2 + \|\hat{\nabla}\Phi(w_k) - \nabla\Phi(w_k)\|_2 \cdot\|x_k - x_{k-1}\|_2
\\&\qquad\qquad
+ \frac{5\eta}{8}\|\hat{\nabla}\Phi(w_k) \|_2^2-\frac{\eta}{2}\pr{\norm{\nabla\Phi(w_k)}^2 + \norm{\hat{\nabla}\Phi(w_k)}^2 - \norm{\nabla\Phi(w_k) - \hat{\nabla}\Phi(w_k)}^2}
\\&\overset{({\rm c})}{\le}
\frac{5}{8\eta} \|x_k-x_{k-1}\|_2^2 - \frac{1}{2\eta}\|x_{k+1}-x_k\|_2^2 + \|\hat{\nabla}\Phi(w_k) - \nabla\Phi(w_k)\|_2 \cdot\|x_k - x_{k-1}\|_2
\\&\qquad\qquad
- \frac{3\eta}{8}\norm{\nabla\Phi(w_k)}^2  + \frac{5\eta}{8}\norm{\nabla\Phi(w_k) - \hat{\nabla}\Phi(w_k)}^2\\
&\overset{({\rm d})}{\le
}\frac{5}{8\eta} \|x_k-x_{k-1}\|_2^2 - \frac{1}{2\eta}\|x_{k+1}-x_k\|_2^2 - \frac{3\eta}{8}\norm{\nabla\Phi(w_k)}^2+ \sigma\|x_k - x_{k-1}\|_2 + \frac{5\eta}{8}\sigma^2
\ ,
\end{align*}
where we use $\tilde{L} \le \frac{1}{4\eta}$ in $\overset{({\rm a})}{\le}$, $\norm{x_k - w_k} = (1-\theta) \norm{x_k - x_{k-1}}\le \norm{x_k - x_{k-1}}$ in $\overset{({\rm b})}{\le}$, triangle inequality in $\overset{({\rm c})}{\le}$ and Lemma~\ref{lem:bias of gradient} in $\overset{({\rm d})}{\le}$.

Summing over above inequality with $k = 0,1,\cdots, \fK-1$ and using $x_0 = x_{-1},$ we have
\begin{align}
&\quad\Phi(x_\fK) - \Phi(x_0)\\
&\le
\frac{1}{8\eta}\sum_{k=0}^{\fK-2}\norm{x_{k+1} - x_k}^2 - \frac{3\eta}{8}\sum_{k=0}^{\fK-1}\norm{\nabla\Phi(w_k)}^2 +\sigma\sum_{k=0}^{\fK-1}\norm{x_k - x_{k-1}} + \frac{5\eta\sigma^2\fK}{8}
\\&\overset{({\rm e})}{\le}
\frac{1}{8\eta}\sum_{k=0}^{\fK-2}\norm{x_{k+1} - x_k}^2 - \frac{3\eta}{8}\sum_{k=0}^{\fK-1}\norm{\nabla\Phi(w_k)}^2 +\sigma\sqrt{\fK - 1}\sqrt{\sum_{k=0}^{\fK-2}\norm{x_{k+1} - x_{k}}^2} + \frac{5\eta\sigma^2\fK}{8}
\\&\overset{({\rm f})}{\le}
\frac{B^2}{8\eta} - \frac{3\eta}{8} \norm{\nabla \Phi(w_{\fK-1})}^2 + \sigma B + \frac{5\eta\sigma^2\fK}{8} \label{equ:gradient bound}
\\& \overset{({\rm g})}{\le}
-\frac{B^2}{4\eta} + \sigma B + \frac{5\eta\sigma^2\fK}{8}\ ,
\end{align}
where we use Cauchy-Schwarz inequality in $\overset{({\rm e})}{\le}$, the “if condition” in $\overset{({\rm f})}{\le}$ and $\|\nabla \Phi(w_{\fK-1})\|_2 > \frac{B}{\eta}$ in~$\overset{({\rm g})}{\le}$.
\end{proof}

Now we consider the case when $\norm{\nabla\Phi(w_{\fK-1})}$ is small.

If $\norm{\nabla\Phi(w_{\fK-1})} \le \frac{B}{\eta}$, then \eqref{2c} and Lemma~\ref{lem:bias of gradient} lead to
\begin{align*}
\norm{x_\fK - x_0}
&\le
\norm{w_{\fK-1} - x_0} + \eta\norm{\nabla\Phi(w_{\fK-1})} + \eta \norm{\hat{\nabla}\Phi(w_{\fK-1}) - \nabla\Phi(w_{\fK-1})}
\\&\le
3B + \eta \sigma
.
\end{align*}
For epoch initialized at $x_0$, we denote $\mH = \nabla^2\Phi(x_0)$, and eigen-decompose it as $\mH = \mU {\bf\Lambda} \mU^\top$ with ${\bf\Lambda}\in\BR^{d_x\times d_x}$ being orthogonal and $\mU\in \BR^{d_x\times d_x}$ being diagonal.
Let $\lambda_j$ be the $j$-th eigenvalue of $\mH$. 
We denote $\tilde{x} = \mU^Tx,\tilde{w} = \mU^Tw$ and $\tilde{\nabla}\Phi(w) = \mU^T\nabla\Phi(w)$. 
Let $\tilde{x}^j$ and $\tilde{\nabla}^j\Phi(w)$ be the $j$-th coordinate of $\tilde{x}$ and $\tilde{\nabla}\Phi(w)$, respectively. 
Since $\Phi$ admits $\tilde{\rho}$-Lipschitz continuous Hessian, we have
\begin{equation}
\begin{aligned}
\Phi(x_\fK) - \Phi(x_0)
&\le
\inner{\nabla\Phi(x_0)}{x_\fK-x_0} + \frac{1}{2}(x_\fK - x_0)^\top\mH (x_\fK-x_0) + \frac{\tilde{\rho}}{6}\norm{x_\fK-x_0}^3
\\&=
\inner{\tilde{\nabla}\Phi(x_0)}{\tilde{x}_\fK - \tilde{x}_0} + \frac{1}{2} (\tilde{x}_\fK- \tilde{x}_0)^\top{\bf\Lambda}(\tilde{x}_\fK - \tilde{x}_0) + \frac{\tilde{\rho}}{6}\norm{x_\fK-x_0}^3
\\&\le
\phi(\tilde{x}_\fK) - \phi(\tilde{x}_0) + \frac{\tilde{\rho}}{6}(3B + \eta \sigma)^3
,
\end{aligned}\label{6a}
\end{equation}
where we denote
$$
\phi(x)
=
\inner{\tilde{\nabla}\Phi(x_0)}{x - \tilde{x}_0} + \frac{1}{2} (x - \tilde{x}_0)^\top{\bf\Lambda}(x - \tilde{x}_0)
,
$$
and
$$
\phi_j(x)
=
\inner{\tilde{\nabla}^j\Phi(x_0)}{x - \tilde{x}_0^j} + \frac{1}{2} \lambda_j(x - \tilde{x}_0^j)^2
.
$$
Let
\begin{align*}
\tilde{\delta}_k^j
=
(\mU^\top\hat{\nabla}\Phi(w_k))^j - \nabla \phi_j(\tilde{w}_k^j)
\quad\text{and} \quad
\tilde{\delta}_k
=
\mU^\top\hat{\nabla}\Phi(w_k) - \nabla\phi(\tilde{w}_k)
\ ,
\end{align*}
then the iteration of the algorithm means
\begin{align}
&
\tilde{w}_k^j
=
\tilde{x}_k^j + (1-\theta)(\tilde{x}_k^j - \tilde{x}_{k-1}^j)
, \label{6b}
\end{align}
and
\begin{align}
&
\tilde{x}_{k+1}^j
=
\tilde{w}_k^j - \eta (\mU^\top\hat{\nabla}\Phi(w_k))^j
=
\tilde{w}_k^j - \eta \nabla\phi_j(\tilde{w}_k^j) - \eta \tilde{\delta}_k^j
.
\label{6c}
\end{align}
For any $k < \fK$, we can bound $\|\tilde{\delta}_k\|_2$ as follows
\begin{align*}
\|\tilde{\delta}_k\| 
&=
\norm{\mU^\top\hat{\nabla}\Phi(w_k) - \tilde{\nabla} \Phi(w_k) + \tilde{\nabla}\Phi(w_k) - \nabla\phi(\tilde{w}_k)}
\\&\le
\norm{\tilde{\nabla}\Phi(w_k) - \tilde{\nabla}\Phi(x_0) - {\bf\Lambda}(\tilde{w}_k - \tilde{x}_0)} + \norm{\mU^\top\hat{\nabla}\Phi(w_k) - \tilde{\nabla} \Phi(w_k)}
\\&=
\norm{\nabla \Phi(w_k) - \nabla\Phi(x_0) - \mH (w_k-x_0)} + \norm{\hat{\nabla}\Phi(w_k) - \nabla\Phi(w_k)}
\\&\le
\norm{\int_0^1\inner{\nabla^2\Phi(x_0 + t(w_k-x_0)) - \mH}{w_k-x_0}\,{\rm d}t} + \sigma
\\&\le
\frac{\tilde{\rho}}{2} \norm{w_k-x_0}^2 + \sigma
\\&\le
2\tilde{\rho}B^2 + \sigma
,
\end{align*}
where the first inequality uses triangle inequality, the second one is based on Lemma~\ref{lem:bias of gradient}, the third one is based on the Lipschitz continuity of Hessian, and the last one uses~\eqref{2c}.
\medskip
\\
Notice that quadratic function $\phi(x)$ equals to the sum of $d_x$ scalar functions $\phi_j(x^j)$.
Then we decompose $\phi(x)$ into~$\sum_{j \in S_1}\phi_j(x^j)$ and $\sum_{j \in S_2}\phi_j(x^j)$, where
$$
S_1 = \left\{j:\lambda_j \ge -\frac{\theta}{\eta} \right\}
\qquad\text{and}\qquad
S_2 = \left\{j:\lambda_j < -\frac{\theta}{\eta} \right\}
\ . 
$$
We first consider the term $\sum_{j\in S_1}\phi_j(x^j)$ in the following lemma.
\begin{lemma}
\label{lem:3.3}
Suppose that Assumption~\ref{asm:1} and Condition~\ref{con:4.1} hold.
Let $\eta \le \frac{1}{4\tilde{L}}$ and $0 \le \theta \le 1$.
When the “if condition” triggers and $\|\nabla \Phi(w_{\fK-1})\|_2 \le \frac{B}{\eta}$, then we have
\begin{equation}
\sum_{j\in S_1}\phi_j(\tilde{x}^j_\fK)
-
\sum_{j\in S_1}\phi_j(\tilde{x}_0^j)
\le
-
\sum_{j \in S_1}\frac{3\theta}{8\eta}\sum_{k=0}^{\fK-1}|\tilde{x}_{k+1}^j-\tilde{x}_k^j|^2
+
\frac{2\eta\fK}{\theta}(2\tilde{\rho}B^2+\sigma)^2
\ .\label{7a} 
\end{equation}
\end{lemma}
\begin{proof}
Since $\phi_j(x)$ is quadratic, we have
\begin{align*}
& \quad \phi_j(\tilde{x}_{k+1}^j) \\
&= \phi_j(\tilde{x}_k^j) + \inner{\nabla\phi_j(\tilde{x}_k^j)}{\tilde{x}_{k+1}^j - \tilde{x}_k^j} + \frac{\lambda_j}{2}|\tilde{x}_{k+1}^j - \tilde{x}_k^j|^2\\
&\overset{({\rm a})}{=}\phi_j(\tilde{x}_k^j) - \frac{1}{\eta}\inner{\tilde{x}_{k+1}^j - \tilde w_k^j + \eta \tilde\delta_k^j}{\tilde{x}_{k+1}^j - \tilde{x}_k^j} + \inner{\nabla\phi_j(\tilde x_k^j) - \nabla\phi_j(\tilde w_k^j)}{\tilde{x}_{k+1}^j - \tilde{x}_k^j}\\
& \qquad\qquad+ \frac{\lambda_j}{2}|\tilde{x}_{k+1}^j - \tilde{x}_k^j|^2\\
&= \phi_j(\tilde{x}_k^j)-\frac{1}{\eta}\inner{\tilde{x}_{k+1}^j - \tilde w_k^j}{\tilde{x}_{k+1}^j - \tilde{x}_k^j} -\inner{\tilde\delta_k^j}{\tilde x_{k+1}^j-\tilde x_k^j} + \lambda_j\inner{\tilde x_k^j - \tilde w_k^j}{\tilde x_{k+1}^j-\tilde x_k^j} \\
& \qquad\qquad +\frac{\lambda_j}{2}|\tilde{x}_{k+1}^j - \tilde{x}_k^j|^2\\
&= \phi_j(\tilde{x}_k^j) + \frac{1}{2\eta}\pr{|\tilde x_k^j -\tilde w_k^j|^2 - |\tilde x_{k+1}^j - \tilde w_k^j|^2 - |\tilde x_{k+1}^j-\tilde x_k^j|^2} -\inner{\tilde \delta_k^j}{\tilde x_{k+1}^j - \tilde x_k^j} \\
&\qquad\qquad + \frac{\lambda_j}{2}\pr{|\tilde x_{k+1}^j - \tilde w_k^j|^2 -|\tilde x_k^j - \tilde w_k^j|^2}\\
&\le \phi_j(\tilde{x}_k^j) + \frac{1}{2\eta}\pr{|\tilde x_k^j -\tilde w_k^j|^2- |\tilde x_{k+1}^j - \tilde w_k^j|^2 - |\tilde x_{k+1}^j-\tilde x_k^j|^2} + \frac{1}{2\alpha}|\tilde\delta_k^j|^2 + \frac{\alpha}{2}|\tilde x_{k+1}^j - \tilde x_k^j|^2 \\
&\qquad\qquad+\frac{\lambda_j}{2}\pr{|\tilde x_{k+1}^j - \tilde w_k^j|^2 -|\tilde x_k^j - \tilde w_k^j|^2},
\end{align*}
where we use~\eqref{6b} in $\overset{({\rm a})}{=}$. 

Using the fact $\tilde L \ge \lambda_j \ge -\frac{\theta}{\eta}$ for $j \in S_1$ and 
\begin{align*}
\pr{-\frac{1}{2\eta} + \frac{\lambda_j}{2}}|\tilde x_{k+1}^j-\tilde w_k^j|^2 \le \pr{-2\tilde L + \frac{\tilde L}{2}}|\tilde x_{k+1}^j-\tilde w_k^j|^2 \le 0,    
\end{align*}
we have 
\begin{align*}
&\quad \phi_j(\tilde x_{k+1}^j) \\
&\le \phi_j(\tilde x_k^j)+\frac{1}{2\eta}\pr{|\tilde x_k^j - \tilde w_k^j|^2 - |\tilde x_{k+1}^j - \tilde x_k^j|^2} + \frac{1}{2\alpha}|\tilde\delta_k^j|^2 + \frac{\alpha}{2}|\tilde x_{k+1}^j - \tilde x_k^j|^2 + \frac{\theta}{2\eta}|\tilde x_k^j - \tilde w_k^j|^2\\
&\overset{({\rm b})}{=} \phi_j(\tilde x_k^j)+ \frac{(1-\theta)^2(1+\theta)}{2\eta}|\tilde x_k^j - \tilde x_{k-1}^j|^2 - \pr{\frac{1}{2\eta} - \frac{\alpha}{2}}|\tilde x_{k+1}^j - \tilde x_k^j|^2 + \frac{1}{2\alpha} |\tilde \delta_k^j|^2\\
&= \phi_j(\tilde x_k^j)+ \frac{(1-\theta)^2(1+\theta)}{2\eta}\pr{|\tilde x_k^j - \tilde x_{k-1}^j|^2 - |\tilde x_{k+1}^j - \tilde x_k^j|^2} \\&
\qquad\qquad- \pr{\frac{1}{2\eta} - \frac{\alpha}{2} - \frac{(1-\theta)^2(1+\theta)}{2\eta}}|\tilde x_{k+1}^j - \tilde x_k^j|^2 + \frac{1}{2\alpha} |\tilde \delta_k^j|^2\\
&\overset{({\rm c})}{\le}  \phi_j(\tilde x_k^j)+ \frac{(1-\theta)^2(1+\theta)}{2\eta}\pr{|\tilde x_k^j - \tilde x_{k-1}^j|^2 - |\tilde x_{k+1}^j - \tilde x_k^j|^2} - \frac{3\theta}{8\eta}|\tilde x_{k+1}^j - \tilde x_k^j|^2 + \frac{2\eta}{\theta} |\tilde \delta_k^j|^2
\end{align*}
for each $j \in S_1$, where we use~\eqref{6c} in $\overset{({\rm b})}{=}$ and let $\alpha = \frac{\theta}{4\eta}$ in $\overset{({\rm c})}{\le}$ which leads to 
\begin{align*}
\frac{1}{2\eta} - \frac{\alpha}{2} - \frac{(1-\theta)^2(1+\theta)}{2\eta} 
=\frac{1}{2\eta} - \frac{\theta}{8\eta} - \frac{(1-\theta)^2(1+\theta)}{2\eta} 
= \frac{3\theta}{8\eta} + \frac{\theta^2 - \theta^3}{2\eta} \ge \frac{3\theta}{8\eta}.    
\end{align*}
Summing over above result with $ k = 0,1,\cdots,\fK-1 $ for $j \in S_1$ and using $x_0 = x_{-1}$, we have
\begin{align*}
&\quad \sum_{j\in S_1} \phi_j(\tilde x_{\fK}^j) \\
&\le \sum_{j\in S_1}\phi_j(\tilde x_0^j) - \sum_{j\in S_1}\frac{3\theta}{8\eta} \sum_{k=0}^{\fK-1}|\tilde{x}_{k+1}^j-\tilde{x}_k^j|^2 + \frac{2\eta}{\theta} \sum_{k=0}^{\fK-1}\norm{\tilde \delta_k}^2 -\frac{(1-\theta)^2(1+\theta)}{2\eta} |\tilde x_\fK^j - \tilde x_{\fK-1}^j|^2 \\
&\le \sum_{j\in S_1}\phi_j(\tilde x_0^j) -\sum_{j \in S_1}\frac{3\theta}{8\eta}\sum_{k=0}^{\fK-1}|\tilde{x}_{k+1}^j-\tilde{x}_k^j| + \frac{2\eta\fK}{\theta}(2\tilde{\rho}B^2+\sigma)^2
.
\end{align*}
This completes the proof.
\end{proof}
\bigskip
\noindent
Next, we consider the term $\sum_{j\in S_2}\phi_j(x^j)$.
\begin{lemma}
\label{lem:3.4}
Suppose that Assumption~\ref{asm:1} and Condition~\ref{con:4.1} hold.
Let $\eta \le \frac{1}{4\tilde{L}}$ and $0 \le \theta \le 1$.
When the “if condition” triggers and $\|\nabla \Phi(w_{\fK-1})\|_2 \le \frac{B}{\eta}$, then we have
\begin{equation}
\sum_{j\in S_2}\phi_j(\tilde{x}^j_\fK)
-
\sum_{j\in S_2}\phi_j(\tilde{x}_0^j)
\le
-
\sum_{j \in S_2}\frac{\theta}{2\eta}\sum_{k=0}^{\fK-1}|\tilde{x}_{k+1}^j-\tilde{x}_k^j|^2
+
\frac{\eta\fK}{2\theta}(2\tilde{\rho}B^2+\sigma)^2
+
\frac{\eta\fK}{2\theta}\sigma^2
\ .   \label{7b}
\end{equation}
\end{lemma}
\begin{proof}
We denote $\nu_j = \tilde x_0^j - \frac{1}{\lambda_j}\tilde\nabla^j\Phi(x_0)$, then $\phi_j(x)$ can be rewritten as
\begin{align*}
\phi_j(x) & = \frac{\lambda_j}{2}\pr{x - \tilde x_0^j + \frac{1}{\lambda_j}\tilde\nabla^j\Phi(x_0)}^2 - \frac{1}{2\lambda_j}|\tilde\nabla^j\Phi(x_0)|^2
= \frac{\lambda_j}{2} (x-\nu_j)^2 - \frac{1}{2\lambda_j}|\tilde\nabla^j\Phi(x_0)|^2\ .
\end{align*}
For each $j\in S_2 = \{j:\lambda_j < -\frac{\theta}{\eta}\}$, we have
\begin{equation}
\label{A.3}
\begin{aligned}
\phi_j(\tilde x_{k+1}^j) - \phi_j(\tilde x_k^j) &= \frac{\lambda_j}{2} |\tilde x_{k+1}^j-\nu_j|^2 - \frac{\lambda_j}{2}|\tilde x_k^j - \nu_j|^2\\
&= \frac{\lambda_j}{2}|\tilde x_{k+1}^j - \tilde x_k^j|^2 + \lambda_j \inner{\tilde x_{k+1}^j - \tilde x_k^j}{\tilde x_k^j-\nu_j}\\
&\le -\frac{\theta}{2\eta}|\tilde x_{k+1}^j - \tilde x_k^j|^2 + \lambda_j\inner{\tilde x_{k+1}^j - \tilde x_k^j}{\tilde x_k^j-\nu_j}\ .
\end{aligned}
\end{equation}
So we only need to bound the second part. From~\eqref{6b} and~\eqref{6c}, we have
\begin{align*}
\tilde x_{k+1}^j - \tilde x_k^j &= \tilde w_k^j - \tilde x_k^j - \eta \nabla \phi_j(\tilde w_k^j) - \eta \tilde\delta_k^j\\
&= (1-\theta) (\tilde x_{k}^j - \tilde x_{k-1}^j) - \eta\lambda_j(\tilde w_k^j - \nu_j) - \eta \tilde\delta_k^j\\
&= (1-\theta) (\tilde x_{k}^j - \tilde x_{k-1}^j) - \eta\lambda_j(\tilde x_k^j-\nu_j + (1-\theta)(\tilde x_k^j - \tilde x_{k-1}^j)) - \eta \tilde\delta_k^j\ .
\end{align*}
So for each $j \in S_2$, we have
\begin{align*}
\small\begin{split}    
&\quad\inner{\tilde x_{k+1}^j - \tilde x_k^j}{\tilde x_k^j-\nu_j}\\
& = (1-\theta)\inner{\tilde x_{k}^j - \tilde x_{k-1}^j}{\tilde x_k^j-\nu_j} - \eta\lambda_j|\tilde x_k^j-\nu_j|^2-\eta\lambda_j(1-\theta)\inner{\tilde x_k^j-\tilde x_{k-1}^j}{\tilde x_k^j-\nu_j}-\eta\inner{\tilde\delta_k^j}{\tilde x_k^j-\nu_j}\\
&\ge (1-\theta)\inner{\tilde x_{k}^j - \tilde x_{k-1}^j}{\tilde x_k^j-\nu_j} - \eta\lambda_j|\tilde x_k^j-\nu_j|^2 \\
&\qquad + \frac{\eta\lambda_j(1-\theta)}{2}(|\tilde x_k^j-\tilde x_{k-1}^j|^2 + |\tilde x_k^j - \nu_j|^2) + \frac{\eta}{2\lambda_j(1+\theta)}|\tilde \delta_k^j|^2 + \frac{\eta\lambda_j(1+\theta)}{2}|\tilde x_k^j-\nu_j|^2\\
&=(1-\theta)\inner{\tilde x_{k}^j - \tilde x_{k-1}^j}{\tilde x_k^j-\nu_j} + \frac{\eta\lambda_j(1-\theta)}{2}|\tilde x_k^j-\tilde x_{k-1}^j|^2 + \frac{\eta}{2\lambda_j(1+\theta)}|\tilde \delta_k^j|^2\\
&= (1-\theta)\inner{\tilde x_{k}^j - \tilde x_{k-1}^j}{\tilde x_{k-1}^j-\nu_j} +(1-\theta)|\tilde x_k^j - \tilde x_{k-1}^j|^2+ \frac{\eta\lambda_j(1-\theta)}{2}|\tilde x_k^j-\tilde x_{k-1}^j|^2 + \frac{\eta}{2\lambda_j(1+\theta)}|\tilde \delta_k^j|^2\\
&\ge (1-\theta)\inner{\tilde x_{k}^j - \tilde x_{k-1}^j}{\tilde x_{k-1}^j-\nu_j} + \frac{\eta}{2\lambda_j}|\tilde \delta_k^j|^2\ ,
\end{split}
\end{align*}
where we use the fact that $\lambda_j < 0$ when $j\in S_2$ in the first inequality and the fact 
\begin{align*}
\pr{1+\frac{\eta\lambda_j}{2}}(1-\theta)\ge \pr{1-\frac{\eta\tilde L}{2}}(1-\theta)\ge0    
\end{align*}
indicates the second inequality. Then we have
\begin{align*}
\small
\begin{split}    
&~~~~\inner{\tilde x_{k+1}^j-\tilde x_k^j}{\tilde x_k^j-\nu_j} \\
&\ge (1-\theta)^k\inner{\tilde x_1^j-\tilde x_0^j}{\tilde x_0^j-\nu_j} + \frac{\eta}{2\lambda_j}\sum_{i=1}^k(1-\theta)^{k-i}|\tilde \delta_i^j|^2\\
&\overset{({\rm a})}{=} -\frac{\eta}{2\lambda_j}(1-\theta)^k \inner{\nabla^j\Phi(x_0)}{\hat{\nabla}^j\Phi(x_0)}  + \frac{\eta}{2\lambda_j}\sum_{i=1}^k(1-\theta)^{k-i}|\tilde \delta_i^j|^2\\
&= -\frac{\eta}{2\lambda_j}(1-\theta)^k \pr{|\nabla^j\Phi(x_0)|^2 + |\hat{\nabla}^j\Phi(x_0)|^2 - |\nabla^j\Phi(x_0)-\hat{\nabla}^j\Phi(x_0)|^2 }  + \frac{\eta}{2\lambda_j}\sum_{i=1}^k(1-\theta)^{k-i}|\tilde \delta_i^j|^2\\
&\overset{({\rm b})}{\ge} \frac{\eta}{2\lambda_j}(1-\theta)^k \pr{ |\nabla^j\Phi(x_0)-\hat{\nabla}^j\Phi(x_0)|^2 }  + \frac{\eta}{2\lambda_j}\sum_{i=1}^k(1-\theta)^{k-i}|\tilde \delta_i^j|^2,
\end{split}
\end{align*}
where we use 
\begin{align*}
\tilde x_1^j - \tilde x_0^j=\tilde x_1^j - \tilde w_0^j = -\eta \pr{\mU^\top\hat{\nabla}\Phi(x_0)}^j
\quad\text{and}\quad
\tilde x_0^j - \nu_j = -\frac{1}{\lambda_j}\tilde\nabla^j\Phi(x_0)
\end{align*}
in $\overset{(\rm{a})}{=}$ and $\lambda_j < 0$ in $\overset{({\rm b})}{\ge}$. Plugging above inequality into~\eqref{A.3} and using $\lambda_j<0$, we have
\begin{align*}
\small
\begin{split}    
\phi_j(\tilde x_{k+1}^j) - \phi_j(\tilde x_k^j)
&\le -\frac{\theta}{2\eta}|\tilde x_{k+1}^j - \tilde x_k^j|^2 + \frac{\eta}{2}(1-\theta)^k \pr{ |\nabla^j\Phi(x_0)-\hat{\nabla}^j\Phi(x_0)|^2 }  + \frac{\eta}{2}\sum_{i=1}^k(1-\theta)^{k-i}|\tilde \delta_i^j|^2.
\end{split}
\end{align*}
Summing over above result with $k=0,1,\dots,\fK-1$ for $j \in S_2$, we have
\begin{align*}
&\quad\sum_{j\in S_2}\phi_j(\tilde{x}^j_\fK) - \sum_{j\in S_2}\phi_j(\tilde{x}_0^j)\\
&\le-\sum_{j \in S_2}\frac{\theta}{2\eta}\sum_{k=0}^{\fK-1}|\tilde{x}_{k+1}^j-\tilde{x}_k^j|^2 + \frac{\eta}{2}\norm{\nabla\Phi(x_0) - \hat{\nabla}\Phi(x_0)}^2\sum_{k=0}^{\fK-1}(1-\theta)^k + \frac{\eta}{2}\sum_{k=0}^{\fK-1}\sum_{i=1}^k(1-\theta)^{k-i}\|\tilde \delta_i\|_2^2\\
&\le -\sum_{j \in S_2}\frac{\theta}{2\eta}\sum_{k=0}^{\fK-1}|\tilde{x}_{k+1}^j-\tilde{x}_k^j|^2 + \frac{\eta}{2}\sigma^2\sum_{k=0}^{\fK-1}(1-\theta)^k + \frac{\eta}{2}\sum_{k=0}^{\fK-1}\sum_{i=1}^k(1-\theta)^{k-i}\|\tilde \delta_i\|_2^2\\
&\le -\sum_{j \in S_2}\frac{\theta}{2\eta}\sum_{k=0}^{\fK-1}|\tilde{x}_{k+1}^j-\tilde{x}_k^j|^2 + \frac{\eta\fK}{2\theta}\sigma^2 + \frac{\eta\fK}{2\theta}(2\tilde\rho B+\sigma)^2,
\end{align*}
which completes the  proof.
\end{proof}

Putting Lemmas~\ref{lem:3.3} and ~\ref{lem:3.4} together we can lower bound the decrement of $\Phi(x)$ in a single epoch.
\begin{lemma}
\label{lem:3.5}
Suppose that Assumption~\ref{asm:1} and Condition~\ref{con:4.1} hold.
Let $\eta \le \frac{1}{4\tilde{L}}$ and $0 \le \theta \le 1$.
When the “if condition” triggers and $\|\nabla \Phi(w_{\fK-1})\|_2 \le \frac{B}{\eta}$, we have
$$
\Phi(x_\fK) - \Phi(x_0)
\le
-\frac{\epsilon^{3/2}}{\sqrt{\tilde \rho}}
\ .
$$
\end{lemma}
\begin{proof}
\noindent
Summing over~\eqref{7a} and~\eqref{7b}, we have
\begin{equation}
\begin{aligned}
\phi(\tilde x_\fK) - \phi(\tilde x_0)
& =
\sum_{j \in S_1\cup S_2} \phi_j(\tilde x_\fK^j) - \phi_j(\tilde x_0^j)
\\&\le
\frac{3\theta}{8\eta}\sum_{k=0}^{\fK-1}\norm{\tilde{x}_{k+1}-\tilde{x}_k}^2 + \frac{5\eta\fK}{2\theta}(2\tilde{\rho}B^2+\sigma)^2
+ \frac{\eta\fK}{2\theta}\sigma^2\\
&\le
-\frac{3\theta B^2}{8\eta\fK} +  \frac{5\eta\fK}{2\theta}(2\tilde{\rho}B^2+\sigma)^2+ \frac{\eta\fK}{2\theta}\sigma^2
,
\end{aligned}
\end{equation}
where we use~\eqref{2a} in the last inequality. Plugging into~\eqref{6a} and using $\fK \le K$, we have
\begin{equation}
\begin{aligned}
\Phi(x_\fK) - \Phi(x_0)
&\le
-\frac{3\theta B^2}{8\eta\fK} +  \frac{5\eta\fK}{2\theta}(2\tilde{\rho}B^2+\sigma)^2 + \frac{\tilde\rho}{6} (3B + \eta\sigma)^3+ \frac{\eta\fK}{2\theta}\sigma^2
\\&\le
-\frac{3\theta B^2}{8\eta K} +  \frac{5\eta K}{2\theta}(2\tilde{\rho}B^2+\sigma)^2 + \frac{\tilde\rho}{6} (3B + \eta\sigma)^3+ \frac{\eta\fK}{2\theta}\sigma^2
\\&\le
-\frac{\epsilon^{3/2}}{\sqrt{\tilde \rho}}
\ .
\end{aligned}
\end{equation}
This completes the proof.
\end{proof}

Now we can provide an upper bound on the total number of epochs, as shown in the following lemma.
\begin{lemma}
\label{lem:epoch_num1}
Consider the setting of Theorem~\ref{thm:one-oder}, and we run \texttt{RAHGD} in Algorithm~\ref{alg:AHGD}.
Then the algorithm terminates in at most $\Delta \sqrt{\tilde\rho}\epsilon^{-3/2}$ epochs.
\end{lemma}

\begin{proof}[Proof of Lemma~\ref{lem:epoch_num1}]
From Lemma~\ref{lem:3.2} and ~\ref{lem:3.5}, we have
\begin{equation}
\Phi(x_\fK) - \Phi(x_0)
\le
-\min\fpr{
\frac{\epsilon^{3/2}}{\sqrt{\tilde \rho}}
,
\frac{\epsilon\tilde L}{\tilde \rho}
} \ .
\end{equation}
Notice that in Algorithm~\ref{alg:AHGD}, we set $x_0$ to be the last iterate $x_{\fK}$ in the previous epoch.
Suppose the total number of epochs $N > \Delta \sqrt{\tilde\rho}\epsilon^{-3/2}$, then summing over all epochs we have
\begin{equation}\label{epoch_number}
\min_{x\in\BR^{d_x}}\Phi(x) - \Phi(x_{\rm{int}})
\le
-N\min \fpr{
\frac{\epsilon^{3/2}}{\sqrt{\tilde \rho}}
,
\frac{\epsilon\tilde L}{\tilde \rho}
}
<
-\Delta
,
\end{equation}
leading to contradiction.
Therefore the algorithm terminates in at most $\Delta \sqrt{\tilde\rho}\epsilon^{-3/2}$ epochs. 
\end{proof}

We are now prepared to finish the proof of Theorem~\ref{thm:one-oder}.
\begin{proof}[Proof of Theorem~\ref{thm:one-oder}]
Lemma~\ref{lem:epoch_num1} says that \texttt{RAHGD} will terminate in at most $\Delta \sqrt{\tilde\rho}\epsilon^{-3/2}$ epochs.
Since each epoch needs at most $K = \frac{1}{2}\big(\tilde L^2/(\tilde\rho\epsilon)\big)^{1/4}$ iterations, the total iterations must be less than $\Delta\tilde L^{1/2}\tilde \rho^{1/4}\epsilon^{-7/4}$. 
Recall that we have $\tilde L = \fO(\kappa^3)$ and $\tilde\rho = \fO(\kappa^5)$, thus the total iterations is at most $\fO(\kappa^{11/4}\epsilon^{-7/4})$.
\medskip
\\
Now we consider the last epoch.
Denote $\tilde w = \mU^\top\hat{w} = \frac{1}{K_0+1}\sum_{k=0}^{K_0}\mU^\top w_{k} = \frac{1}{K_0 + 1}\sum_{k=0}^{K_0}\tilde w_k$.
Since $\phi$ is quadratic, we have
\begin{align*}
\norm{\phi(\tilde w)}
&=
\norm{\frac{1}{K_0+1}\sum_{k=0}^{K_0}\nabla \phi(\tilde w_k)}
\\&\overset{({\rm a})}{=}
\frac{1}{\eta(K_0+1)}\norm{\sum_{k=0}^{K_0}\pr{\tilde x_{k+1} - \tilde w_k + \eta \tilde\delta_k}}
\\&=
\frac{1}{\eta(K_0+1)}\norm{\sum_{k=0}^{K_0}\pr{\tilde x_{k+1} - \tilde x_k - (1-\theta)(\tilde x_k - \tilde x_{k-1})+\eta\tilde\delta_k}}
\\&\overset{({\rm b})}{=}
\frac{1}{\eta(K_0+1)}\norm{\tilde x_{K_0+1} - \tilde x_0 -(1-\theta)(\tilde x_{K_0} - \tilde x_0 + \eta\sum_{k=0}^{K_0}\tilde\delta_k}
\\&=
\frac{1}{\eta (K_0+1)}\norm{\tilde x_{K_0+1} - \tilde x_{K_0} + \theta(\tilde x_{K_0} - \tilde x_0) + \eta\sum_{k=0}^{K_0}\tilde\delta_k}
\\&\le
\frac{1}{\eta(K_0+1)}\pr{\norm{\tilde x_{K_0+1} - \tilde x_{K_0}} + \theta \norm{\tilde x_{K_0} - \tilde x_0} + \eta\sum_{k=0}^{K_0}\norm{\tilde\delta_k}}
\\&\overset{({\rm c})}{\le}
\frac{2}{\eta K}\norm{\tilde x_{K_0+1} - \tilde x_{K_0}} + \frac{2\theta B}{\eta K} + 2\tilde \rho B^2 + \sigma
,
\end{align*}
where we use~\eqref{6c} in $\overset{({\rm a})}{=}$, $x_{-1} = x_0$ in $\overset{({\rm b})}{=}$; $K_0 + 1 \ge \frac{K}{2}$,~\eqref{3a} and~\eqref{3b} in~$\overset{({\rm c})}{\le}$. 

From $K_0 = \argmin_{\lfloor \frac{K}{2}\rfloor\leq k\leq K-1}~\norm{x_{k+1}-x_{k}}$, we have
\begin{align*}
\norm{x_{K_0+1} - x_{K_0}}^2
&\le
\frac{1}{K - \floor{K/2}}\sum_{k=\floor{K/2}}^{K-1}\norm{x_{k+1} - x_k}^2
\le
\frac{1}{K - \floor{K/2}}\sum_{k=0}^{K-1}\norm{x_{k+1} - x_k}^2
\\&\overset{({\rm d})}{\le}
\frac{1}{K - \floor{K/2}}\frac{B^2}{K}
\le
\frac{2B^2}{K^2}
,
\end{align*}
where we use~\eqref{3a} in $\overset{({\rm d})}{\le}$. On the other hand, we also have
\begin{align*}
\norm{\nabla\Phi(\hat{w})}
=
\norm{\tilde \nabla \Phi(\hat{w})}
&\le
\norm{\nabla\phi(\tilde w)} + \norm{\tilde \nabla \Phi(\hat{w}) - \nabla\phi(\tilde w)}
\\&=
\norm{\nabla\phi(\tilde w)} \norm{\tilde \nabla\Phi(\hat{w}) - \tilde\nabla\Phi(x_0) - {\bf\Lambda}(\tilde w - \tilde x_0)}
\\&=
\norm{\nabla\phi(\tilde w)} + \norm{\nabla \Phi(\hat{w}) - \nabla \Phi(x_0) - \mH(\hat{w} - x_0)}
\\&\le
\norm{\nabla\phi(\tilde w)} + \frac{\tilde\rho}{2}\norm{\hat{w}-x_0}^2
\overset{({\rm e})}{\le}
\norm{\nabla\phi(\tilde w)} + 2\tilde\rho B^2
 \ ,
\end{align*}
where we use $
\norm{\hat{w} - x_0} \le \frac{1}{K_0+1}\sum_{k=0}^{K_0}\norm{w_k - x_0}\le 2B
$ from~\eqref{3b} in $\overset{({\rm e})}{\le}$.
So we have
$$
\norm{\nabla\Phi(\hat{w})}
\le
\frac{2\sqrt{2}B}{\eta K^2} + \frac{2\theta B}{\eta K} + 4\tilde\rho B^2 + \sigma
\le
83\epsilon
.
$$
This completes our proof of Theorem~\ref{thm:one-oder}.
\end{proof}

\pb\subsection{Proof of Proposition~\ref{thm:T_bound1}}
\label{sec:proof_T_bound1}
\begin{proof}
We first consider the iterations of CG in Algorithm~\ref{alg:AHGD} in one epoch. We set $T_{t,k}'$ as 
\begin{equation}\label{equ:T_tk'}
T_{t,k}' = \left\{
\begin{array}{ll}
\ceil{\frac{\sqrt{\kappa}+1}{2}\log\pr{\frac{4\ell\sqrt{\kappa}}{\sigma}\pr{\norm{v_{0,-1}} +  \frac{M}{\mu}}}}
,&k=0,
\\[0.3cm]
\ceil{\frac{\sqrt{\kappa}+1}{2}\log\pr{\frac{4\ell\sqrt{\kappa}}{\sigma}\pr{\frac{\sigma}{2\ell} + \frac{2M}{\mu}}}}
,& k \ge 1.
\end{array}
\right.
\end{equation}
We denote 
\begin{align*}    
v^*(x,y) = \big(\nabla^2_{yy} g(x,y)\big)^{-1} \nabla_y f(x, y),
\end{align*}
then 
\begin{align*}    
\norm{v^*(x,y)} \le \frac{M}{\mu}
,\quad
\forall x \in \BR^{d_x},y \in \BR^{d_y}.
\end{align*}
We use induction to show that
\begin{align*}    
\norm{v_{t,k} - v^*_{t,k}} \le \frac{\sigma}{2\ell}
\end{align*}
holds for any $k \ge 0$.
For $k = 0$, Lemma~\ref{lem:CG} straightforwardly implies that 
\begin{align*}
\norm{v_{t,0} - v_{t,0}^*} \le \frac{\norm{v_{0,-1}-v_{t,0}^*}}{\norm{v_{0,-1}} + M/\mu}\cdot\frac{\sigma}{2\ell}
\le
\frac{\sigma}{2\ell}.    
\end{align*}
Suppose it holds that $\norm{v_{t,k} - v^*_{t,k}} \le \frac{\sigma}{2\ell}$ for any $k = k'-1$, then we have
\begin{align*}
\norm{v_{t,k'} - v_{t,k'}^*}
&\le
2\sqrt{\kappa}\pr{1-\frac{2}{1+\sqrt{\kappa}}}^{T_{t,k'}'}\norm{v_{t,k'-1} - v_{t,k'}^*}
\\&\le
2\sqrt{\kappa}\pr{1-\frac{2}{1+\sqrt{\kappa}}}^{T_{t,k'}'} \pr{\norm{v_{t,k'-1} - v_{t,k'-1}^*} + \norm{v_{t,k'-1}^* - v_{t,k'}^*}}
\\&\le
2\sqrt{\kappa}\pr{1-\frac{2}{1+\sqrt{\kappa}}}^{T_{t,k'}'} \pr{\frac{\sigma}{2\ell} + \frac{2M}{\mu}}
\le
\frac{\sigma}{2\ell},
\end{align*}
where the first inequality is based on Lemma~\ref{lem:CG}, the second one uses triangle inequality, the third one uses the definition of $T_k'$. Therefore, \eqref{equ:cond2} in Condition~\ref{con:4.1} can hold.

The total iteration number of CG in Algorithm~\ref{alg:AHGD} in one epoch satisfies
\begin{align*}
\small
\begin{split}
&\quad \sum_{k=0}^{\fK-1}T_k' \\
&\le
\fK + \frac{\sqrt{\kappa}+1}{2}\left(\frac{2T_0'}{\sqrt{\kappa}+1} + \sum_{k=1}^{\fK-1} \log \pr{\frac{4\ell\sqrt{\kappa}}{\sigma}\pr{\frac{\sigma}{2\ell} + \frac{2M}{\mu}}} \right)
\\&=
\fK + \frac{\sqrt{\kappa}+1}{2}\left(\frac{2T_0'}{\sqrt{\kappa}+1} + \pr{\fK-1}\log \pr{\frac{4\ell\sqrt{\kappa}}{\sigma}\pr{\frac{\sigma}{2\ell} + \frac{2M}{\mu}}} \right)
\\&=
\fK + \frac{\sqrt{\kappa}+1}{2}\fK \left(\frac{1}{\fK}\log\pr{\frac{4\ell\sqrt{\kappa}}{\sigma}\pr{\norm{v_{0,-1}} + \frac{M}{\mu}}}
+ \pr{1-\frac{1}{\fK}} \log\pr{\frac{4\ell\sqrt{\kappa}}{\sigma}\pr{\frac{\sigma}{2\ell} + \frac{2M}{\mu}}}\right)
\ .
\end{split}
\end{align*}
Now we consider the iterations of AGD in Algorithm~\ref{alg:AHGD}.We first show the following lemma.
\begin{lemma}
\label{lem:used in proof T_bound1}
Consider the setting of Theorem~\ref{thm:one-oder}, and we run Algorithm~\ref{alg:AHGD}, then we have
$$\norm{y^*(w_{t,-1})} \le \hat{C}$$
for any $t>0$, where $\hat{C} = \norm{y^*(x_{0,0})} + (2B + \eta\sigma + \eta C)\kappa\Delta \sqrt{\tilde\rho}\epsilon^{-3/2}.$
\end{lemma}
Then consider the iteratons of AGD in Algorithm~\ref{alg:AHGD}. We choose $T_{t,k}$ as
\begin{equation}\label{equ:T_tk}
T_{t,k} = \left\{
\begin{array}{ll}
\ceil{2\sqrt{\kappa}\log\pr{\frac{2\tilde L\sqrt{\kappa+1}}{\sigma}\hat{C}}}
,&
k=-1.
\\[0.3cm]
\ceil{2\sqrt{\kappa}\log\pr{\frac{2\tilde L\sqrt{\kappa+1}}{\sigma}\pr{\frac{\sigma}{2\tilde L}+2\kappa B}}}
,&
k \ge 0.
\end{array}
\right.
\end{equation}
We will use induction to show that Lemma~\ref{lem:used in proof T_bound1} as well as ~\eqref{equ:cond1} in Condition~\ref{con:4.1} will hold.
For $t = 0$, Lemma~\ref{lem:used in proof T_bound1} hold trivially. Then we use induction with respect to $k$ to prove that
$$
\norm{y_{t,k} - y^*(w_{t,k})}\le \frac{\sigma}{2\tilde L}
$$
holds for any  $k \ge -1$. For $k = -1$, Lemma~\ref{lem:AGD} directly implies 
\begin{align*}
\norm{y_{t,-1} - y^*(w_{t,-1})} \le\frac{\norm{y^*(w_{t,-1)}}}{\hat{C}}\cdot\frac{\sigma}{2\tilde L} \le\frac{\sigma}{2\tilde L}, 
\end{align*}
where the second inequality is based on Lemma~\ref{lem:used in proof T_bound1}.
Suppose it holds that
\begin{align*}
\norm{y_{t,k-1}-y^*(w_{t,k-1})}\le \frac{\sigma}{2\tilde L}    
\end{align*}
for any $k \le k'-1$, then we have
\begin{align*}
&\quad
\norm{y_{t,k'} -y^*(w_{t,k'})}
\\&\le
\sqrt{1+\kappa}\pr{1-\frac{1}{\sqrt{\kappa}}}^{T_{t,k'}/2} \norm{y_{t,k'-1} - y^*(w_{t,k'})}
\\&\le
\sqrt{1+\kappa}\pr{1-\frac{1}{\sqrt{\kappa}}}^{T_{t,k'}/2} (\norm{y_{t,k'-1} - y^*(w_{t,k'-1})} + \norm{y^*(w_{t,k'-1}) - y^*(w_{t,k'})})
\\&\le
\sqrt{1+\kappa}\pr{1-\frac{1}{\sqrt{\kappa}}}^{T_{t,k'}/2} \pr{\frac{\sigma}{2\tilde L} + \kappa \norm{w_{t,k'-1} - w_{t,k'}}}
\\&\le
\sqrt{1+\kappa}\pr{1-\frac{1}{\sqrt{\kappa}}}^{T_{t,k'}/2} \pr{\frac{\sigma}{2\tilde L} + 2\kappa B}
\\&\le
\frac{\sigma}{2\tilde L}
\ ,
\end{align*}
where the first inequality is based on Lemma~\ref{lem:AGD}, the second one uses triangle inequality, the third one is based on induction hypothesis and Lemma~\ref{lem:y*-Lip}, the fourth one uses~\eqref{2d}, and the last step use the definition of $T_{t,k}$. Therefore, \eqref{equ:cond1} in Condition~\ref{con:4.1} can hold.

Suppose Lemma~\ref{lem:used in proof T_bound1} and ~\eqref{equ:cond1} in Condition~\ref{con:4.1} hold for any $t \le t'-1$, then we have shown that when we choose $T_{t,k}'$ as defined in ~\eqref{equ:T_tk'}, then ~\eqref{equ:cond2} 
 in Condition~\ref{con:4.1} can hold. Thus, from Lemma~\ref{lem:bias of gradient} we obtain that:
 \begin{equation}\label{equ:bias}
\| \nabla\Phi(w_{t,k})- \hat{\nabla}\Phi(w_{t,k}) \|_2 \le \sigma
\ .
 \end{equation}
 We claim that for any $t$, we can find some constant $C$ to satisfy:
 \begin{equation}\label{equ:real bound gradient}
 \norm{\Phi(w_{t,\fK-1})}\le C\ .
 \end{equation}
 Otherwise, ~\eqref{equ:gradient bound} in Lemma~\ref{lem:3.2} shows that $\Phi(w_{t,\fK})$ can go to $-\infty$ and contradict with the assumption $\min_{x\in\mathbf{R}^{d_x}} \Phi(x) > -\infty$.
 
 For any epoch $t \le t'-1$, we have
\begin{equation}\label{equ:y* initial bound}
\begin{aligned}
&\quad\norm{x_{t,\fK}-x_{t,0}}\\
&= \norm{x_{t,\fK} - x_{t,\fK-1} + x_{t,\fK-1} - x_{t,0}}\\
&= \norm{(1-\theta)(x_{t,\fK-1} - x_{t,\fK-2}) - \eta \hat\nabla\Phi(w_{t,\fK-1}) + x_{t,\fK-1} - x_{t,0}}\\
&\le \norm{x_{t,\fK-1}-x_{t,\fK-2}} + \norm{x_{t,\fK-1}-x_{t,0}} + \eta\norm{\hat\nabla\Phi(w_{t,\fK-1)}}\\
&\le  2B + \eta\norm{\hat{\nabla}\Phi(w_{t,\fK-1}) - \nabla\Phi(w_{t,\fK-1}) + \nabla\Phi(w_{t,\fK-1})}\\
&\le 2B + \eta\norm{\hat{\nabla}\Phi(w_{t,\fK-1}) - \nabla\Phi(w_{t,\fK-1})} + \eta\norm{\nabla\Phi(w_{t,\fK-1})}\\
&\le 2B + \eta \sigma + \eta\norm{\nabla\Phi(w_{t,\fK-1})}\\
&\le 2B + \eta(\sigma + C)
\end{aligned}
\end{equation}
for some constant $C$. Here we use triangle inequality in the first inequality; ~\eqref{2b} in the second one; triangle inequality again in the third one; ~\eqref{equ:bias} in the fourth one and ~\eqref{equ:real bound gradient} in the last one.

Then for $t'$-th epoch, we have
\begin{align*}
\norm{y^*(w_{t',-1})-y^*(x_{0,0})}&\le\kappa\norm{w_{t',-1}-x_{0,0}}\\
&=\kappa \norm{x_{t',0}-x_{0,0}}\\
&=\kappa \norm{x_{t'-1,\fK} - x_{0,0}}\\
&\le \kappa (\norm{x_{t'-1,0} - x_{0,0}} + \norm{x_{t'-1,\fK} - x_{t'-1,0}})\\
&\le \kappa (\norm{x_{t'-1,0} - x_{1,0}} + (2B + \eta\sigma + \eta C))\\
&\le (2B + \eta\sigma + \eta C)\kappa t\ ,
\end{align*}
where the first inequality is based on the Lipschitz continuous of  $y^*(x)$ shown in Lemma~\ref{lem:y*-Lip}; the second one uses triangle inequality; the third one is based on~\eqref{equ:y* initial bound}, and the last one uses induction.
Then we have
\begin{align*}
\norm{y^*(w_{t',-1})} &\le \norm{y^*(x_{0,0})} + B\kappa t'\\
&\le \norm{y^*(x_{0,0})} + \frac{(2B + \eta\sigma + \eta C)\kappa\Delta \sqrt{\tilde\rho}}{\epsilon^{3/2}}
,
\end{align*}
where we use Lemma~\ref{lem:epoch_num1} in the last inequality. 

Similarly with the case $t=0$, we use induction with respect to $k$ again, we have that \eqref{equ:cond1} in Condition~\ref{con:4.1} hold.

This also finishes the proof for Lemma~\ref{lem:used in proof T_bound1}. 

The total gradient calls from AGD in Algorithm~\ref{alg:AHGD} in one epoch satisfies
\begin{align*}
\sum_{k=-1}^{\fK-1}T_{t,k}
&\le
2\sqrt{\kappa}\left(\frac{T_{-1}}{2\sqrt{\kappa}} + \sum_{k=0}^{\fK-1}\log\pr{\sqrt{\kappa+1} + \frac{4\tilde L\kappa\sqrt{\kappa+1} B}{\sigma}} \right) + \fK +1
\\&=
2\sqrt{\kappa} \left(\frac{T_{-1}}{2\sqrt{\kappa}} + \fK \log\pr{\sqrt{\kappa+1} + \frac{4\tilde L\kappa\sqrt{\kappa+1} B}{\sigma}} \right) + \fK +1
\\&=
2\sqrt{\kappa}\fK \pr{\frac{1}{\fK}\log\pr{\frac{2\tilde L\sqrt{\kappa+1}}{\sigma}\hat{C}} +\log\pr{\sqrt{\kappa+1} + \frac{4\tilde L\kappa\sqrt{\kappa+1} B}{\sigma}} } + \fK + 1
.
\end{align*}

This completes our proof of Proposition~\ref{thm:T_bound1}.
\end{proof}
\pb\subsection{Proof of Corollary~\ref{cor:one-order}}
\begin{proof}
Theorem~\ref{thm:one-oder} says that \texttt{RAHGD} can output an $\epsilon$-FOSP within at most $\fO\pr{\Delta\tilde L^{1/2}\tilde\rho^{1/4}\epsilon^{-7/4}}$ iterations in the outer loop. Then we have
\begin{align*}
Gc(f,\epsilon) = \fO\pr{\dfrac{\Delta\tilde L^{1/2}\tilde 
\rho^{1/4}}{\epsilon^{7/4}}}
\qquad \text{and} \qquad
JV(g,\epsilon) = \fO\pr{\dfrac{\Delta\tilde L^{1/2}\tilde 
\rho^{1/4}}{\epsilon^{7/4}}}
\ .
\end{align*}
Recall that $\tilde L = \fO(\kappa^3)$ and $\tilde\rho = \fO(\kappa^5)$, we have
\begin{align*}
Gc(f,\epsilon) = \fO\pr{\kappa^{11/4}\epsilon^{-7/4}}
\qquad \text{and} \qquad
JV(g,\epsilon) = \fO\pr{\kappa^{11/4}\epsilon^{-7/4}}
\ .
\end{align*}
Gradients of $g(x,\cdot)$ and Hessian-vector products are occurred in AGD and CG respectively. Proposition~\ref{thm:T_bound1} shows that we only require $
\fO\pr{\sqrt{\kappa}\fK\log(\frac{1}{\epsilon})}
$ iterates of AGD and CG in one epoch to have Condition~\ref{con:4.1} hold.
From Lemma~\ref{lem:epoch_num1} we know that \texttt{RAHGD} will terminate in at most $
\Delta \sqrt{\tilde\rho}\epsilon^{-3/2}
$ epochs.
Recall that $
\fK \le K = \frac{1}{2}\big(\tilde L^2/(\tilde\rho\epsilon)\big)^{1/4}
$, we have
\begin{align*}
Gc(g,\epsilon) = \fO\pr{\dfrac{\Delta\tilde L^{1/2}\tilde 
\rho^{1/4}\kappa^{1/2}\log(1/\epsilon)}{\epsilon^{7/4}}}
\qquad \text{and} \qquad
HV(g,\epsilon) = \fO\pr{\dfrac{\Delta\tilde L^{1/2}\tilde 
\rho^{1/4}\kappa^{1/2}\log(1/\epsilon)}{\epsilon^{7/4}}}
\ .
\end{align*}
Hiding polylogarithmic factor and pluging $\tilde L = \fO(\kappa^3)$ and $\tilde\rho = \fO(\kappa^5)$ into it, we have
\begin{align*}
Gc(g,\epsilon) = \tilde\fO\pr{\kappa^{13/4}\epsilon^{-7/4}}
\qquad \text{and} \qquad
HV(g,\epsilon) = \tilde\fO\pr{\kappa^{13/4}\epsilon^{-7/4}}
\ .
\end{align*}
\end{proof}

\pb\section{Proofs for Section~\ref{sec_perturb}}\label{apd:sec_perturb}
In this section, we provide the proofs for theorems in Section~\ref{sec_perturb}. We first show that the number of epochs can be bounded. Then we prove that \texttt{PRAHGD} can output an $(\epsilon,\sqrt{\tilde\rho \epsilon\,})$-SOSP. Finally, we provide the oracle complexity analysis.

\pb\subsection{Proof of Theorem~\ref{thm:second-order}}
\label{sec:proof_sec_order}
We will first provide two lemmas. Lemma~\ref{lem:epoch_num2} shows that the number of epochs is bounded. Lemma~\ref{lem:maxlength} is prepared to show that \texttt{PRAHGD} can escape saddle point with high  probability. Finally we provide the proof of theorem~\ref{thm:second-order}.

\begin{lemma}
\label{lem:epoch_num2}
Consider the setting of Theorem~\ref{thm:second-order}, and we run Algorithm~\ref{alg:AHGD}, then the algorithm will terminate in at most $\fO\pr{\Delta\sqrt{\tilde \rho}\chi^5\epsilon^{-3/2}}$ epochs.
\end{lemma}
\begin{proof}

From the Lipschitz continuity of gradient, we have
\begin{align*}
\Phi(x_{t+1,0}) - \Phi(x_{t,\fK})&\le \inner{\nabla\Phi(x_{t,\fK})}{x_{t+1,0} - x_{t,\fK}} + \frac{\tilde L}{2}\norm{x_{t+1,0}-x_{t,\fK}}^2\\
&= \inner{\nabla\Phi(x_{t,\fK})}{\xi_t} + \frac{\tilde L}{2}\norm{\xi_t}^2\\
&\le \norm{\nabla\Phi(x_{t,\fK})}r + \frac{\tilde Lr^2}{2}\ .
\end{align*}
If $\norm{\nabla\Phi(w_{\fK-1})}> \frac{B}{\eta}$,  then Lemma~\ref{lem:3.2} means when the “if condition” triggers, we have 
\begin{equation}\label{equ:251}
\Phi(x_\fK) - \Phi(x_0) \le - \frac{B^2}{4\eta}+ \sigma B  + \frac{5\eta\sigma^2 K}{8}\ .
\end{equation}
We say that $\norm{\nabla\Phi(x_{t,\fK})}$ is bounded. Otherwise, one gradient descent step $z = x_{t,\fK} - \eta\nabla\Phi(x_{t,\fK})$ leads to
\begin{align*}
\Phi(z)&\le \Phi(x_{t,\fK}) + \inner{\nabla\Phi(x_{t,\fK})}{-\eta \nabla\Phi(x_{t,\fK})} + \frac{\tilde L \eta^2}{2}\norm{\nabla\Phi(x_{t,\fK})}^2\\
&= \Phi(x_{t,\fK}) - \frac{7\eta}{8}\norm{\nabla\Phi(x_{t,\fK})}^2\ ,
\end{align*}
which means $\Phi(z)\sim -\infty $ and contradicts with the assumption $\min_{x\in\BR^{d_x}}\Phi(x) > -\infty$. 
Let $\norm{\nabla\Phi(x_{t,\fK})} \le C$, then we have
\begin{equation}\label{equ:252}
    \begin{aligned}
\Phi(x_{t+1,0}) - \Phi(x_{t,\fK}) &\le Cr + \frac{\tilde Lr^2}{2} \le \frac{B^2}{8\eta}\ ,
\end{aligned}
\end{equation}
where we use the definition of $r$ in the second inequality.
Summing over \eqref{equ:251} and ~\eqref{equ:252}, we obtain
\begin{align*}
\Phi(x_{t+1,0}) - \Phi(x_{t,0}) &\le - \frac{B^2}{8\eta}+ \sigma B  + \frac{5\eta\sigma^2 K}{8}
\le - \frac{B^2}{8\eta}
= -\frac{\epsilon\tilde L}{165888\tilde \rho \chi^4}
\end{align*}
for all epochs. 
On the other hand, if $\norm{\nabla \Phi(w_{\fK-1})}\le \frac{B}{\eta}$, Lemma~\ref{lem:3.5} means
$$\Phi(x_\fK) - \Phi(x_0) \le -\frac{3\theta B^2}{8\eta K} +  \frac{5\eta K}{2\theta}(2\tilde{\rho}B^2+\sigma)^2 + \frac{\tilde\rho}{6} (3B + \eta\sigma)^3 + \frac{\eta K}{2\theta}\sigma^2.$$
We also have
\begin{align*}
\norm{\nabla\Phi(x_\fK)} &\le \norm{\nabla\Phi(w_{\fK-1})} + \norm{\nabla\Phi(x_\fK) - \nabla\Phi(w_{\fK-1})}\\
&\le \norm{\nabla\Phi(w_{\fK-1})} + \tilde L\norm{x_\fK - w_{\fK-1}}\\
&\le \norm{\nabla\Phi(w_{\fK-1})} + \tilde L\eta\pr{\norm{\nabla\Phi(w_{\fK-1})} + \norm{\hat{\nabla}\Phi(w_{\fK-1}) - \nabla\Phi(w_{\fK-1})}}\\
&\le \frac{B}{\eta} + \tilde LB + \frac{\sigma}{4}= \frac{5B}{4\eta} + \frac{\sigma}{4}\ .
\end{align*}
So we obatin
$$
\Phi(x_{t+1,0}) - \Phi(x_{t,\fK}) \le \frac{5Br}{4\eta} + \frac{\sigma r}{4} + \frac{\tilde Lr^2}{2}\le \frac{\theta B^2}{8\eta K}+  \frac{\sigma B^2}{4}
,
$$
and
\begin{align*}
\Phi(x_{t+1,0}) - \Phi(x_{t,0}) &\le  -\frac{\theta B^2}{4\eta K} +  \frac{5\eta K}{2\theta}(2\tilde{\rho}B^2+\sigma)^2 + \frac{\tilde\rho}{6} (3B + \eta\sigma)^3+ \frac{\eta K}{2\theta}\sigma^2 + \frac{\sigma B^2}{4} \\
&\le -\frac{\epsilon^{1.5}}{663552\sqrt{\tilde \rho}\chi^5}\ .
\end{align*}
Hence, the algorithm will terminate in at most $\fO\pr{{\Delta\sqrt{\tilde \rho}\chi^5}{\epsilon^{-3/2}}}$ epochs. 
\end{proof}
Before proving that \texttt{PRAHGD} can output an $(\epsilon,\sqrt{\tilde\rho\epsilon}\,)$-SOSP,  we first show the following lemma.
\begin{lemma}
\label{lem:maxlength}
Following the setting of Theorem~\ref{thm:second-order}, we additionally suppose that $\lambda_{\min}(\mH) < -\sqrt{\epsilon\tilde\rho}$, where $\mH = \nabla^2\Phi(x)$ for given $x\in\BR^{d_x}$. 
We suppose points $x'_0,x_0''\in\BR^{d_x}$ satisfy $\norm{x'_0-x}\le r$,  $\norm{x''_0-x}\le r$ and $x_0'-x_0'' = r_0e_1$, where $e_1$ is the minimum eigen-direction of $\mH$ and $r_0 = \frac{\zeta r}{\sqrt{d_{x}}}$.  
Running 
\emph{\texttt{PRAHGD}} in Algorithm~\ref{alg:AHGD} with initialization $x_{0,0}=x_0'$ and $x_{0,0}=x_0''$, respectively, then at least one of these two initial points leads to its iterations trigger the “if condition”.
\end{lemma}
\begin{proof}
Recall that the update rule of \texttt{PRAHGD} can be written as:
$$x_{k+1} = (2-\theta)x_k-(1-\theta)x_{k-1} - \eta \hat{\nabla}\Phi((2-\theta)x_k-(1-\theta)x_{k-1}).$$
We denote $z_k = x_k'-x_k''$, then 
\begin{align*}
z_{k+1} &= (2-\theta)z_k - (1-\theta)z_{k-1}-\eta(\hat{\nabla}\Phi(w_k')-\hat{\nabla}\Phi(w_k'')))\\
&= (2-\theta)(\mI - \eta \mH - \eta {\bf\Omega}_k)z_k - (1-\theta)(\mI-\eta \mH-\eta {\bf\Omega}_k)z_{k-1}-\eta(\varsigma_k - \varsigma_k') 
,
\end{align*}
where 
\begin{align*}
\small\begin{split}    
{\bf\Omega}_k = \int_0^1(\nabla^2\Phi(tw_k + (1-t)w_k')-K)\,{\rm d} t,~~ \varsigma_k'= \nabla\Phi(w_k') - \hat{\nabla}\Phi(w_k') ~~\text{and}~~\varsigma_k'' =\nabla\Phi(w_k'') - \hat{\nabla}\Phi(w_k'').     
\end{split}
\end{align*}
In the last step, we use
$$
\nabla\Phi(w_k') - \nabla\Phi(w_k'') = (\mH+{\bf\Omega}_k)(w_k'-w_k'') = (\mH+{\bf\Omega}_k)((2-\theta)z_k - (1-\theta)z_{t-1})\ .
$$
We thus get the update of $z_k$ in matrix form as follows
\begin{align*}
\begin{pmatrix}
z_{k+1}\\
z_k
\end{pmatrix} 
&= 
\begin{pmatrix}
(2-\theta)(\mI-\eta \mH)&-(1- \theta)(\mI-\eta \mH)\\
\mI& 0
\end{pmatrix}
\begin{pmatrix}
w_k\\
w_{k-1}
\end{pmatrix} \\
& \qquad + \eta 
\begin{pmatrix}
(2-\theta){\bf\Omega}_kz_k - (1-\theta){\bf\Omega}_kz_{k-1} +\varsigma_k' - \varsigma_k''\\
0
\end{pmatrix}\\
&=\mA 
\begin{pmatrix}
z_k\\
z_{k-1}
\end{pmatrix}
-\eta
\begin{pmatrix}
\omega_k\\
0
\end{pmatrix}
=\mA^{k+1}
\begin{pmatrix}
z_0\\
z_{-1}
\end{pmatrix}
-\eta\sum_{i=0}^k
\mA^{k-i}
\begin{pmatrix}
\omega_i\\
0
\end{pmatrix}\ ,
\end{align*}
where $\omega_k = (2-\theta){\bf\Omega}_kz_k - (1-\theta){\bf\Omega}_kz_{k-1} +\varsigma_k' - \varsigma_k''  $. 
Then we have
$$
z_k = 
\begin{pmatrix}
\mI & 0
\end{pmatrix}
\mA^{k}
\begin{pmatrix}
z_0\\
z_{0}
\end{pmatrix}
-\eta
\begin{pmatrix}
\mI & 0
\end{pmatrix}
\sum_{i=0}^{k-1}
\mA^{k-i-1}
\begin{pmatrix}
\omega_i\\
0
\end{pmatrix}\ .
$$
Assume that none of the iteration on $(x_0',x_1',\dots,x_K')$ and $(x_0'',x_1'',\dots,x_K'')$ trigger the “if condition”, then we have
\begin{equation}\label{equ:B.1}
\begin{aligned}
& \norm{x_k'-x_0'} \le B,\ \norm{w_k'-x_0'}\le 2B,\ \forall k \le K, \\
& \norm{x_k''-x_0''} \le B,\ \norm{w_k''-x_0''}\le 2B,\ \forall k \le K.
\end{aligned}
\end{equation}
Then we achieve
\begin{align*}
\norm{{\bf\Omega}_k} &\le \tilde\rho \max\pr{\norm{w_k'-x},\norm{w_k''-x}}\\
&\le \tilde\rho \max\pr{\norm{w_k'-x_0'},\norm{w_k''-x_0''}} + \tilde \rho r \le 3\tilde\rho B
\end{align*}
and
\begin{align*}
\norm{\omega_k} &\le 6\tilde \rho B \pr{\norm{z_k} + \norm{z_{k-1}}} + \norm{\varsigma_k' - \varsigma_k''}\\
&\le 6\tilde \rho B \pr{\norm{z_k} + \norm{z_{k-1}}} + 2\sigma,
\end{align*}
where we use Lemma~\ref{lem:bias of gradient} in the last step. We can show the following inequality for all $k\le K$ by induction:
$$
\norm{
\eta
\begin{pmatrix}
\mI &0
\end{pmatrix}
\sum_{i=0}^{k-1}\mA^{k-1-i}
\begin{pmatrix}
\omega_i\\
0
\end{pmatrix}
} 
\le
\frac{1}{2}
\norm{
\begin{pmatrix}
\mI &0
\end{pmatrix}
\mA^{k}
\begin{pmatrix}
z_0\\
z_0
\end{pmatrix}
}\ .
$$
It is easy to check the base case holds for $k=0$. Assume the inequality holds for all steps equal to or less than $k$. Then we have
\begin{align*}
\norm{z_k}\le \frac{3}{2} \norm{
\begin{pmatrix}
\mI &0
\end{pmatrix}
\mA^{k}
\begin{pmatrix}
z_0\\
z_0
\end{pmatrix}
}
\qquad\text{and}\qquad
\norm{\omega_k} \le 18\tilde \rho B 
\norm{
\begin{pmatrix}
\mI &0
\end{pmatrix}
\mA^{k}
\begin{pmatrix}
z_0\\
z_0
\end{pmatrix}
} + 2\sigma ,
\end{align*}
where we use the monotonicity of $\norm{\begin{pmatrix}
\mI &0 
\end{pmatrix}
\mA^{k}
\begin{pmatrix}
z_0\\
z_0
\end{pmatrix}} $ in $k$ (Lemma 38 in ~\cite{jin2018accelerated}) in the last inequality.

We define 
\begin{align*}   
\begin{pmatrix}
    a_k & -b_k
\end{pmatrix}
=
\begin{pmatrix}
1 & 0
\end{pmatrix}
\mA_{\min}^k
\quad\text{and}\quad
\mA_{\min}=
\begin{pmatrix}
(2-\theta)(1-\eta\lambda_{\min}) & -(1-\theta)(1-\eta\lambda_{\min})\\
1&0
\end{pmatrix},
\end{align*}
then
\begin{align*}
&\quad \norm{
\eta
\begin{pmatrix}
\mI &0
\end{pmatrix}
\sum_{i=0}^{k}\mA^{k-i}
\begin{pmatrix}
\omega_i\\
0
\end{pmatrix}
} \\
&\le 
\eta\sum_{i=0}^k
\norm{
\begin{pmatrix}
\mI &0
\end{pmatrix}
\sum_{i=0}^{k}\mA^{k-i}
\begin{pmatrix}
\mI\\
0
\end{pmatrix}
} 
\norm{\omega_i}\\
&\le \eta\sum_{i=0}^k
\norm{
\begin{pmatrix}
\mI &0
\end{pmatrix}
\sum_{i=0}^{k}\mA^{k-i}
\begin{pmatrix}
\mI\\
0
\end{pmatrix}
}  \pr{
18\tilde \rho B 
\norm{
\begin{pmatrix}
\mI &0
\end{pmatrix}
\mA^{i}
\begin{pmatrix}
z_0\\
z_0
\end{pmatrix}
} + 2\sigma
}\\
&\overset{({\rm a})}{=} \eta \sum_{i=0}^k |a_{k-i}|\ 
(\ 18\tilde\rho Br_0 |a_i - b_i| + 2\sigma)
\\
&\overset{({\rm b})}{\le}
\eta \sum_{i=0}^k |a_{k-i}|\ 
(\ 20\tilde\rho Br_0 |a_i - b_i| )
\\&
\overset{({\rm c})}{\le}
20\tilde\rho B\eta\sum_{i=0}^{k}\pr{\frac{2}{\theta}+k+1} |a_{k+1}-b_{k+1}|r_0\\
&\le 20\tilde\rho B\eta K\pr{\frac{2}{\theta}+K} \norm{
\begin{pmatrix}
\mI &0
\end{pmatrix}
\mA^{k+1}
\begin{pmatrix}
z_0\\
z_0
\end{pmatrix}}\ ,
\end{align*}
where 
the step $\overset{({\rm a})}{=}$ uses the fact that $z_0=r_0e_1$ is along the minimum eigenvector direction of $\mH$;
the step $\overset{b}{\le}$ is based on the fact that $\sigma  \le\tilde\rho Br_0 |a_i - b_i|$;
the step $\overset{c}{\le}$ uses Lemma 36 in \citep{jin2018accelerated}. 
From Lemma 38 in \citep{jin2018accelerated}, we have
\begin{align*}
|a_i-b_i| \ge \frac{\theta}{2}\pr{1 + \frac{\theta}{2}}^i 
\ge \frac{\theta}{2},
\end{align*}
 and thus $\tilde\rho Br_0 |a_i - b_i| \ge \dfrac{\tilde \rho B \zeta r \theta}{2\sqrt{d_x}} \ge \sigma$.
From the parameter settings, we have 
\begin{align*}
20\tilde\rho B\eta K\pr{\frac{2}{\theta}+K} \le \frac{1}{2}. 
\end{align*}
Therefore, we complete the induction, which yields
\begin{align*}
\norm{z_K}&\ge 
\norm{
\begin{pmatrix}
\mI & 0
\end{pmatrix}
\mA^{K}
\begin{pmatrix}
z_0\\
z_{0}
\end{pmatrix}}
-\norm{\eta
\begin{pmatrix}
\mI & 0
\end{pmatrix}
\sum_{i=0}^{K-1}
\mA^{K-i-1}
\begin{pmatrix}
\omega_i\\
0
\end{pmatrix}}\\
&\ge \frac{1}{2}\norm{
\begin{pmatrix}
\mI & 0
\end{pmatrix}
\mA^{K}
\begin{pmatrix}
z_0\\
z_{0}
\end{pmatrix}} = \frac{r_0}{2}|a_K-b_K|\\
&\ge \frac{\theta r_0}{4}\pr{1+\frac{\theta}{2}}^K \ge 5B,
\end{align*}
where we use Lemma 38 in \citep{jin2018accelerated} 
and $\eta\lambda_{\min}\le -\theta^2 $ in the third inequality and the last step comes from $K = \frac{2}{\theta}\log(\frac{20B}{\theta r_0})$. 
However, from \eqref{equ:B.1} we can obtain:
\begin{align*}
\norm{z_K} &\le \norm{x_K' - x_0'} + \norm{x_K''-x_0''} + \norm{x_0'-x_0''} \le 2B + 2r\le 4B,
\end{align*}
which leads to contradiction. 
Thus we conclude that at least one of the iteration triggers the “if condition” and we finish the proof.
\end{proof}

Having established the necessary groundwork, we are now prepared to present the proof of Theorem~\ref{thm:second-order}.
\begin{proof}
From Lemma~\ref{lem:epoch_num2}, we know that Algorithm~\ref{alg:AHGD} will terminate in at most $\fO\pr{{\Delta\sqrt{\tilde \rho}\chi^5}{\epsilon^{-3/2}}}$ epochs. Since each epoch needs at most $K = \fO
\pr{\chi\big({\tilde L}^2/(\epsilon\tilde \rho)\big)^{1/4}}$ gradient evaluations,
the total number of gradient
evaluations must be less than
$\fO\pr{{\Delta\tilde L^{1/2}\tilde 
\rho^{1/4}\chi^6}{\epsilon^{-1.75}}}$.\\
Now we consider the last epoch.
Following a similar methodology employed in the proof of Theorem~\ref{thm:one-oder}, 
we also have
$$\norm{\nabla\Phi(\hat{w})}\le 
\frac{2\sqrt{2}B}{\eta K^2} + 
\frac{2\theta B}{\eta K} + 4\tilde\rho B^2 +
\sigma \le \frac{\epsilon}{\chi^3} + \epsilon^2\le \epsilon\ .$$
\\
If $\lambda_{\min}(\nabla^2\Phi(x_{t,\fK})) \ge -\sqrt{\epsilon\tilde\rho }$, then from the  perturbation theory of eigenvalues of \citet{bhatia1997matrix}, we have
\begin{align*}
|\lambda_j(\nabla^2\Phi(\hat{w}_{t+1})) - \lambda_j(\nabla^2\Phi(x_{t,\fK}))|&\le \norm{\nabla^2\Phi(\hat{w}_{t+1})-\nabla^2\Phi(x_{t,\fK})}\\
&\le \tilde \rho\norm{\hat{w}_{t+1}-x_{t,\fK}}\\
&\le \tilde\rho\norm{\hat{w}_{t+1}-x_{t+1,0}} + \tilde \rho r\\
&\le 3\tilde\rho B
\end{align*}
for any $j$, where we use $\norm{\hat{w}_{t+1} - x_{t+1,0}}\le \frac{1}{K_0+1} \sum_{k=0}^{K_0}\norm{w_{t+1,k}-x_{t+1,0}}\le 2B$ in the last inequality. Then we have
\begin{align*}
\lambda_j(\nabla^2\Phi(\hat{w}_{t+1})) &\ge  \lambda_j(\nabla^2\Phi(x_{t,\fK})) - |\lambda_j(\nabla^2\Phi(\hat{w}_{t+1})) - \lambda_j(\nabla^2\Phi(x_{t,\fK}))| \\
&\ge -\sqrt{\epsilon\tilde\rho} - 3\tilde\rho B \ge -1.011\sqrt{\epsilon\tilde\rho}
.
\end{align*}
Now we consider $\lambda_{\min}(\nabla^2\Phi(x_{t,\fK})) < -\sqrt{\epsilon\tilde\rho }$. Define the stuck region in $\sB(r)$ centered at $x_{t,\fK}$ to be the set of points starting from which the “if condition” does not trigger in $K$ iterations, that is, the algorithm terminates and outputs a saddle point.  From Lemma~\ref{lem:maxlength}, we know that the length along the minimum eigen-direction of $\nabla^2\Phi(x_{t,\fK})$ is less than $r_0$. Therefore, the probability of the starting point $x_{t+1,0} = x_{t,\fK} + \xi_t$ located in the stuck region is less than 
$$\frac{r_0V_{d-1}(r)}{V_d(r)} \le \frac{r_0\sqrt{d}}{r}\le \zeta ,$$
where we let $r_0 = \frac{\zeta r}{\sqrt{d}}$.  Thus, the output $\hat{w}$ satisfies $\lambda_{\min}(\nabla^2\Phi(\hat{w}) \ge -1.011\sqrt{\epsilon\tilde\rho}$ with probability at least $1-\zeta$.
This completes our whole proof of Theorem~\ref{thm:second-order}.
\end{proof}
\pb\subsection{Proof of Proposition~\ref{thm:T_bound2}}
\label{sec:proof T_bound2}
The proof of Proposition~\ref{thm:T_bound2} is similar to that of Proposition~\ref{thm:T_bound1}.
We provide the proof for  Proposition~\ref{thm:T_bound2} as follows.
\begin{proof}
We first consider the iterations of CG in Algorithm~\ref{alg:AHGD} in one epoch. We choose $T_{t,k}'$ as 
\begin{equation}\label{equ:T_tk'2}
T_{t,k}' = \left\{
\begin{array}{ll}
\ceil{\frac{\sqrt{\kappa}+1}{2}\log\pr{\frac{4\ell\sqrt{\kappa}}{\sigma}\pr{\norm{v_{0,-1}} +  \frac{M}{\mu}}}}
,&k=0,
\\[0.3cm]
\ceil{\frac{\sqrt{\kappa}+1}{2}\log\pr{\frac{4\ell\sqrt{\kappa}}{\sigma}\pr{\frac{\sigma}{2\ell} + \frac{2M}{\mu}}}}
,& k \ge 1.
\end{array}
\right.
\end{equation}
Following the proof of that in Section~\ref{sec:proof_T_bound1} in exact fashions we arrive at that ~\eqref{equ:cond2} in Condition~\ref{con:4.1} can hold.

The total iterates of CG when running \texttt{PRAHGD} in Algorithm~\ref{alg:AHGD} in one epoch satisfies
\begin{align*}
\small\begin{split}    
&\quad \sum_{k=0}^{\fK-1}T_k' \\
&\le
\fK + \frac{\sqrt{\kappa}+1}{2}\left(\frac{2T_0'}{\sqrt{\kappa}+1} + \sum_{k=1}^{\fK-1} \log \pr{\frac{4\ell\sqrt{\kappa}}{\sigma}\pr{\frac{\sigma}{2\ell} + \frac{2M}{\mu}}} \right)
\\&=
\fK + \frac{\sqrt{\kappa}+1}{2}\left(\frac{2T_0'}{\sqrt{\kappa}+1} + \pr{\fK-1}\log \pr{\frac{4\ell\sqrt{\kappa}}{\sigma}\pr{\frac{\sigma}{2\ell} + \frac{2M}{\mu}}} \right)
\\&=
\fK + \frac{\sqrt{\kappa}+1}{2}\fK \left(\frac{1}{\fK}\log\pr{\frac{4\ell\sqrt{\kappa}}{\sigma}\pr{\norm{v_{0,-1}} + \frac{M}{\mu}}}
+ \pr{1-\frac{1}{\fK}} \log\pr{\frac{4\ell\sqrt{\kappa}}{\sigma}\pr{\frac{\sigma}{2\ell} + \frac{2M}{\mu}}}\right)
\ .
\end{split}
\end{align*}

Now we consider the iterations of AGD \texttt{PRAHGD} in Algorithm~\ref{alg:AHGD} in one epoch.

We first show the following lemma.
\begin{lemma}\label{lem:used in proof T_bound2}
Consider the setting of Theorem~\ref{thm:second-order}, and we run \emph{\texttt{PRAHGD}} in Algorithm~\ref{alg:AHGD}, then we have
$$\norm{y^*(w_{t,-1})} \le \tilde C$$
for any $t>0$, where $\tilde C = \norm{y^*(x_{0,0})} + {(2B +B^2+ \eta\sigma + \eta C)\kappa\Delta \sqrt{\tilde\rho}}{\epsilon^{-3/2}}  .$
\end{lemma}
Then  we choose $T_{t,k}$ as
\begin{equation}\label{equ:T_tk2}
T_{t,k} = \left\{
\begin{array}{ll}
\ceil{2\sqrt{\kappa}\log\pr{\frac{2\tilde L\sqrt{\kappa+1}}{\sigma}\tilde C}}
,&
k=-1
\\[0.3cm]
\ceil{2\sqrt{\kappa}\log\pr{\frac{2\tilde L\sqrt{\kappa+1}}{\sigma}\pr{\frac{\sigma}{2\tilde L}+2\kappa B}}}
,&
k \ge 0
\end{array}
\right.
\end{equation}
We will use induction to show that Lemma~\ref{lem:used in proof T_bound2} as well as ~\eqref{equ:cond1} in Condition~\ref{con:4.1} will hold.

For $t = 0$, Lemma~\ref{lem:used in proof T_bound2} hold trivially. Then we use induction with respect to $k$ to prove that
$$
\norm{y_{t,k} - y^*(w_{t,k})}\le \frac{\sigma}{2\tilde L}
$$
holds for any  $k \ge -1$. For $k = -1$, Lemma~\ref{lem:AGD} directly implies 
\begin{align*}
\norm{y_{t,-1} - y^*(w_{t,-1})} \le\frac{\norm{y^*(w_{t,-1)}}}{\tilde{C}}\cdot\frac{\sigma}{2\tilde L} \le\frac{\sigma}{2\tilde L}, 
\end{align*}
where the second inequality is based on Lemma~\ref{lem:used in proof T_bound2}.
Suppose it holds that $\norm{y_{t,k-1}-y^*(w_{t,k-1})}\le \frac{\sigma}{2\tilde L}$ for any $k \le k'-1$, then we have
\begin{align*}
&\quad
\norm{y_{t,k'} -y^*(w_{t,k'})}
\\&\le
\sqrt{1+\kappa}\pr{1-\frac{1}{\sqrt{\kappa}}}^{T_{t,k'}/2} \norm{y_{t,k'-1} - y^*(w_{t,k'})}
\\&\le
\sqrt{1+\kappa}\pr{1-\frac{1}{\sqrt{\kappa}}}^{T_{t,k'}/2} (\norm{y_{t,k'-1} - y^*(w_{t,k'-1})} + \norm{y^*(w_{t,k'-1}) - y^*(w_{t,k'})})
\\&\le
\sqrt{1+\kappa}\pr{1-\frac{1}{\sqrt{\kappa}}}^{T_{t,k'}/2} \pr{\frac{\sigma}{2\tilde L} + \kappa \norm{w_{t,k'-1} - w_{t,k'}}}
\\&\le
\sqrt{1+\kappa}\pr{1-\frac{1}{\sqrt{\kappa}}}^{T_{t,k'}/2} \pr{\frac{\sigma}{2\tilde L} + 2\kappa B}
\\&\le
\frac{\sigma}{2\tilde L}
\ ,
\end{align*}
where the first inequality is based on Lemma~\ref{lem:AGD}, the second one uses triangle inequality, the third one is based on induction hypothesis and Lemma~\ref{lem:y*-Lip}, the fourth one uses~\eqref{2d}, and the last step use the definition of $T_{t,k}$. Therefore, \eqref{equ:cond1} in Condition~\ref{con:4.1} can hold.

Suppose Lemma~\ref{lem:used in proof T_bound2} and ~\eqref{equ:cond1} in Condition~\ref{con:4.1} hold for any $t \le t'-1$, then we have shown that when we choose $T_{t,k}'$ as defined in ~\eqref{equ:T_tk'2}, then ~\eqref{equ:cond2} 
 in Condition~\ref{con:4.1} can hold. Thus, from Lemma~\ref{lem:bias of gradient} we obtain that:
 \begin{equation}\label{equ:bias2}
\| \nabla\Phi(w_{t,k})- \hat{\nabla}\Phi(w_{t,k}) \|_2 \le \sigma
\ .
 \end{equation}

 For any epoch $t \le t'-1$, we have
\begin{equation}\label{equ:y* initial bound2}
\begin{aligned}
&\quad\norm{x_{t,\fK}-x_{t,0}}\\
&= \norm{x_{t,\fK} - x_{t,\fK-1} + x_{t,\fK-1} - x_{t,0}}\\
&= \norm{(1-\theta)(x_{t,\fK-1} - x_{t,\fK-2}) - \eta \hat\nabla\Phi(w_{t,\fK-1}) + x_{t,\fK-1} - x_{t,0}}\\
&\le \norm{x_{t,\fK-1}-x_{t,\fK-2}} + \norm{x_{t,\fK-1}-x_{t,0}} + \eta\norm{\hat\nabla\Phi(w_{t,\fK-1)}}\\
&\le  2B + \eta\norm{\hat{\nabla}\Phi(w_{t,\fK-1}) - \nabla\Phi(w_{t,\fK-1}) + \nabla\Phi(w_{t,\fK-1})}\\
&\le 2B + \eta\norm{\hat{\nabla}\Phi(w_{t,\fK-1}) - \nabla\Phi(w_{t,\fK-1})} + \eta\norm{\nabla\Phi(w_{t,\fK-1})}\\
&\le 2B + \eta \sigma + \eta\norm{\nabla\Phi(w_{t,\fK-1})}\\
&\le 2B + \eta(\sigma + C)
\end{aligned}
\end{equation}
for some constant $C$. Here we use triangle inequality in the first inequality; ~\eqref{2b} in the second one; triangle inequality again in the third one; ~\eqref{equ:bias2} in the fourth one and ~\eqref{equ:real bound gradient} in the last one.

Then for $t'$-th epoch, we have
\begin{align*}
&\quad \norm{y^*(w_{t',-1})-y^* (x_{0,0})} \\
&\le\kappa\norm{w_{t',-1}-x_{0,0}}\\
&=\kappa \norm{x_{t',0}-x_{0,0}}\\
&=\kappa \norm{x_{t'-1,\fK} - x_{0,0}}\\
&\le \kappa (\norm{x_{t'-1,0} - x_{0,0}} + \norm{x_{t'-1,\fK} - x_{t'-1,0}} + r)\\
&\le \kappa (\norm{x_{t'-1,0} - x_{1,0}} + (2B +B^2+ \eta\sigma + \eta C))\\
&\le (2B +B^2+ \eta\sigma + \eta C)\kappa t\ ,
\end{align*}
where the first inequality is based on the Lipschitz continuous of  $y^*(x)$ shown in Lemma~\ref{lem:y*-Lip}; the second one uses triangle inequality; the third one is based on~\eqref{equ:y* initial bound2}, and the last one uses induction.
Then we have
\begin{align*}
\norm{y^*(w_{t',-1})} \le \norm{y^*(x_{0,0})} + B\kappa t'
\le \norm{y^*(x_{0,0})} + \frac{(2B+B^2 + \eta\sigma + \eta C)\kappa\Delta \sqrt{\tilde\rho}}{\epsilon^{3/2}},
\end{align*}
where we use Lemma~\ref{lem:epoch_num2} in the last inequality. 

Similarly with the case $t=0$, we use induction with respect to $k$ again,  and then we can prove that \eqref{equ:cond1} in Condition~\ref{con:4.1} hold.

This also completes the proof for Lemma~\ref{lem:used in proof T_bound2}. 

The total gradient calls from AGD in Algorithm~\ref{alg:AHGD} in one epoch satisfies
\begin{align*}
\sum_{k=-1}^{\fK-1}T_{t,k}
&\le
2\sqrt{\kappa}\left(\frac{T_{-1}}{2\sqrt{\kappa}} + \sum_{k=0}^{\fK-1}\log\pr{\sqrt{\kappa+1} + \frac{4\tilde L\kappa\sqrt{\kappa+1} B}{\sigma}} \right) + \fK +1
\\&=
2\sqrt{\kappa} \left(\frac{T_{-1}}{2\sqrt{\kappa}} + \fK \log\pr{\sqrt{\kappa+1} + \frac{4\tilde L\kappa\sqrt{\kappa+1} B}{\sigma}} \right) + \fK +1
\\&=
2\sqrt{\kappa}\fK \pr{\frac{1}{\fK}\log\pr{\frac{2\tilde L\sqrt{\kappa+1}}{\sigma}\tilde{C}} +\log\pr{\sqrt{\kappa+1} + \frac{4\tilde L\kappa\sqrt{\kappa+1} B}{\sigma}} } + \fK + 1
.
\end{align*}

This finishes our whole proof for Proposition~\ref{thm:T_bound2}.
\end{proof}

\pb\subsection{Proof of Corollary~\ref{cor:second-order}}
\begin{proof}
From Theorem~\ref{thm:second-order}, we have that \texttt{PRAHGD} in Algorithm~\ref{alg:AHGD} can find an $(\epsilon,\sqrt{\tilde\rho\epsilon}\,)$ SOSP within at most $\fO\pr{{\Delta\tilde L^{1/2}\tilde 
\rho^{1/4}\chi^6}{\epsilon^{-7/4}}}$ iterations in the outer loop. Then we have
\begin{align*}
Gc(f,\epsilon) = \fO\pr{\dfrac{\Delta\tilde L^{1/2}\tilde 
\rho^{1/4}\chi^6}{\epsilon^{7/4}}}
\qquad\text{and}\qquad
JV(g,\epsilon) = \fO\pr{\dfrac{\Delta\tilde L^{1/2}\tilde 
\rho^{1/4}\chi^6}{\epsilon^{7/4}}}\ .
\end{align*}
Omitting polylogarithmic factor and pluging $\tilde L = \fO(\kappa^3)$ and $\tilde\rho = \fO(\kappa^5)$ into it, we have
\begin{align*}
Gc(f,\epsilon) = \tilde\fO\pr{\kappa^{11/4}\epsilon^{-7/4}}
\qquad\text{and}\qquad
JV(g,\epsilon) = \tilde\fO\pr{\kappa^{11/4}\epsilon^{-7/4}}\ .
\end{align*}
Lemma~\ref{lem:epoch_num2} shows that \texttt{PRAHGD} in Algorithm~\ref{alg:AHGD} will terminate in at most $\fO\pr{\frac{\Delta\sqrt{\tilde \rho}\chi^5}{\epsilon^{3/2}}}$ epochs.
From Proposition~\ref{thm:T_bound2} we can obtain that for each epoch $t$, we have
the inner loops $$
\sum_{k=-1}^{\fK-1}T_{t,k}
\le
\fO\left( \kappa^{1/2}\fK\log(1/\epsilon) \right)
\qquad\text{and}\qquad
\sum_{k=0}^{\fK-1}T_{t,k}'
\le
\fO\left( \kappa^{1/2}\fK\log(1/\epsilon) \right)
$$
hold.
Then we have
\begin{align*}
Gc(g,\epsilon) = \fO\pr{\dfrac{\Delta\tilde L^{1/2}\tilde 
\rho^{1/4}\kappa^{1/2}\chi^6\log(1/\epsilon)}{\epsilon^{7/4}}}
~\text{and}~ 
HV(g,\epsilon) = \fO\pr{\dfrac{\Delta\tilde L^{1/2}\tilde 
\rho^{1/4}\kappa^{1/2}\chi^6\log(1/\epsilon)}{\epsilon^{7/4}}}\ ,
\end{align*}
where we use $\fK \le K = \fO
\big(\chi\big({\tilde L}^2/(\epsilon\tilde \rho)\big)^{1/4}\big)$.
Omit polylogarithmic factor and plug $\tilde L = \fO(\kappa^3)$ and~$\tilde\rho = \fO(\kappa^5)$ into it, we have
\begin{align*}
Gc(g,\epsilon) = \tilde\fO\pr{\kappa^{13/4}\epsilon^{-7/4}}
\qquad\text{and}\qquad
HV(g,\epsilon) = \tilde\fO\pr{\kappa^{13/4}\epsilon^{-7/4}}\ .
\end{align*}
\end{proof}
This completes our proof of Corollary~\ref{cor:second-order}.

\pb\section{Proofs for Section~\ref{sec_indications}}
In this section, we provide the proof of Theorem~\ref{thm:minimax}.
\begin{proof}
Lemma~\ref{lem:minimax} shows that in minimax problem settings, $\tilde L = (\kappa +1)\ell$ and $\tilde\rho = 4\sqrt{2}\kappa^3\rho$.
Recall that our \texttt{PRAGDA} evolves directly from \texttt{PRAHGD} ---- removing the CG step in \texttt{PRAHGD} because we do not need to compute the Hessian-vector products when solving the minmax problem.
Therefore, we can straightforwardly apply the theoretical  results for \texttt{PRAHGD}.

Applying Theorem~\ref{thm:second-order}, we have that Algorithm~\ref{alg:PRAGDA} can find an $\big(\epsilon,\fO(\kappa^{1.5}\sqrt{\epsilon}\,)\big)$-SOSP. 

Now we provide the gradient oracle calls complexities. From Lemma~\ref{lem:epoch_num2}, we know that Algorithm~\ref{alg:PRAGDA} will terminate in at most $\fO\big(\Delta\sqrt{\tilde \rho}\chi^5\epsilon^{-3/2}\big)$ epochs.
Proposition~\ref{thm:T_bound2} shows that, for each $t$, the total iteration number of AGD step satisfies:
\begin{align*}
\sum_{k=-1}^{\fK-1} T_{t,k}\le \fO(\kappa^{1/2}\fK\log(1/\epsilon))\ .
\end{align*}
Recall that $\fK\le K=\fO
\pr{\chi\big({\tilde L}^2/(\epsilon\tilde \rho)\big)^{1/4}}$, we have that the total gradient oracle calls is at most:
\begin{align*}
\fO\pr{\frac{\Delta \tilde\rho^{1/4}\tilde L^{1/2} \kappa^{1/2}\chi^6\log(1/\epsilon)}{\epsilon^{7/4}}}\ . 
\end{align*}
Hide polylogarithmic factor and plug $\tilde L$ and $\tilde\rho$ into it, we have the total gradient oracle calls within at most  $\tilde\fO(\kappa^{7/4}\epsilon^{-7/4})$.
\end{proof}

\pb\section{Empirical Studies}\label{sec:experiment}
We conducted a series of experiments to validate the theoretical contributions presented in this paper.
Specifically, we evaluated the effectiveness of our proposed algorithms, \texttt{RAHGD} and \texttt{PRAHGD}, by applying them to two different tasks: data hyper-cleaning for the MNIST dataset and hyperparameter optimization of logistic regression for the 20 News Group dataset.
Our experiments demonstrate that our algorithms outperform several established baseline algorithms, such as BA, AID-BiO, ITD-BiO, and PAID-BiO, with much faster convergence rates.
Additionally, we conducted a synthetic minimax problem experiment, wherein our \texttt{PRAGDA} algorithm exhibits a faster convergence rate when compared to iMCN.

\begin{figure}[!tb]
    \centering
     \includegraphics[width=8cm]{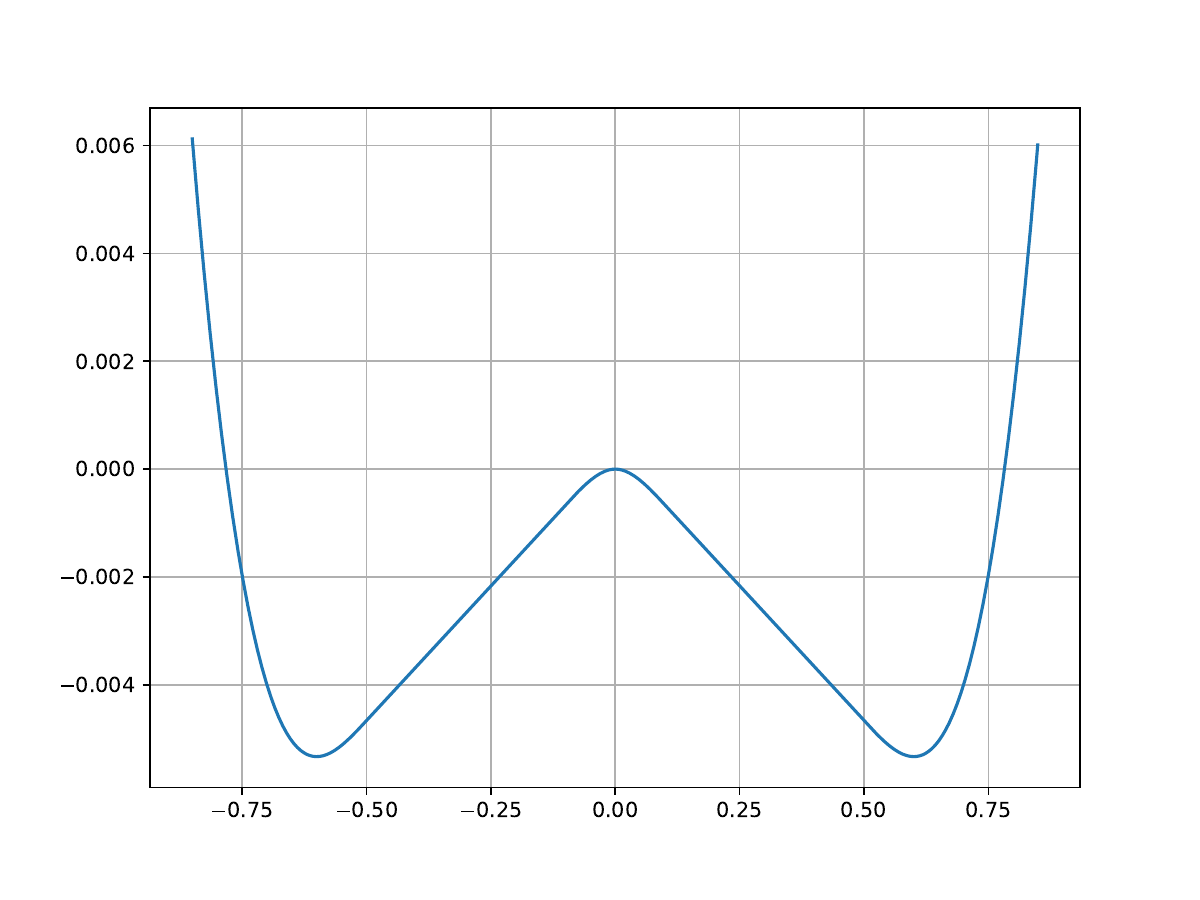}
    \caption{W-shape function designed by \citet{tripuraneni2018stochastic}}
    \label{fig:w_shape}
\end{figure}

\begin{figure}[!tb]
    \centering
    \subfigure[Initial point $(x_1,y_1)$]{
     \includegraphics[width=6cm]{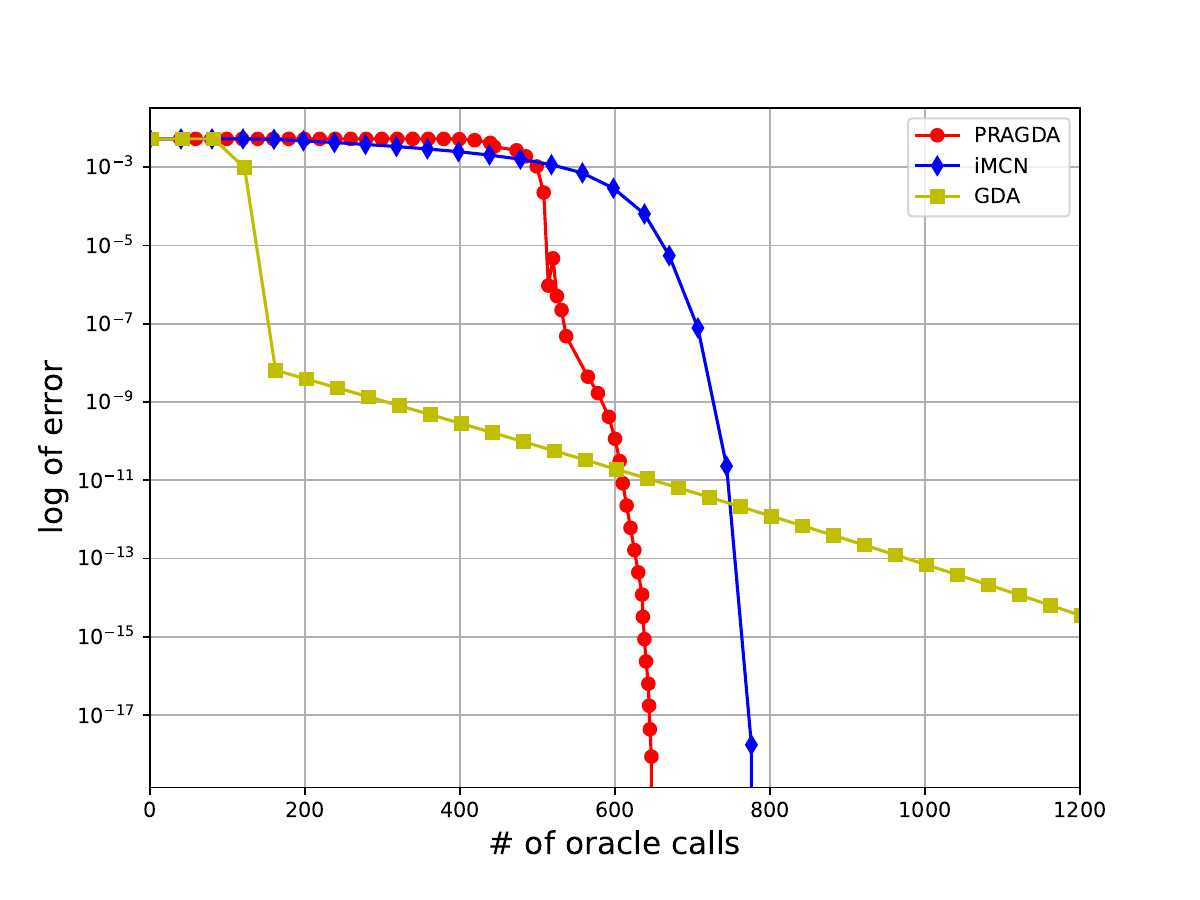}
    }
   \hspace{3mm}
   \subfigure[Initial point $(x_2,y_2)$]{
   \includegraphics[width=6cm]{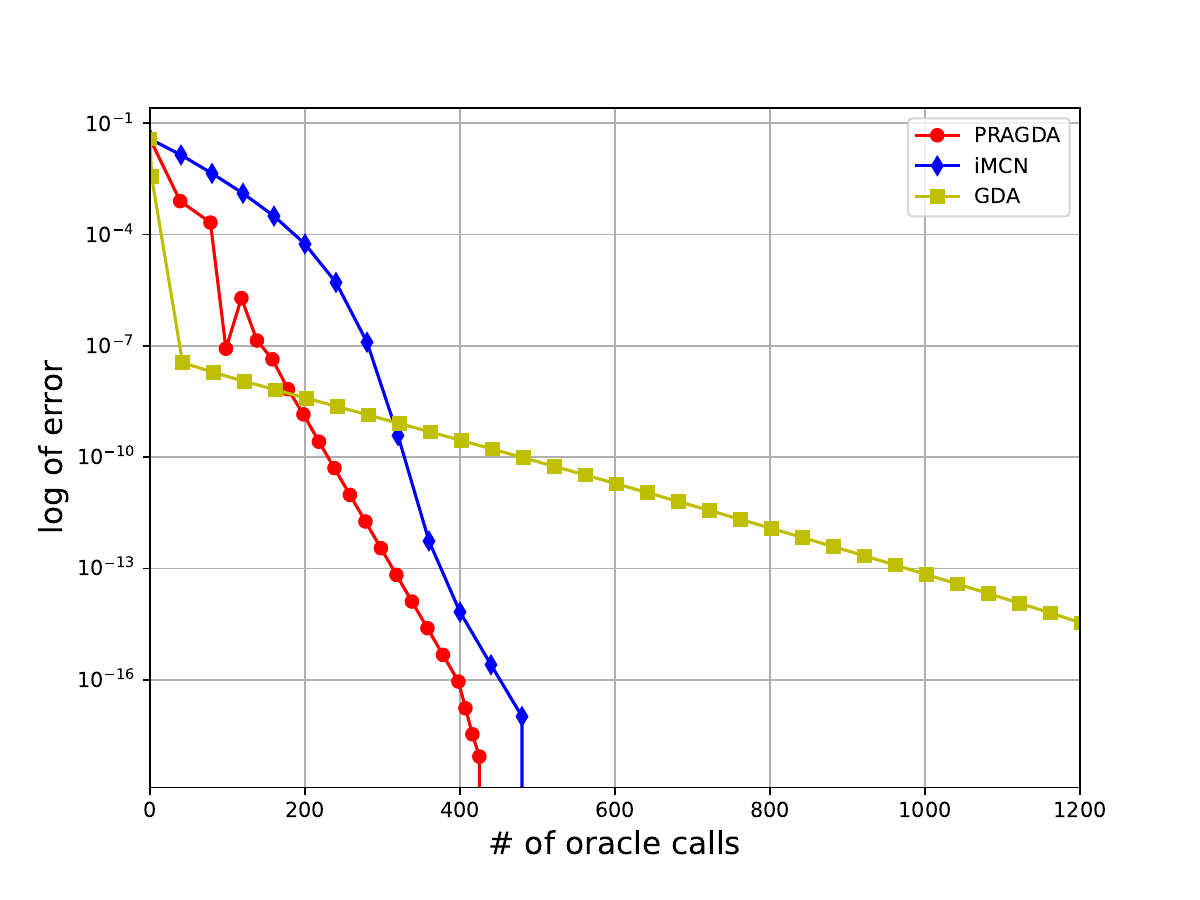}
   }
    \caption{Comparison of various minmax algorithms at different initial points}
    \label{fig:pragda}
\end{figure}

\pb\subsection{Synthetic Minimax Problem}
We construct the following nonconvex-strong-concave minimax problem:
$$
\min_{x\in \mathbb{R}^3}\max_{y\in\mathbb{R}^2}~f(x,y)
=
w(x_3) - 10y_1^2+x_1y_1 - 5y_2^2+x_2y_2
,
$$
where $x = [x_1,x_2,x_3]^T$ and $y = [y_1,y_2]^T$ and 
\begin{equation*}
w(x) = \left\{
\begin{array}{ll}
  \sqrt{\epsilon}(x+(L+1)\sqrt{\epsilon}\,)^2 - \frac{1}{3}(x+(L+1)\sqrt{\epsilon}\,)^3 -\frac{1}{3}(3L+1)\epsilon^{3/2},
&  x \le -L\sqrt{\epsilon};\\
\epsilon x + \frac{\epsilon^{3/2}}{3} ,&-L\sqrt{\epsilon} < x \le -\sqrt{\epsilon};\\
-\sqrt{\epsilon}x^2 - \frac{x^3}{3},& -\sqrt{\epsilon} < x\le 0;\\
-\sqrt{\epsilon}x^2 + \frac{x^3}{3},&   0<x\le \sqrt{\epsilon};\\
-\epsilon x + \frac{\epsilon^{3/2}}{3} ,&\sqrt{\epsilon} < x \le L\sqrt{\epsilon};\\
\sqrt{\epsilon}(x-(L+1)\sqrt{\epsilon}\,)^2 + \frac{1}{3}(x-(L+1)\sqrt{\epsilon}\,)^3 -\frac{1}{3}(3L+1)\epsilon^{3/2},
& L\sqrt{\epsilon} < x ;
\end{array}
\right.
\end{equation*}
is the W-shape-function~\citep{tripuraneni2018stochastic} and we set $\epsilon = 0.01, L = 5$ in our experiment.
We visualize the $w(\cdot)$ in Figure~\ref{fig:w_shape}.
It is easy to verify that $(x_0,y_0) = ([0,0,0]^\top,[0,0]^\top)$ is a saddle point of $f(x,y)$.
We construct our experiment with two different initial points:
\begin{align*}
(x_1,y_1) = \left([10^{-3},10^{-3},10^{-16}]^\top,[0,0]^\top\right) \qquad\text{and}\qquad (x_2,y_2) = \left([0,0,1]^\top,[0,0]^\top\right).    
\end{align*}
Note that $(x_1,y_1)$ is close to $(x_0,y_0)$ while $(x_2,y_2)$ is far from $(x_0,y_0)$.
We compare our \texttt{PRAGDA} with iMCN~\citep{luo2022finding} algorithm and classical GDA~\citep{lin2020gradient} algorithm.
The results are shown in Figure~\ref{fig:pragda}.
We use a grid search to choose the learning rate of AGD steps and GDA and outer-loop learning rate of \texttt{PRAGDA} from $\{c \times 10^i:c\in\{1,5\},i\in\{1,2,3\}\}$ and the momentum parameter from $\{c \times0.1:c\in\{1,2,3,4,5,6,7,8,9\} \}$.

\begin{figure}[!tb]
    \centering
    \subfigure[Corruption rate p = 0.2]{
     \includegraphics[width=6cm]{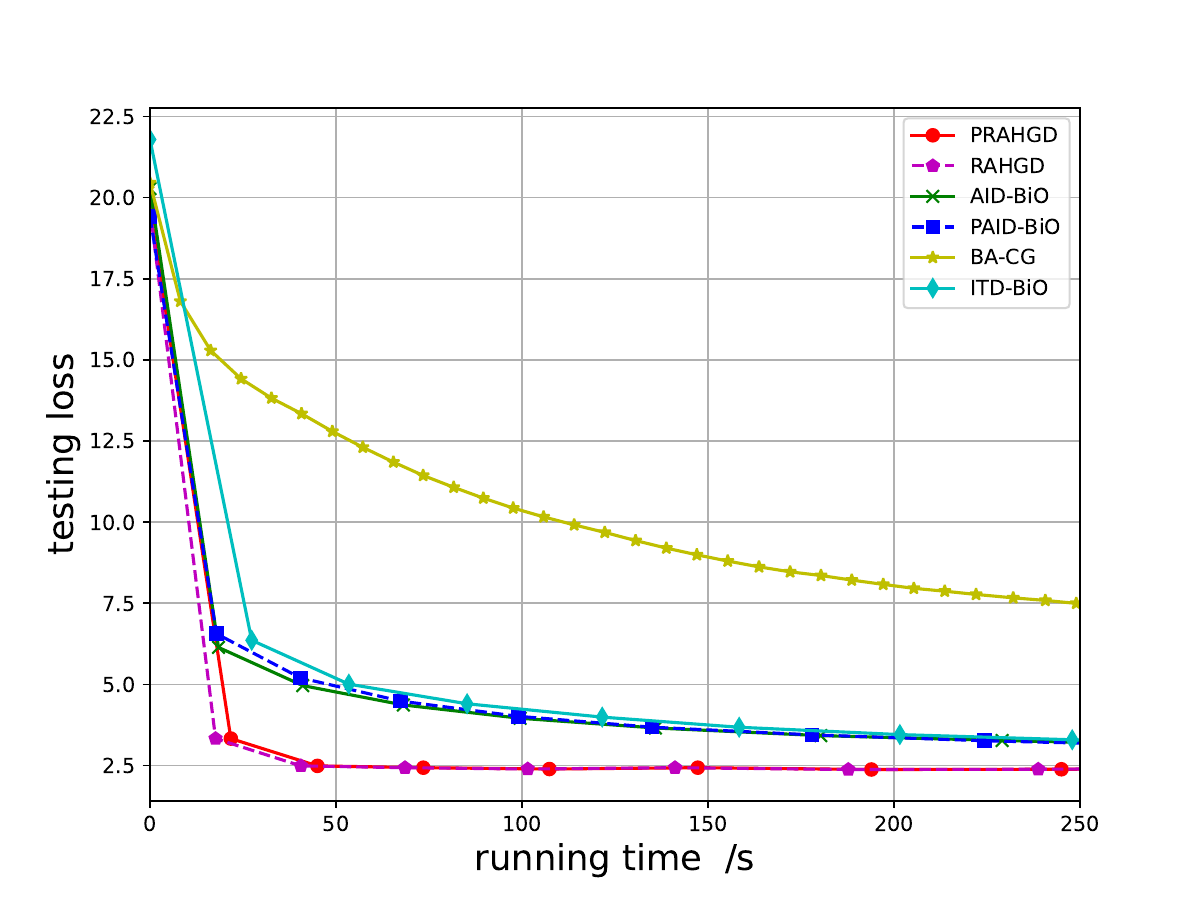}
    }
   \hspace{3mm}
   \subfigure[Corruption rate p = 0.4]{
   \includegraphics[width=6cm]{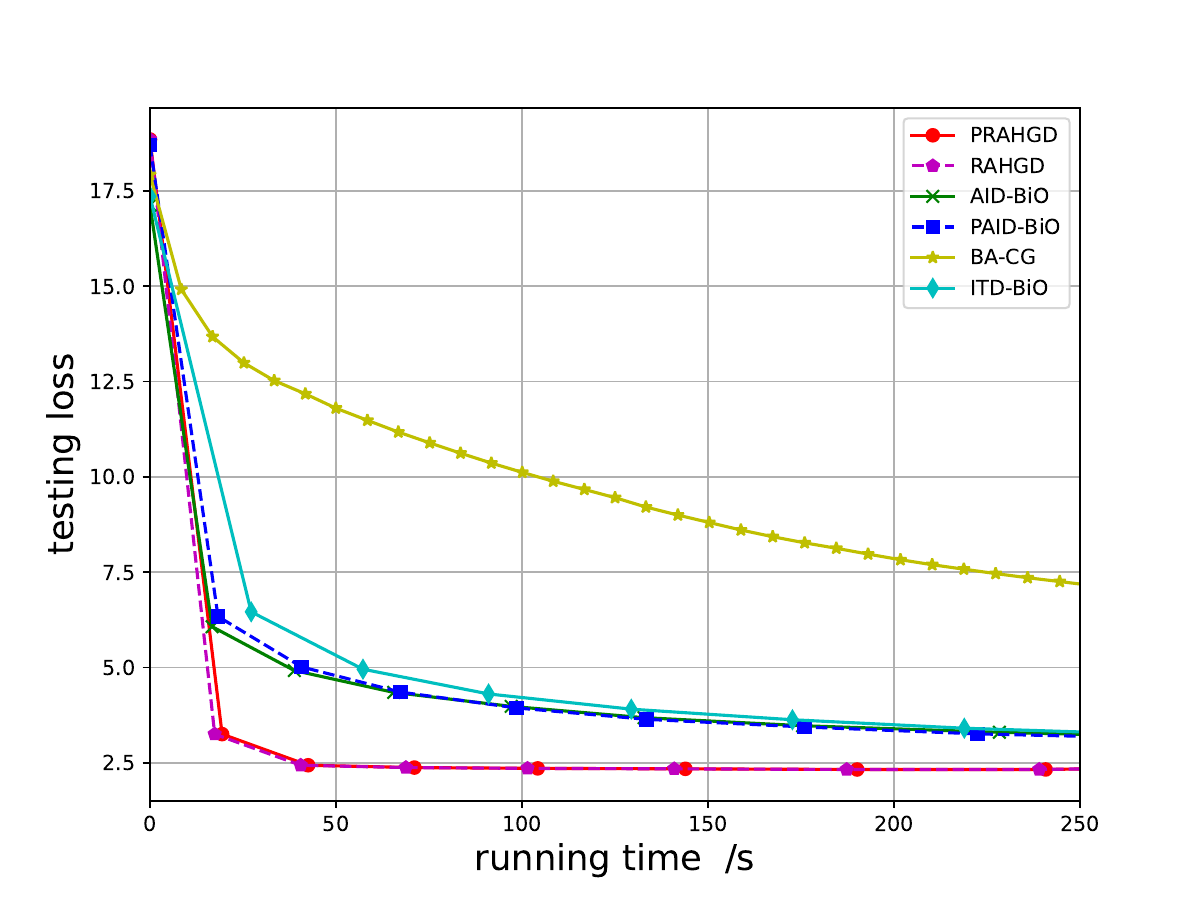}
   }
    \caption{Comparison of various bilevel algorithms for data hypercleaning at different corruption rate}
    \label{fig:data_cleaning}
\end{figure}

\pb\subsection{Data Hypercleaning}\label{sec:data_hypercleaning}
Data hypercleaning \citep{franceschi2017forward,shaban2019truncated} is a classic application of bilevel optimization. In practice, we may have a dataset with label noise and we could only offer some time or cost to cleanup a subset of the noise data.
To train a model in such a setting, we can treat the cleaned data as the validation set and the rest data as the training set. 
Then it can be transferred into a bilevel optimization:
\begin{equation}\label{equ:data hypercleaning}
\begin{aligned}
&
\min_{\lambda \in \mathbb{R}^{N_{\rm tr}}} f(W^*(\lambda),\lambda) \triangleq \frac{1}{|\fD_{\rm val}|}\sum_{(x_i,y_i)\in \fD_{\rm val}}-\log(y_i^\top W^*(\lambda)x_i)
\\&\textrm{s.t.}\quad 
W^*(\lambda) = \argmin_{W\in \mathbb{R}^{d_y \times d_x} }\  g(W,\lambda) \triangleq \frac{1}{|\fD_{\rm tr}|}\sum_{(x_i,y_i)\in \fD_{\rm tr}}-\sigma(\lambda_i)\log(y_i^\top Wx_i) + C_r ||W||^2
\ ,
\end{aligned}
\end{equation}
where $\fD_{\rm tr}=\{(x_i,y_i)\}$ is the training dataset, $\fD_{\rm val}=\{(x_i,y_i)\}$ is the validation dataset, $W$ is the weight of the  classifier, $\lambda_i \in \mathbb{R}$, $\sigma$ is the sigmoid function and $C_r$ is a regularization parameter. Following \citet{shaban2019truncated} and \citet{ji2021bilevel}, we choose $C_r=0.001$. 

We conducted the experiment on MNIST\citep{lecun1998gradient}. We have $x \in \mathbb{R}^{785}$, $y\in\mathbb{R}^{10}$ and $W\in\mathbb{R}^{10\times785}$ in \eqref{equ:data hypercleaning}.  
Training set contains 20,000 images, some of which have their labels randomly disrupted. We called the ratio of images with disrupted labels as corruption rate $p$. Validation set contains 5,000 images with correct labels. The testing set consists 10,000 images. The results are shown in Figure~\ref{fig:data_cleaning}. 

For BA algorithm proposed by \citet{ghadimi2018approximation}, we also use 
conjugate gradient descent method to compute the Hessian vector since they didn't specify it and we called it BA-CG in Figure~\ref{fig:data_cleaning}.  For all algorithms, We choose the inner-loop learning rate and outer-loop learning rate from $\{0.001,0.01,0.1,1,10\}$ and the iteration number of CG step from $\{3, 6, 12, 24\}$. For all algorithms except BA, we choose the iteration number of GD or AGD steps from $\{50, 100, 200, 500, 1000\}$  and for BA algorithm, as adopted by \citet{ghadimi2018approximation}, we choose the iteration number of GD steps from $\{\lceil c (k+1)^{1/4} \rceil:c\in\{0.5,1,2,4\} \}$ . We observe that both \texttt{RAHGD} and \texttt{PRAHGD} converge faster than other algorithms.

\begin{figure}[!tb]
    \centering
    \subfigure[testing accuracy vs.~running time]{
     \includegraphics[width=6cm]{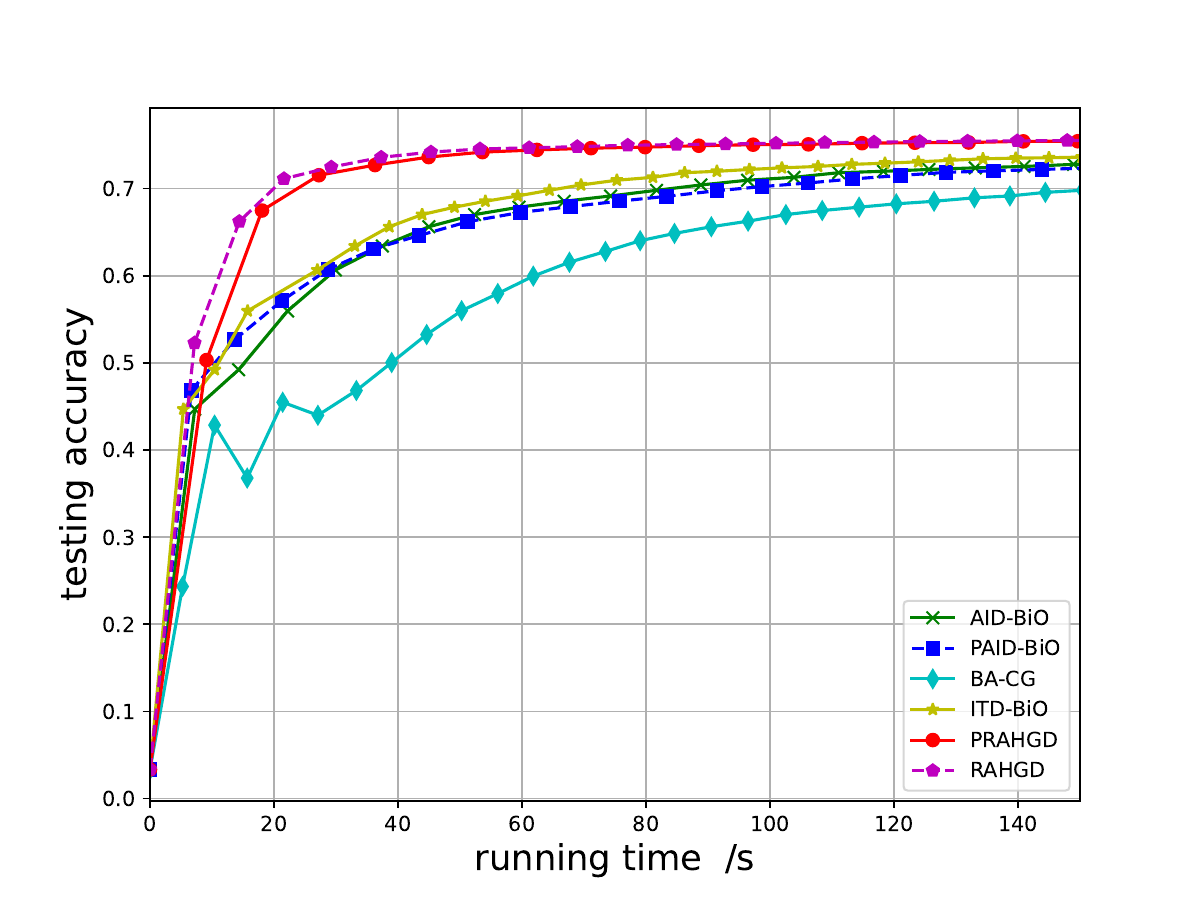}
    }
   \hspace{3mm}
   \subfigure[testing lose vs.~running time]{
   \includegraphics[width=6cm]{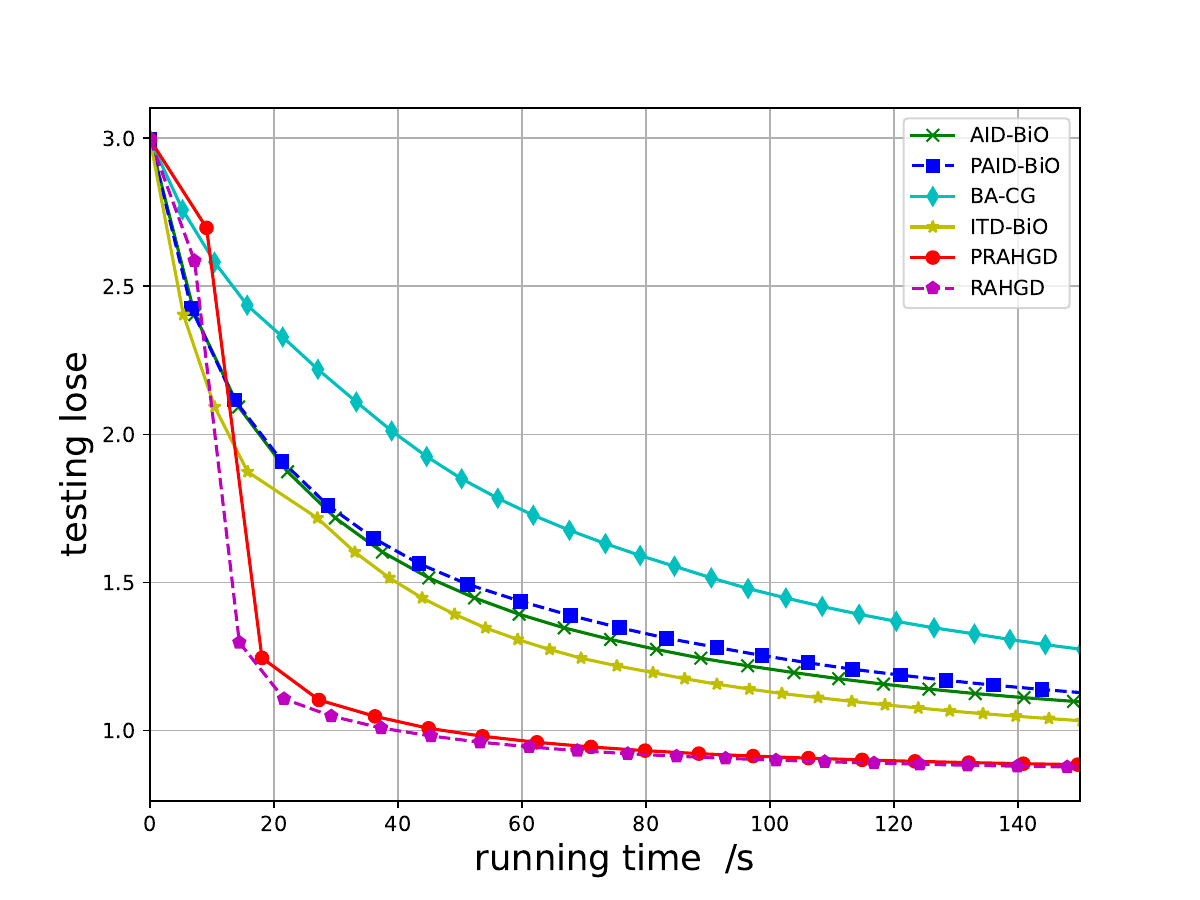}
   }

    \subfigure[testing accuracy vs.~\# of oracle calls]{
     \includegraphics[width=6cm]{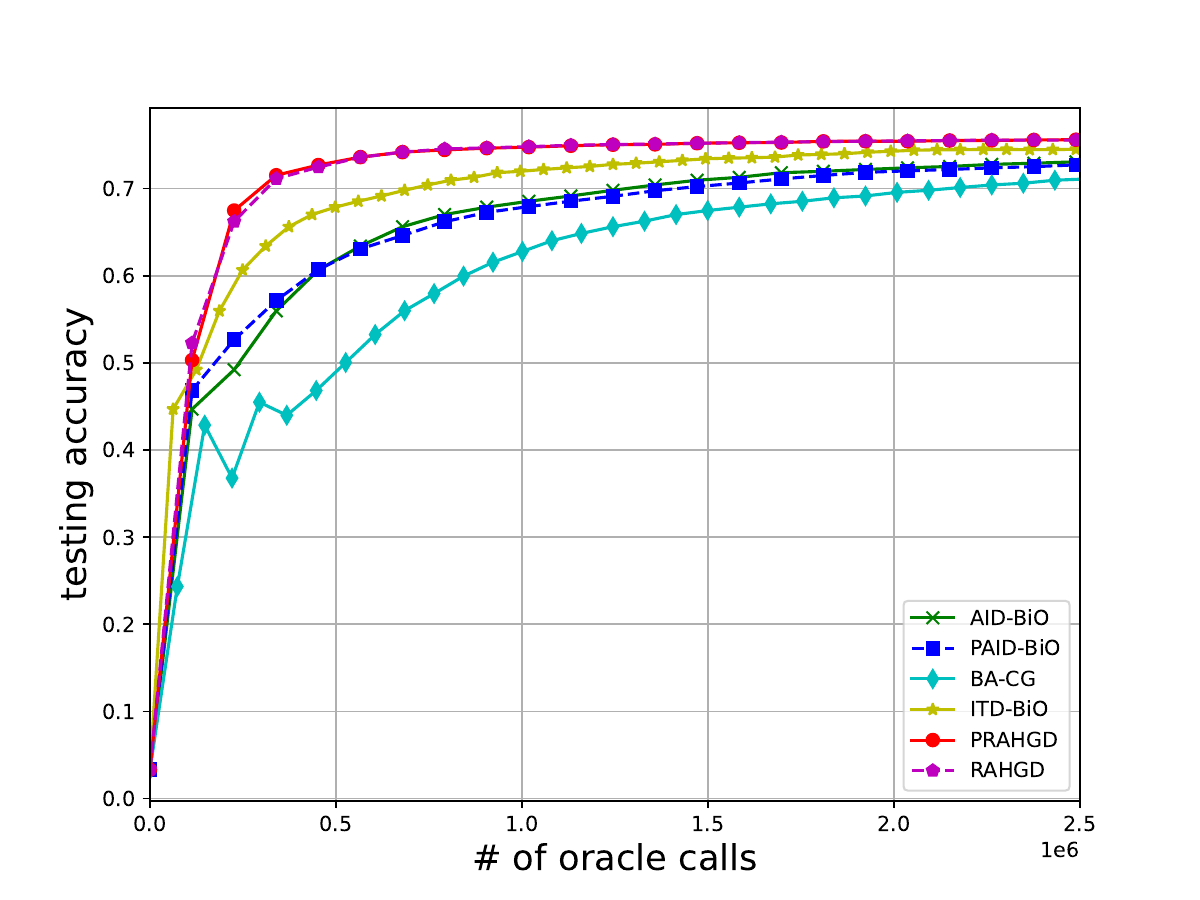}
    }
   \hspace{3mm}
   \subfigure[testing lose vs.~\# of oracle calls]{
   \includegraphics[width=6cm]{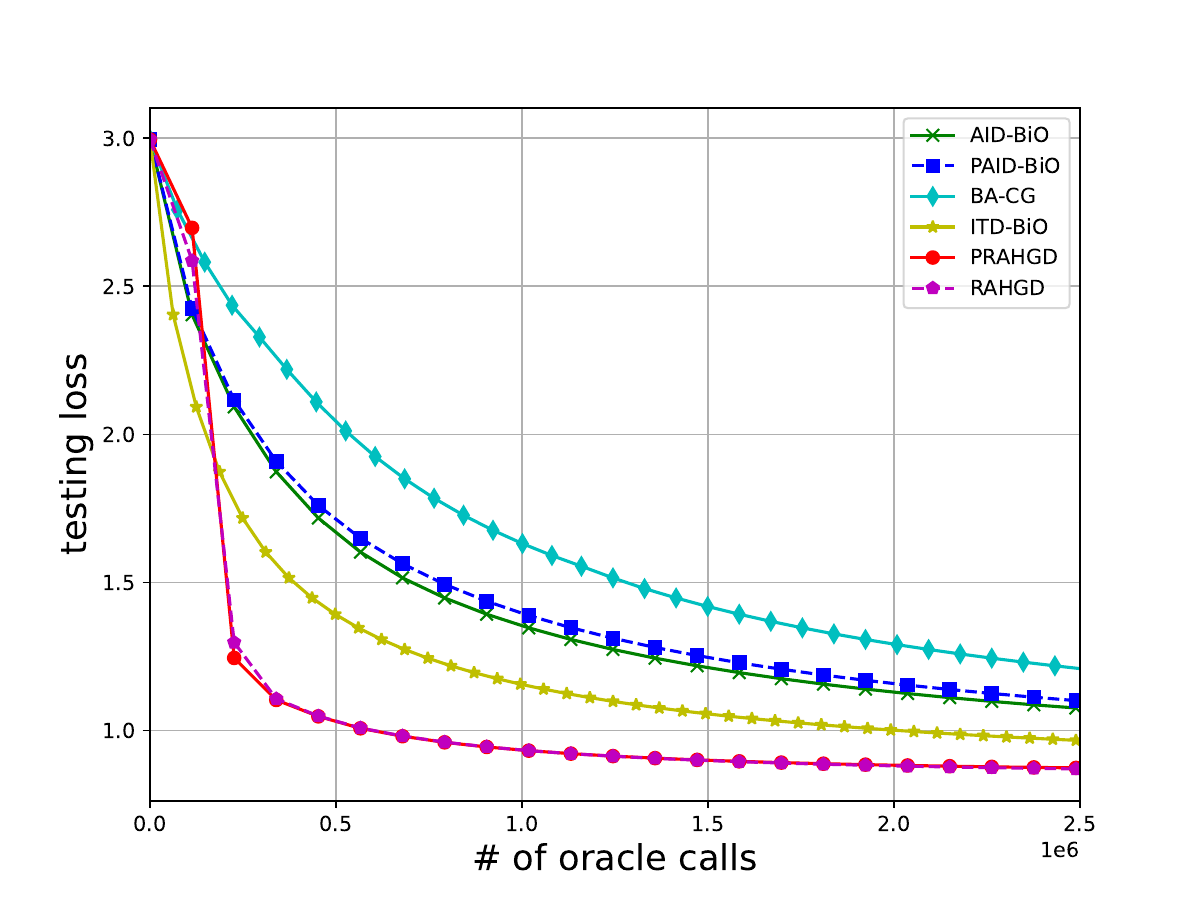}
   }
   \caption{Comparison of various bilevel algorithms on logistic regression on 20 Newsgroup dataset. Figure (a) and (b) show the results of testing accuracy and testing lose vs.running time respectively. Figure (c) and (d) show the results of testing accuracy and testing lose vs. \# of oracles calls respectively.}
    \label{fig:logistic_reg1}
\end{figure}

\pb\subsection{Hyperparemeter Optimization}
Hyperparameter optimization is a classical bilevel problem. The goal of hyperparameter optimization is to find the optimal hyperparameter to minimize the loss on the validation dataset.  We compare the performance of our algorithm \texttt{RAHGD} and \texttt{PRAHGD} with the baseline algorithms listed in Table~\ref{table:comparision_fir} and Table~\ref{table:comparision_sec} over a logistic regression problem on 20 News group dataset\cite{grazzi2020iteration}. 
This dataset consists of 18,846 news items divided into 20 topics, and features include 130,170 tf-idf sparse vectors. We divided the data into three parts: 5,657 for training, 5,657 for validation, and 7,532 for testing.
Then the objective function of this problem can be written in the following form.
\begin{align*}
&
\min_{\lambda\in\mathbb{R}^p}~%
\frac{1}{|\fD_{\rm val}|}\sum_{(x_i,y_i)\in\fD_{\rm val}}L(w^*(\lambda);x_i,y_i)
\\&s.t.~%
\ w^*(\lambda) = \argmin_{w\in\mathbb{R}^{c\times p}}\  \frac{1}{|\fD_{\rm tr}|}\sum_{(x_i,y_i)\in\fD_{\rm tr}} L(w;x_i,y_i) + \frac{1}{2cp}\sum_{j=1}^{c}\sum_{k=1}^{p}\exp(\lambda_k)w_{jk}^2, 
\end{align*}
where $\fD_{\rm tr}=\{(x_i,y_i)\}$ is the training dataset, $\fD_{\rm val}=\{(x_i,y_i)\}$ is the validation dataset, $L$ is the cross-entropy loss function, $c=20$ is the number of topics and $p=130,170$ is the dimension of features. Same as that in section~\ref{sec:data_hypercleaning}, we use the conjugate gradient descent method to approximate the Hessian vector.

For all algorithms listed in Figure~\ref{fig:logistic_reg1}, we choose the inner-loop learning rate and out-loop learning rate from $\{0.001,0.01,0.1,1,10,100,1000\}$, the iteration number of GD or AGD step from $\{5,10,30,50\}$, and the iteration number of CG step from $\{5,10,30,50 \}$. For BA-CG, we choose the iteration number of GD steps from $\{\lceil c (k+1)^{1/4} \rceil:c\in\{0.5,1,2,4\} \}$ as adopted by \citet{ghadimi2018approximation}. The results are shown in Figure~\ref{fig:logistic_reg1}. We observe that our \texttt{RAHGD} and \texttt{PRAHGD} converge faster than other algorithms.

\end{document}